\documentclass[12pt]{article}
\usepackage[includeheadfoot,top=2 cm, bottom=2 cm, left=2 cm, right=2 cm]{geometry}
\usepackage{graphicx}
\usepackage[numbers,sort&compress]{natbib}

\usepackage{amsfonts,amsthm}
\usepackage{upgreek}

\usepackage{hyperref}

\usepackage{algcompatible}
\usepackage{algorithm}
\usepackage{amssymb}
\usepackage{multirow}
\usepackage{array}
\usepackage{mathtools}
\usepackage{color}
\usepackage{xcolor}
\usepackage{float}
\usepackage{subcaption}
\hypersetup{
	colorlinks,
	linkcolor={blue!},
	citecolor={blue!},
	urlcolor={blue!}
}

\newcommand{\bl}{\mathbf{l}}
\newcommand{\ba}{\mathbf{a}}
\newcommand{\bb}{\mathbf{b}}
\newcommand{\bu}{\mathbf{u}}

\newcommand{\bv}{\mathbf{v}}
\newcommand{\by}{\mathbf{y}}
\newcommand{\bt}{\mathbf{t}}
\newcommand{\bz}{\mathbf{z}}
\newcommand{\bx}{\mathbf{x}}
\newcommand{\bw}{\mathbf{w}}

\newcommand{\bbm}{\mathbf{m}}

\newcommand{\bn}{\mathbf{n}}
\newcommand{\br}{\mathbf{r}}
\newcommand{\bnull}{\mathbf{0}}
\newcommand{\bTheta}{\mathbf{\Theta}}

\newcommand{\bPhi}{\mathbf{\Phi}}
\newcommand{\bR}{\mathbf{R}}
\newcommand{\bA}{\mathbf{A}}
\newcommand{\bE}{\mathbf{E}}
\newcommand{\bB}{\mathbf{B}}

\newcommand{\bT}{\mathbf{T}}
\newcommand{\bM}{\mathbf{M}}

\newcommand{\bH}{\mathbf{H}}
\newcommand{\bD}{\mathbf{D}}
\newcommand{\bG}{\mathbf{G}}
\newcommand{\bOmega}{\mathbf{\Omega}}
\newcommand{\bGamma}{\mathbf{\Gamma}}

\newcommand{\bQ}{\mathbf{Q}}
\newcommand{\bC}{\mathbf{C}}
\newcommand{\bP}{\mathbf{P}}
\newcommand{\bI}{\mathbf{I}}

\newcommand{\bU}{\mathbf{U}}

\newcommand{\bV}{\mathbf{V}}

\newcommand{\keywords}[1]{{\textit{Keywords---}} #1}

\newtheorem{lemma}{Lemma}[section]
\newtheorem{proposition}[lemma]{Proposition}
\newtheorem{corollary}[lemma]{Corollary}
\newtheorem{theorem}[lemma]{Theorem}
\newtheorem{remark}[lemma]{Remark}
\newtheorem{definition}[lemma]{Definition}

\begin{document}

	\title{Randomized linear algebra for model reduction. \\
		Part I: Galerkin methods and error estimation.
	}
	\author{ 
		Oleg Balabanov\footnotemark[1]~\footnotemark[2] ~and Anthony Nouy\footnotemark[1]~\footnotemark[3]
	}
	
	\renewcommand{\thefootnote}{\fnsymbol{footnote}}
	\footnotetext[1]{Centrale Nantes, LMJL, UMR CNRS 6629, France.}
	\footnotetext[2]{Polytechnic University of Catalonia, LaC\`an, Spain.}
	\footnotetext[3]{Corresponding author (anthony.nouy@ec-nantes.fr).}
	\renewcommand{\thefootnote}{\arabic{footnote}}
	\date{}	
	\maketitle

\begin{abstract}
We propose a probabilistic way for reducing the cost of classical projection-based model order reduction methods for parameter-dependent linear equations. 
A reduced order model is here approximated from its random sketch, which is a set of low-dimensional random projections of the reduced approximation space and the spaces of associated residuals.
This approach exploits the fact that the residuals associated with approximations in low-dimensional spaces are also contained in low-dimensional spaces. We provide conditions on the dimension of the random sketch for the resulting reduced order model to be quasi-optimal with high probability. Our approach can be used for reducing both complexity and memory requirements. The provided algorithms are well suited for any modern computational environment. Major operations, except solving linear systems of equations, are embarrassingly parallel. Our version of proper orthogonal decomposition  can be computed on multiple workstations with a communication cost independent of the dimension of the full order model. The reduced order model can even be constructed in a so-called streaming environment, i.e., under extreme memory constraints. In addition, we provide an efficient way for estimating the error of the reduced order model, which is not only more efficient than the classical approach but is also less sensitive to round-off errors. Finally, the methodology is validated on benchmark problems. 

\keywords{model reduction, reduced basis, proper orthogonal decomposition, random sketching, subspace embedding}
%\subclass{15B52 \and 35B30 \and 65F99 \and 65N15}
\end{abstract}

\section{Introduction}
Projection-based model order reduction (MOR) methods, including {the} reduced basis (RB) {method} or proper orthogonal decomposition (POD), are popular approaches for {approximating} large-scale parameter-dependent equations (see the recent surveys and monographs \cite{benner2015survey,quarteroni2015reduced,HesthavenRozzaStamm2015,morbook2017}). {They can be considered in the contexts of optimization, uncertainty quantification, inverse problems, real-time simulations, etc.}  An essential feature of MOR methods is offline/online splitting of the computations. 
The construction of the reduced order (or surrogate) model, which is usually the most computationally demanding task, is performed during the offline stage. This stage consists of (i) the generation of a reduced approximation space with a greedy algorithm for RB method or a principal component analysis of a set of samples of the solution for POD and (ii) the efficient representation of the reduced system of equations, usually obtained {through} (Petrov-)Galerkin projection, and of all the quantities needed for evaluating output quantities of interest and error estimators. {In the online stage, the reduced order model is evaluated for each value of the parameter and provides prediction of the output quantity of interest with a small computational cost, which is independent of the dimension of the initial system of equations.}

In this paper, we address the reduction of computational costs for both offline and online stages of projection-based model order reduction methods by adapting random sketching methods \cite{achlioptas2003database,sarlos2006improved} to the context of RB and POD. These methods were proven capable of significant complexity reduction for basic problems in numerical linear algebra such as computing {products} or factorizations of matrices~\cite{halko2011finding,woodruff2014sketching}. {We show how a reduced order model can be approximated from a small set, called a sketch, of efficiently computable random projections of the reduced basis vectors and the vectors involved in the affine expansion{\footnote{{A parameter-dependent quantity $\bv(\mu)$ with values in vector space $V$ over a field $\mathbb{K}$ is said to admit an affine representation (or be parameter-separable) if $\bv(\mu) = \sum^d_{i=1} \bv_i \lambda_i(\mu)$ with $\lambda_i(\mu) \in \mathbb{K}$ and $\bv_i \in V$. Note that for $V$ of finite dimension, $\bv(\mu)$ always admits an affine representation with a finite number of terms.}}}
of the residual, which is assumed to contain a small number of terms.} Standard algebraic operations are performed on the sketch, which avoids heavy operations on large-scale matrices and vectors. {Sufficient conditions on the dimension of the sketch for quasi-optimality of approximation of the reduced order model can be obtained by exploiting the fact that the residuals associated with reduced approximation spaces are contained in low-dimensional spaces.} 
Clearly, the randomization inevitably implies a probability of failure. This probability, however, is a user-specified parameter that can be chosen extremely small without affecting {considerably} the computational costs. 
{Even though this paper is concerned only with linear equations, similar considerations should also apply to a wide range of nonlinear problems.}

{Note that deterministic techniques have also been proposed for adapting POD methods to modern (e.g., multi-core or limited-memory) computational architectures~\cite{oxberry2017limited,himpe2016hierarchical,braconnier2011towards}.
{Compared to the aforementioned deterministic approaches, our randomized version of POD (see~Section\nobreakspace \ref {sk_pod}) has the advantage of not requiring the computation of the full reduced basis vectors, but only of their small random sketches.} In fact, maintaining and operating with large vectors can be completely avoided. This remarkable feature makes our algorithms particularly well suited for distributed computing and streaming context{s}.}

{Randomized linear algebra has} been employed for reducing the computational cost of MOR in~\cite{hochman2014reduced,alla2016randomized}, where the authors considered random sketching only as a tool for efficient evaluation of low-rank approximations of large matrices (using randomized versions of SVDs). They, however, did not adapt the MOR methodology itself and therefore did not fully exploit randomization techniques. In~\cite{buhr2017randomized} a probabilistic range finder based on random sketching has been used for combining {the} RB method with domain decomposition. Random sketching was also used for building parameter-dependent preconditioners for projection-based MOR in~\cite{zahm2016interpolation}.

The rest of the paper is organized as follows. Section\nobreakspace \ref {Contributions} presents the main contributions and discusses the benefits of the proposed methodology. In Section\nobreakspace \ref {MOR} we introduce the problem of interest and  present the ingredients of standard projection-based model order reduction methods.  In Section\nobreakspace \ref {RS}, we extend  the classical sketching technique in Euclidean spaces to a more general framework. In Section\nobreakspace \ref {l2embeddingsMOR}, we introduce the concept of a \emph{sketch of a model} and 
{propose} new and efficient randomized versions of Galerkin projection, residual based error estimation, and primal-dual correction. In  Section\nobreakspace \ref {efficient_RB}, we present and discuss {the} randomized greedy algorithm and POD for the efficient generation of reduced approximation spaces.  In~Section\nobreakspace \ref {Numerical}, the methodology is validated on two benchmarks. Finally, in~Section\nobreakspace \ref {Conclusions}, we provide conclusions and perspectives. 

{Proofs of propositions and theorems are provided in the Appendix.}

\subsection{Main contributions} \label{Contributions}

Our methodology can be used for the efficient  construction of a reduced order model. In classical projection-based methods, the  cost of evaluating samples (or snapshots) of the solution for a training set of {parameter} values can be much smaller than the cost of  other computations. This is the case when the samples are computed {using a sophisticated method for solving linear systems of equations requiring log-linear complexity,} or beyond the main routine, e.g., using a highly optimised commercial {solvers} or a server with limited budget, and possibly obtained using multiple workstations. 
This is also the case when, due to memory constraints, the computational time of algorithms for constructing the reduced order model are greatly affected by the number of passes taken over the data. In all these cases the cost of the offline stage is dominated by the post-processing of samples but not their computation. We here assume that the cost of solving high-dimensional systems is irreducible and focus on the reduction of other computational costs. The metric for efficiency depends on the computational environment and how data is presented to us. Our algorithms can be beneficial in basically all computational environments. \\

\subsubsection*{Complexity reduction}
Consider a parameter-dependent linear system of equations $\bA(\mu) \bu(\mu) = \bb(\mu)$ of dimension $n$ and assume that the parameter-dependent matrix $\bA(\mu)$ and vector $\bb(\mu)$ are parameter-separable with $m_A$ and $m_b$ terms, respectively (see Section\nobreakspace \ref {MOR} for more details). Let $r \ll n$ be the dimension of the reduced approximation space. Given a basis of this space, the classical construction of a reduced order model requires the evaluation of inner products between high-dimensional vectors. More precisely, it consists in multiplying each of the $r m_A+ m_b$ vectors in the affine expansion of the residual by $r$ vectors for constructing the reduced systems and by $r m_A + m_b$ other vectors for estimating the error. These two operations result in $\mathcal{O}(n r^2 m_A +  n r m_b)$ and $\mathcal{O}(n r^2 m^2_A + n m^2_b)$ flops respectively. It can be argued that the aforementioned complexities can dominate the {total complexity} of the offline stage {(see Section\nobreakspace \ref {sketch}).}
With the methodology presented in this work the complexities can be reduced to $\mathcal{O}(n r m_A \log{k} +  n m_b \log{k})$, where $r \leq k \ll n$. 

Let $m$ be the number of samples in the training set. The {computation} of the POD basis using a direct eigenvalue solver requires multiplication of two $n \times m$ matrices, i.e.,  $\mathcal{O}(n m \min(n, m))$ flops, while using a Krylov solver it requires multiplications of a $n \times m$ matrix by $\mathcal{O}(r)$ adaptively chosen vectors, i.e., $\mathcal{O}(n m r)$ flops. In the prior work \cite{alla2016randomized} on randomized algorithms for MOR, the authors proposed to use a randomized version of SVD introduced in~\cite{halko2011finding} for the  computation of the POD basis. More precisely, the SVD can be performed by applying Algorithms 4.5 and 5.1 in~\cite{halko2011finding} with complexities $\mathcal{O}(n m \log{k} + n k^2)$ and $\mathcal{O}(n m k)$, respectively. However, the authors in \cite{alla2016randomized} did not take any further advantage of random sketching methods, besides the SVD, and did not provide any theoretical analysis. In addition, they considered the Euclidean norm for the basis construction, which can be far from optimal. Here we reformulate the classical POD and obtain an algebraic form {(see Proposition\nobreakspace \ref {thm:approx_pod})} well suited for the application of efficient low-rank approximation algorithms, e.g., randomized or {incremental} SVDs~\cite{baker2012low}. We consider a general inner product associated with a self-adjoint positive  definite  matrix. More importantly, we provide a new version of POD (see Section\nobreakspace \ref {sk_pod}) which does not require  evaluation of high-dimensional basis vectors. In this way, the complexity of POD can be reduced to only $\mathcal{O}(n m \log{k}$).\\

\subsubsection*{Restricted memory and streaming environments}

Consider an environment where the memory consumption is the primary constraint. {The} classical offline stage involves evaluations of inner products of high-dimensional vectors. 
These operations require many passes over large data sets, e.g., a set of samples of the solution or the reduced basis, and can result in a computational burden. We show how to build the reduced order model with only one pass over the data. In extreme cases our algorithms may be employed in a streaming environment, where samples of the solution are provided as data-streams and storage of only a few large vectors is allowed. Moreover, with our methodology one can build a reduced order model without storing any high-dimensional vector.\\

\subsubsection*{Distributed computing}

The computations involved in our version of POD can be efficiently distributed among multiple workstations. Each sample of the solution can be evaluated and processed on a different machine with absolutely no communication. Thereafter, small sketches of the samples can be sent to the master workstation for building the reduced order model. The total amount of communication required by our algorithm is proportional to $k$ (the dimension of the sketch) and is independent of the dimension of the initial full order model.\\

\subsubsection*{Parallel computing}

Recently, parallelization was {considered} as a workaround to address large-scale computations~\cite{knezevic2011high}. The authors did not propose a new methodology but rather exploited the key opportunities for parallelization in a standard approach. We, on the other hand, propose a new methodology which can be better suited for parallelization than the classical one. The computations involved in our algorithms mainly consist in evaluating random matrix-vector products and solving high-dimensional systems of equations. The former operation is embarrassingly parallel (with a good choice of random matrices), while the latter one can be efficiently parallelized with state-of-the-art algorithms. \\

\subsubsection*{Online-efficient and robust error estimation} 
In addition, we provide a new way for estimating the error associated with a solution of the reduced order model, the error being defined as some norm of the residual. It does not require any assumption on the way to obtain the approximate solution and can be employed  separately from the rest  of the methodology. For example, it could be used for the efficient estimation of the error associated with a classical Galerkin projection. Our approach yields cost reduction for the offline stage but it is also online-efficient. Given the solution of the reduced order model, it requires only $\mathcal{O}(r m_A + m_b)$ flops for estimating the residual-based error while a classical procedure takes $\mathcal{O}(r^2 m^2_A +m^2_b)$ flops. Moreover, {compared to} the classical approach, our method is less sensitive to round-off errors.

\section{Projection-based model order reduction methods} \label{MOR}
In this section, we introduce the problem of interest and present the basic ingredients of classical MOR algorithms {in a form well suited for random sketching methods}. We consider a discrete setting, e.g, a problem arising after discretization of a parameter-dependent PDE {or integral equation}. We use notations that are standard in the context of variational methods for PDEs. However, for models simply described by algebraic equations, the notions of solution spaces, dual spaces, etc., can be disregarded.

Let $U := \mathbb{K}^{n}$ (with $\mathbb{K} = \mathbb{R}$ or  $\mathbb{C}$) denote the solution space equipped with inner product $\langle \cdot , \cdot  \rangle_U := \langle \bR_U \cdot,  \cdot \rangle$, where $\langle \cdot, \cdot \rangle$ is the canonical $\ell_2$-inner product on $\mathbb{K}^{n}$ and $\bR_U \in \mathbb{K}^{n \times n}$ is some self-adjoint (symmetric if $\mathbb{K} = \mathbb{R}$ and Hermitian if  $\mathbb{K} = \mathbb{C}$) positive definite matrix. The dual space of $U$ is identified with $U':= \mathbb{K}^{n}$, which is endowed with inner product $\langle \cdot , \cdot \rangle_{U'}:= \langle  \cdot, \bR_U^{-1}\cdot  \rangle $. For a matrix $\bM \in \mathbb{K}^{n\times n}$ we denote by $\bM^\mathrm{H}$ its adjoint (transpose if $\mathbb{K} = \mathbb{R}$ and Hermitian transpose if $\mathbb{K} = \mathbb{C}$).

%\begin{remark}
%The notions of $U$ and $U'$ (which both denote $\mathbb{K}^{n}$) are introduced to specify the origins of vectors and matrices involved in the discrete problem. In the framework of numerical methods for PDEs, if a vector represents an element from the solution space for the PDE, then it is said to belong to $U$. On the other hand, the vectors identifying elements from the dual space are considered to lie in $U'$. The inner product $\langle \cdot , \cdot  \rangle_U$ (or $\langle \cdot , \cdot  \rangle_{U'}$) between two vectors from $U$ (or $U'$) is equal to the inner product of the corresponding elements from the solution space for the PDE (or the dual space). Furthermore, the Euclidean inner product  between a vector from $U$ and a vector from $U'$ is equal to the canonical inner product between the quantities which these vectors represent.
%\end{remark}}
\begin{remark}
{The} matrix $\bR_U$ is seen as a map from $U $ to $U'$. In the framework of numerical methods for PDEs, the entries of $\bR_U$ can be obtained by evaluating inner products of corresponding basis functions. For example, if the PDE is defined on a space equipped with $H^1$ inner product, then $\bR_U$ is equal to the stiffness (discrete Laplacian) matrix. For algebraic parameter-dependent equations, $\bR_U$ can be taken as identity.
\end{remark}
Let $\mu$ denote parameters taking values in a set $\mathcal{P}$ {(which is typically a subset of  $\mathbb{K}^p$, but could also be a subset of function spaces, etc.)}.  Let parameter-dependent linear forms $\bb(\mu) \in U'$ and $\bl(\mu) \in U'$ represent the right-hand side and the extractor of a quantity of interest, respectively, and let $\bA (\mu): U \to U'$ represent the parameter-dependent operator. The problem of interest can be formulated as follows: for each given $\mu \in \mathcal{P}$ find the quantity of interest $s(\mu):= \langle \bl(\mu), \bu(\mu) \rangle$, where $\bu(\mu) \in U$ is such that
\begin{equation} \label{eq:initialproblem}
 \bA(\mu)\bu(\mu)= \bb(\mu).
\end{equation}

Further, we suppose that the solution manifold {$\{ \bu(\mu) : \mu \in \mathcal{P} \}$} can be well approximated by some low dimensional subspace of $U$. Let $U_r \subseteq U$ be such a subspace and $\bU_r \in \mathbb{K}^{n \times r}$ be a matrix whose column vectors form a basis for $U_r$.  The question of finding a good $U_r$ is addressed in Sections\nobreakspace \ref {Greedy} and\nobreakspace  \ref {POD}. In projection-based MOR methods, $\bu(\mu)$ is approximated by a projection $\bu_r(\mu) \in U_r$. 

\subsection{Galerkin projection}

Usually, a Galerkin projection $\bu_r(\mu)$ is obtained by imposing the following orthogonality condition to the residual~\cite{quarteroni2015reduced}:
\begin{equation} \label{eq:galproj} 
 \langle \br(\bu_r(\mu);\mu),\bw \rangle=0, ~\forall \bw \in U_r,  
\end{equation}
where $\br(\bx;\mu):= \bb(\mu)- \bA(\mu)\bx, ~\bx \in U$. 
This  condition can be expressed in a different form that will be particularly handy in further sections. For this we define the following semi-norm over $U'$:
\begin{equation} \label{eq:Urseminorm}
 \| \by \|_{U_r'} := \underset{ \bw \in U_r \backslash \{ \bnull \}} \max \frac{|\langle \by , \bw\rangle|}{\| \bw \|_U}, ~\by \in U'.
\end{equation}
Note that replacing $U_r$ by $U$ in definition~(\ref {eq:Urseminorm}) yields a norm consistent with the one induced by $\langle \cdot, \cdot \rangle_{U'}$. The relation~(\ref {eq:galproj}) can now be rewritten as
\begin{equation} \label{eq:galproj2}
 \| \br(\bu_r(\mu);\mu) \|_{U_r'}= 0. 
\end{equation}

Let us define the following parameter-dependent constants {characterizing quasi-optimality of Galerkin projection:}
\begin{subequations}
 \begin{align} 
  &\alpha_r(\mu):=  \underset{ \bx  \in U_r \backslash \{ \bnull \}} \min \frac{\| \bA(\mu) \bx  \|_{U_r'}}{\| \bx \|_U}, \label{eq:alphar} \\ 
  &\beta_r(\mu):=  \underset{ \bx  \in \left (\mathrm{span} \{ \bu(\mu) \}+ U_r \right ) \backslash \{ \bnull \}} \max \frac{ \| \bA(\mu) \bx \|_{U_r'}}{\| \bx \|_U}.\label{eq:betar}
 \end{align} 
\end{subequations}
It has to be mentioned that $\alpha_r(\mu)$ and $\beta_r(\mu)$ can be bounded by the coercivity constant $\theta(\mu)$ and the continuity constant (the maximal singular value) $\beta(\mu)$ of $\bA(\mu)$, respectively defined by
\begin{subequations} \label{eq:thetabeta}
	\begin{align} 
	\theta(\mu) &:=  \underset{ \bx  \in U \backslash \{ \bnull \}} \min \frac{\langle \bA(\mu) \bx, \bx \rangle}{\| \bx \|^2_U} \leq \alpha_r(\mu), \\
	\beta(\mu)  &:=  \underset{ \bx  \in U \backslash \{ \bnull \}} \max \frac{\| \bA(\mu) \bx  \|_{U'}}{\| \bx \|_U} \geq \beta_r(\mu). \label{eq:beta}
	\end{align} 
\end{subequations}
For some problems it is possible to provide lower and upper bounds for $\theta(\mu)$ and $\beta(\mu)$~\cite{haasdonk2017reduced}.

If $\alpha_r(\mu)$ is positive, then the reduced problem~(\ref {eq:galproj}) is well-posed. For given $V \subseteq U$, let $\bP_{V}:U \rightarrow V$ denote the orthogonal projection on $V$ with respect to $\| \cdot \|_{U}$, i.e.,
\begin{equation}
 \forall \bx \in U,~\bP_{V} \bx = \arg\min_{\bw \in V} \| \bx- \bw \|_{U}.
\end{equation}
We now provide {a quasi-optimality characterization} for the projection $\bu_r(\mu)$. 
\begin{proposition} [modified Cea's lemma] \label{thm:cea}
If $\alpha_r(\mu)>0$, then the solution $\bu_r(\mu)$ of~(\ref {eq:galproj}) is such that
\begin{equation} \label{eq:quasi-opt}
 \| \bu(\mu)- \bu_r(\mu) \|_{U} \leq (1+ \frac{\beta_r(\mu)}{\alpha_r(\mu)}) \| \bu(\mu)- \bP_{U_r} \bu(\mu) \|_{U}. 
\end{equation}
\begin{proof}
See appendix. 
\end{proof}
\end{proposition}
{Note that~Proposition\nobreakspace \ref{thm:cea} is a slightly modified version of the classical Cea's lemma with the continuity constant $\beta(\mu)$ replaced by $\beta_r(\mu)$.}

The coordinates of $\bu_r(\mu)$ in the basis $\bU_r$, i.e., $\ba_r(\mu) \in \mathbb{K}^{r}$ such that $\bu_r(\mu) = \bU_r \ba_r(\mu)$, can be found by solving the following system of equations 
\begin{equation} \label{eq:reduced_system}
 \bA_r(\mu) \ba_r(\mu) = \bb_r(\mu), 
\end{equation} 
where $\bA_r(\mu)= \bU_r^{\mathrm{H}}\bA(\mu)\bU_r \in \mathbb{K}^{r\times r}$ and $\bb_r(\mu)= \bU_r^{\mathrm{H}}\bb(\mu) \in \mathbb{K}^r$. The numerical stability of~(\ref {eq:reduced_system}) is usually obtained by orthogonalization of $\bU_r$.

\begin{proposition} \label{thm:clsstability}
If $\bU_r$ is orthogonal with respect to $\langle \cdot, \cdot \rangle_U$, then the condition number of $\bA_r(\mu)$ is bounded by $\frac{\beta_r(\mu)}{\alpha_r(\mu)}$.
\begin{proof}
See appendix. 
\end{proof}
\end{proposition}

\subsection{Error estimation}
When an approximation $\bu_r(\mu)\in U_r$ of the exact solution $\bu(\mu)$ has been evaluated, it is important to be able to certify how close they are. 
The error $\| \bu(\mu)-  \bu_r(\mu)\|_{U}$ can be bounded by the following error indicator 
\begin{equation} \label{eq:errorind} 
 \Delta(\bu_r(\mu); \mu):= \frac{\| \br(\bu_r(\mu); \mu) \|_{U'}}{\eta(\mu)},
\end{equation}
where $\eta(\mu)$ is such that 
\begin{equation} \label{eq:eta}
 \eta(\mu)\leq \underset{ \bx  \in U \backslash \{ \bnull \}} \min \frac{\| \bA(\mu) \bx  \|_{U'}}{\| \bx \|_U}.
\end{equation}
In its turn, the certification of the output quantity of interest $s_r(\mu):= \langle \bl(\mu), \bu_r(\mu) \rangle$ is provided by
\begin{equation} \label{eq:scert}
 |s(\mu)-  s_r(\mu)| \leq  \| \bl(\mu) \|_{U'} \| \bu(\mu)-  \bu_r(\mu) \|_{U} \leq \| \bl(\mu) \|_{U'} \Delta(\bu_r(\mu); \mu).
\end{equation}

\subsection{Primal-dual correction}
The accuracy of the output quantity obtained by the aforementioned methodology can be improved by goal-oriented correction~\cite{rozza2008reduced} explained below. A dual problem can be formulated as follows: for each $\mu \in \mathcal{P}$, find $\bu^\mathrm{du}(\mu) \in U$ such that
\begin{equation} \label{eq:dualproblem}
 \bA(\mu)^{\mathrm{H}} \bu^\mathrm{du}(\mu) = -\bl(\mu).
\end{equation}
The dual problem can be tackled in the same manner as the primal problem. For this we can use a Galerkin projection onto a certain $r^{\mathrm{du}}$-dimensional subspace $U^{\mathrm{du}}_r \subseteq U$.

Now suppose that besides approximation $\bu_r(\mu)$ of $\bu(\mu)$, we also have obtained an approximation of $\bu^\mathrm{du}(\mu)$ denoted by $\bu_r^\mathrm{du}(\mu) \in U^{\mathrm{du}}_r$. The quantity of interest can be estimated by
\begin{equation} \label{eq:correction}
 {s_r^{\mathrm{pd}}(\mu)}:= s_r(\mu)- \langle \bu_r^\mathrm{du}(\mu), \br(\bu_r(\mu); \mu) \rangle.
\end{equation}

\begin{proposition}\label{thm:error_correction}
 The estimation $s_r^{\mathrm{pd}}(\mu)$ of $s(\mu)$ is such that 
 \begin{equation} \label{eq:error_correction}
  |s(\mu)- {s_r^{\mathrm{pd}}(\mu)}| \leq  \| \br^{\mathrm{du}}(\bu_r^\mathrm{du}(\mu); \mu) \|_{U'} \Delta(\bu_r(\mu); \mu),
 \end{equation}	
 where $\br^{\mathrm{du}}(\bu_r^\mathrm{du}(\mu); \mu):= -\bl(\mu) -\bA(\mu)^{\mathrm{H}}\bu_r^\mathrm{du}(\mu)$.
 \begin{proof}
 	See appendix.
 \end{proof}
\end{proposition}
We observe that the error bound~(\ref {eq:error_correction}) of the quantity of interest is now quadratic in the residual norm in contrast to~(\ref {eq:scert}). 

\subsection{Reduced basis generation}

Until now we have assumed that the reduced subspaces $U_r$ and $U^{\mathrm{du}}_r$ were given. Let us briefly outline the standard procedure for the reduced basis generation with {the} greedy algorithm and POD. {The POD is here presented in a general algebraic form, which allows a non-intrusive use of any low-rank approximation algorithm.} Below we consider only the primal problem noting that similar algorithms can be used for the dual one. We also assume that a training set $\mathcal{P}_{\mathrm{train}} \subseteq \mathcal{P}$ with finite cardinality $m$ is provided.

\subsubsection{Greedy algorithm} \label{Greedy}
 The approximation subspace $U_{r}$ can be constructed recursively with a (weak) greedy algorithm. At iteration $i$, the basis of $U_i$ is enriched by snapshot $\bu(\mu^{i+1})$, i.e., $$U_{i+1}:= U_{i}+\mathrm{span}(\bu(\mu^{i+1})),$$ evaluated at a parameter value  $\mu^{i+1}$ that maximizes a certain error indicator $\widetilde{\Delta}(U_{i}; \mu)$ over the training set. Note that for efficient evaluation of $\arg \max_{\mu \in \mathcal{P}_{\mathrm{train}}}{\widetilde{\Delta}(U_i; \mu)}$ a provisional online solver associated with $U_{i}$ has to be provided.

The error indicator $\widetilde{\Delta}(U_{i}; \mu)$ for the greedy selection is typically chosen as an upper bound or estimator of $\| \bu(\mu)- \bP_{U_i}\bu(\mu) \|_{U}$. One can readily take $\widetilde{\Delta}(U_{i}; \mu) := {\Delta}(\bu_i(\mu); \mu) $, where $\bu_i(\mu)$ is the Galerkin projection defined by~(\ref {eq:galproj2}). The quasi-optimality of such $\widetilde{\Delta}(U_{i}, \mu)$ can then be characterized by using~Proposition\nobreakspace \ref{thm:cea} and definitions~(\ref{eq:beta}) and (\ref{eq:errorind}). 
%More precisely, if $\mu^{i+1}=\arg \max_{\mu \in \mathcal{P}_{\mathrm{train}}}{{\Delta}(\bu_i(\mu); \mu)}$, then 
% \begin{equation} \label{eq:greedyqualitystandard}
%  \| \bu(\mu^{i+1})- \bP_{U_i}\bu(\mu^{i+1}) \|_{U} \geq \frac{1}{\gamma_i}~ \underset{\mu \in \mathcal{P}_{\mathrm{train}}} \max \| \bu(\mu)- \bP_{U_i}\bu(\mu) \|_{U},
% \end{equation}
% where $\gamma_i= \underset{\mu \in \mathcal{P}_{\mathrm{train}}} \max {(1+\frac{\beta_i(\mu)}{\alpha_i(\mu)})} \frac{\beta(\mu)}{\eta(\mu)}$.

\subsubsection{Proper Orthogonal Decomposition} \label{POD}

In the context of POD we assume that the samples (snapshots) of $\bu(\mu)$, associated with the training set, are available. Let them be denoted as $\{ \bu (\mu^i) \}_{i=1}^{m}$, where $\mu^i \in \mathcal{P}_{\mathrm{train}}$, $1\leq i \leq m$. Further, let us define $\bU_m:= \left [ \bu (\mu^1), \bu (\mu^2), ..., \bu (\mu^m) \right ] \in \mathbb{K}^{n\times m}$ and $U_m:= \mathrm{range}(\bU_m)$. POD aims at finding a low dimensional subspace $U_r \subseteq U_m$ for the approximation of the set of vectors $\{ \bu (\mu^i) \}^{m}_{i=1}$.

For each $r\leq \mathrm{dim}(U_m)$ we define
\begin{equation} \label{eq:poddef}
POD_r(\bU_m, \| \cdot \|_U ):=\arg\min_{\substack{W_r  \subseteq U_m  \\ \mathrm{dim}(W_r) =r}}
\sum^{m}_{i=1} \| \bu (\mu^i) - \bP_{W_r} \bu (\mu^i) \|^2_{U}.
\end{equation}
The standard POD consists {in choosing $U_r$ as $POD_r(\bU_m, \| \cdot \|_U)$ and using the method of snapshots~\cite{sirovich1987turbulence}, or SVD of matrix $\bR^{1/2}_U \bU_m$, for computing the basis vectors. For large-scale problems, however, performing the method of snapshots or the SVD can become a computational burden. In such a case the standard eigenvalue decomposition and SVD have to be replaced by other low-rank approximations, e.g., incremental SVD, randomized SVD, hierarchical SVD, etc. For each of them it can be important to characterize quasi-optimality of the approximate POD basis. 
Below we provide a generalized algebraic version of POD well suited for a combination with low-rank approximation algorithms as well as state-of-the-art SVD. Note that obtaining (e.g., using a spectral decomposition) and operating with $\bR^{1/2}_U$ can be expensive and should be avoided for large-scale problems. The usage of this matrix for constructing the POD basis can be easily circumvented (see~Remark \nobreakdash \ref{rmk:Q_eval} ).	} 

\begin{proposition} \label{thm:approx_pod}
	Let $\bQ \in \mathbb{K}^{s\times n}$ be such that $\bQ^{\mathrm{H}}\bQ= \bR_U$. Let  $\bB^*_r \in \mathbb{K}^{s \times m}$ be a best rank-$r$ approximation of $\bQ \bU_m$ with respect to the Frobenius norm $\|\cdot \|_{F}$. Then for any rank-$r$ matrix $\bB_r \in \mathbb{K}^{s\times m}$, it holds
	\begin{equation}
	\frac{1}{m} \| \bQ \bU_m - {\bB^*_r}\|^2_{F} \leq \frac{1}{m} \sum^{m}_{i=1} \|\bu (\mu^i) - \bP_{{U_r}} \bu (\mu^i) \|^2_{U}  \leq  \frac{1}{m} \| \bQ \bU_m - {\bB_r}\|^2_{F},
	\end{equation}
	where ${U_r}:= \{ \bR_{U}^{-1}\bQ^{\mathrm{H}}\bb : \bb \in \mathrm{span}({\bB_r}) \}$.
\begin{proof}
See appendix.
 \end{proof} \end{proposition}

\begin{corollary}\label{thm:exact_pod}
	Let $\bQ \in \mathbb{K}^{s\times n}$ be such that $\bQ^{\mathrm{H}}\bQ= \bR_U$. Let $\bB^*_r \in \mathbb{K}^{s\times m}$ be a best rank-$r$ approximation of $\bQ\bU_m$ with respect to the Frobenius norm $\|\cdot \|_{F}$. Then 
	\begin{equation}
	POD_r(\bU_m, \| \cdot \|_U )= \{ \bR_{U}^{-1}\bQ^{\mathrm{H}}\bb : \bb \in \mathrm{range}(\bB^*_r) \}.
	\end{equation}
\end{corollary}
It follows that the approximation subspace $U_r$ for $\{ \bu(\mu^i) \}^{m}_{i=1}$ can be obtained by computing a low-rank approximation of $\bQ \bU_m$. According to  Proposition\nobreakspace \ref{thm:approx_pod}, for given $r$, quasi-optimality of {$U_r$} can be guaranteed by quasi-optimality of {$\bB_r$}.

\begin{remark} \label{rmk:Q_eval}
  {The} matrix $\bQ$ in Proposition\nobreakspace \ref {thm:approx_pod} and Corollary\nobreakspace \ref {thm:exact_pod} can be seen as a map from $U$ to $ \mathbb{K}^{s}$. Clearly, it can be computed with a Cholesky (or spectral) decomposition of $\bR_U$. For large-scale problems, however, it might be a burden to obtain, store or operate with such a matrix. We would like to underline that $\bQ$ does not have to be a square matrix. It can be easily obtained in the framework of numerical methods for PDEs (e.g., finite elements, finite volumes, etc.). Suppose that $\bR_U$ can be expressed as an assembly of smaller self-adjoint positive semi-definite matrices $\bR_U^{(i)}$ each corresponding to the contribution, for example, of a finite element or subdomain. In other words, 
  \begin{equation*}
  \bR_U = \sum_{i=1}^l\bE^{(i)} \bR_U^{(i)} [\bE^{(i)}]^{\mathrm{T}},
  \end{equation*}
  where $\bE^{(i)}$ is an extension operator mapping a local vector to the global one (usually a boolean matrix). Since $\bR_U^{(i)}$ are small matrices, their Cholesky (or spectral) decompositions are easy to compute. Let $\bQ^{(i)}$ denote the adjoint of the Cholesky factor of $\bR_U^{(i)}$. It can be easily verified that 
  \begin{equation*}
  \bQ :=\left[ \begin{array}{c}
  \bQ^{(1)} [\bE^{(1)}]^{\mathrm{T}} \\
  \bQ^{(2)} [\bE^{(2)}]^{\mathrm{T}}  \\ 
  ... \\ 
  \bQ^{(l)} [\bE^{(l)}]^{\mathrm{T}} 
  \end{array}\right] 
  \end{equation*}
  satisfies $ \bQ^{\mathrm{H}}\bQ= \bR_U$.
\end{remark}
The POD procedure using low-rank approximations is depicted in~Algorithm\nobreakspace \ref {alg:approx_pod}.
\begin{algorithm} \caption{Approximate Proper Orthogonal Decomposition} \label{alg:approx_pod}
 \begin{algorithmic}
  \STATE{\textbf{Given:} $\mathcal{P}_{\mathrm{train}}$, $\bA(\mu)$, $\bb(\mu)$, $\bR_U$}
  \STATE{\textbf{Output}: $\bU_r$ and $\Delta^{\mathrm{POD}}$}
  \STATE{1. Compute the snapshot matrix $\bU_m$.}
  \STATE{2. Determine $\bQ$ such that $\bQ^{\mathrm{H}}\bQ= \bR_U$.}
  \STATE{3. Compute {a} rank-$r$ approximation, $\bB_r $, of $\bQ\bU_m$.}
  \STATE{4. Compute an upper bound, $\Delta^{\mathrm{POD}}$, of $\frac{1}{m} \| \bQ\bU_m- \bB_r \|^2_{F}$. }
  \STATE{5. Find a matrix, $\bC_r$ whose column space is $\mathrm{span}(\bB_r)$.}
  \STATE{6. Evaluate $\bU_r:=\bR_{U}^{-1}\bQ^{\mathrm{H}}\bC_r$.}
 \end{algorithmic}
\end{algorithm}

\section{Random sketching} \label{RS}
In this section, we adapt the classical sketching theory in Euclidean spaces~\cite{woodruff2014sketching} to a slightly more general framework. The sketching technique is seen as a modification of inner product for a given subspace. The modified inner product is approximately equal to the original one but it is much easier to operate with. Thanks to such interpretation of the methodology, integration of the sketching technique to the context of projection-based MOR will become straightforward. 

\subsection{$\ell_2$-embeddings} \label{l2embeddings}

Let $X:= \mathbb{K}^{n}$ be endowed with inner product $\langle \cdot, \cdot \rangle_{X}:= \langle \bR_X \cdot, \cdot \rangle $ for some self-adjoint positive definite matrix $\bR_X \in \mathbb{K}^{n\times n}$, and let $Y$ be a subspace of $X$ of moderate dimension. The dual of $X$ is identified with $X':=\mathbb{K}^{n}$ and the dual of $Y$ is identified with $Y':=\{ \bR_{X} \by : \by \in Y \}$. $X'$ and $Y'$ are both equipped with inner product $\langle \cdot, \cdot \rangle_{X'}:=\langle \cdot, \bR_X^{-1} \cdot \rangle$. The inner products $\langle \cdot , \cdot \rangle_{X}$ and $\langle \cdot, \cdot \rangle_{X'}$ can be very expensive to evaluate.  The computational cost can be reduced drastically if we are interested {solely} in operating with vectors lying in subspaces $Y$ or $Y'$. For this we introduce the concept of $X \to \ell_2$ subspace embeddings.

Let $\bTheta \in \mathbb{K}^{k\times n}$ with $k\leq n$. Further, $\bTheta$ is seen as an embedding for subspaces of $X$. It maps vectors from the subspaces of $X$ to vectors from $\mathbb{K}^{k}$ equipped with {the canonical} {$\ell_2$}-inner product $\langle \cdot, \cdot \rangle$, so $\bTheta$ is referred to as an $X \to \ell_2$ subspace embedding. 
Let us now introduce the following semi-inner products on $X$:
\begin{equation} \label{eq:thetainnerdef}
 \langle \cdot, \cdot \rangle^{\bTheta}_{X}:= \langle \bTheta \cdot, \bTheta \cdot \rangle, \textup{ and }
 \langle \cdot, \cdot \rangle^{\bTheta}_{X'} := \langle \bTheta \bR_X^{-1} \cdot, \bTheta \bR_X^{-1} \cdot \rangle.
\end{equation}
Let $\| \cdot \|^{\bTheta}_{X}$ and $\| \cdot \|^{\bTheta}_{X'}$ denote the associated semi-norms. In general, $\bTheta$ is chosen so that $ \langle \cdot, \cdot \rangle^{\bTheta}_{X}$ approximates well $\langle \cdot, \cdot \rangle_{X}$ for all vectors in $Y$ or, in other words, $\bTheta$ is $X \to \ell_2$ $\varepsilon$-subspace embedding for $Y$, as defined below.
\begin{definition} \label{def:epsilon_embedding}
 If $\bTheta$ satisfies 
 \begin{equation} \label{eq:epsilon_embedding}
  \forall \bx, \by \in Y, \ \left | \langle \bx, \by \rangle_X - \langle \bx, \by \rangle^{\bTheta}_{X} \right |\leq \varepsilon \| \bx \|_X \| \by \|_X,
 \end{equation}
 for some $\varepsilon \in [0,1)$, then it is called a $X \to \ell_2$ $\varepsilon$-subspace embedding {(or simply, $\varepsilon$-embedding)} for $Y$. 
\end{definition}

\begin{corollary} \label{thm:dual_embedding}
 If $\bTheta$ is a $X \to \ell_2$ $\varepsilon$-subspace embedding for $Y$, then 
 \begin{equation*}
  \forall  \bx', \by' \in Y', \ \left | \langle \bx', \by' \rangle_{X'}- \langle \bx', \by' \rangle^{\bTheta}_{X'} \right | \leq \varepsilon \| \bx' \|_{X'} \| \by' \|_{X'}.
 \end{equation*}
 \end{corollary}
\begin{proposition} \label{thm:innerproduct}
 If $\bTheta$ is a $X \to \ell_2$ $\varepsilon$-subspace embedding for $Y$, then $\langle \cdot, \cdot \rangle^{\bTheta}_{X}$ and $\langle  \cdot , \cdot  \rangle^{\bTheta}_{X'}$ are inner products on $Y$ and $Y'$, respectively.  
\begin{proof}
 See appendix.   
\end{proof} 
\end{proposition}
Let $Z \subseteq Y$ be a subspace of $Y$. A semi-norm $ \| \cdot \|_{Z'}$ over $Y'$ can be defined by
\begin{equation} \label{eq:seminorm}
 \| \by' \|_{Z'}:= \underset{\bx \in Z \backslash \{ \bnull \}} \max \frac{|\langle \by', \bx \rangle|}{\| \bx \|_X}=\underset{\bx \in Z \backslash \{ \bnull \}} \max \frac{|\langle \bR_X^{-1} \by', \bx \rangle_{X}|}{\| \bx  \|_{X}},~ \by' \in Y'.
\end{equation}
We propose to approximate $\| \cdot \|_{Z'}$  by the semi norm $\| \cdot \|^{\bTheta}_{Z'}$ given by
\begin{equation} \label{eq:skseminorm}
 \| \by' \|^{\bTheta}_{Z'}:= \underset{\bx \in Z \backslash \{ \bnull \}} \max \frac{|\langle \bR_X^{-1}\by', \bx \rangle^{\bTheta}_{X}|}{\| \bx  \|^{\bTheta}_{X}},~ \by' \in Y'.
\end{equation} 
Observe that letting $Z= Y$ in~Equations\nobreakspace \textup {(\ref {eq:seminorm})} and\nobreakspace  \textup {(\ref {eq:skseminorm})} leads to norms on $Y'$ which are induced by $\langle \cdot, \cdot \rangle_{X'}$ and $\langle \cdot, \cdot \rangle^{\bTheta}_{X'}$. 
\begin{proposition} \label{thm:skseminorm_ineq}
 If $\bTheta$ is a $X \to \ell_2$ $\varepsilon$-subspace embedding for $Y$, then for all $\by' \in Y'$,
 \begin{equation} \label{eq:skseminorm_ineq}
  \frac{1}{\sqrt{1+\varepsilon}} (\| \by' \|_{Z'}- \varepsilon\| \by' \|_{X'})\leq \| \by' \|^{\bTheta}_{Z'} \leq \frac{1}{\sqrt{1-\varepsilon}}(\| \by' \|_{Z'}+ \varepsilon\| \by' \|_{X'}).
 \end{equation}
\begin{proof}
	See appendix.  
 \end{proof} 
\end{proposition}

\subsection{Data-oblivious embeddings} \label{obleddings}
Here we show how to build a $X \to \ell_2$ $\varepsilon$-subspace embedding  $\bTheta$ as a realization of a carefully chosen probability distribution over matrices. A reduction of the complexity of an algorithm can be obtained when $\bTheta$ is a structured matrix (e.g., sparse or hierarchical)~\cite{woodruff2014sketching} so that it can be efficiently multiplied by a vector. In such a case $\bTheta$ has to be operated {as a function outputting products with vectors}. For environments where the memory consumption or the cost of communication between cores is the primary constraint, unstructured $\bTheta$ can still provide drastic reductions and be more expedient~\cite{halko2011finding}. 
\begin{definition} \label{def:oblepsilon_embedding}
 $\bTheta$ is called a $( \varepsilon, \delta, d)$ oblivious $X \to \ell_2$ subspace embedding if for any $d$-dimensional subspace {$V$} of $X$ it holds 
 \begin{equation} \label{eq:oblepsilon_embedding}
  \mathbb{P} \left (\bTheta \text{ is a }X \to \ell_2 \text{ subspace embedding for } {V} \right) \geq 1-\delta.
 \end{equation}
\end{definition}
\begin{corollary} \label{thm:oblepsilon_prime}
 If $\bTheta$ is a $( \varepsilon, \delta, d)$ oblivious  $X \to \ell_2$ subspace embedding, then $\bTheta \bR_X^{-1}$ is a $( \varepsilon, \delta, d)$ oblivious $X' \to \ell_2$ subspace embedding.
 \end{corollary}
The advantage of oblivious embeddings is that they do not require any a priori knowledge of the embedded subspace. In this work we shall consider three well-known oblivious $\ell_2 \to \ell_2$ subspace embeddings: the rescaled Gaussian distribution, the rescaled Rademacher distribution, and the partial Subsampled Randomized Hadamard Transform (P-SRHT). The rescaled Gaussian distribution is such that the entries of $\bTheta$ are independent normal random variables with mean $0$ and variance $k^{-1}$. For the rescaled Rademacher distribution, the entries of $\bTheta$ are independent random variables satisfying $\mathbb{P} \left ( [\bTheta]_{i,j}= \pm k^{-1/2} \right )=1/2$. Next we recall a standard result that states that the rescaled Gaussian and Rademacher distributions with sufficiently large $k$ are $(\varepsilon, \delta, d)$ oblivious $\ell_2 \to \ell_2$ subspace embeddings. This can be found in~\cite{sarlos2006improved,woodruff2014sketching}. The authors, however, provided the bounds for $k$ in $\mathcal{O}$ (asymptotic) notation with no concern about the constants. These bounds can be impractical for certification (both a priori and a posteriori) of the solution. Below we provide explicit bounds for $k$. 

\begin{proposition} \label{thm:Rademacher}
 Let $\varepsilon$ and $\delta$ be such that $0<\varepsilon<0.572$ and $0 <\delta <1$. The rescaled Gaussian and the rescaled Rademacher distributions over $\mathbb{R}^{k \times n}$ with $k\geq 7.87 \varepsilon^{-2}({6.9} d + {\log ({1}/\delta)})$ for $\mathbb{K} = \mathbb{R}$ and $k\geq 7.87 \varepsilon^{-2}({13.8} d + {\log ({1}/\delta)})$ for $\mathbb{K} = \mathbb{C}$ are $( \varepsilon, \delta, d)$ oblivious $\ell_2 \to \ell_2$ subspace embeddings.
\begin{proof}
See appendix.
 \end{proof} \end{proposition}

\begin{remark} \label{rmk:complexGauss}
	For $\mathbb{K} = \mathbb{C}$, an embedding with a better theoretical bound for $k$ than the one in~Proposition\nobreakspace \ref {thm:Rademacher} can be obtained by taking $\bTheta:= \frac{1}{\sqrt{2}}(\bTheta_{\mathrm{Re}} + j \bTheta_{\mathrm{Im}})$, where $j = \sqrt{-1}$ and $\bTheta_{\mathrm{Re}}, \bTheta_{\mathrm{Im}} \in \mathbb{R}^{k \times n}$ are rescaled Gaussian matrices. It can be shown that such $\bTheta$ is an $( \varepsilon, \delta, d)$ oblivious $\ell_2 \to \ell_2$ subspace embedding for $k\geq 3.94 \varepsilon^{-2}(13.8 d + \log (1/\delta))$. A detailed proof of this fact is provided in the supplementary material.  In this work, however, we shall consider only real-valued embeddings.
\end{remark}

For {the} P-SRHT distribution, $\bTheta$ is taken to be the first $n$ columns of the matrix $k^{-1/2} (\bR \bH_s \bD) \in \mathbb{R}^{k\times s}$, where $s$ is the power of 2 such that $n\leq s <2n$, $\bR \in \mathbb{R}^{k\times s}$ are the first $k$ rows of a random permutation of {rows} of the identity matrix, $\bH_s \in  \mathbb{R}^{s\times s}$ is a Walsh-Hadamard matrix\footnote{{The Walsh-Hadamard matrix $\bH_s$ of dimension $s$, with $s$ being a power of $2$, is a structured matrix defined recursively by $\bH_s = \bH_{s/2} \otimes \bH_2$, with $\bH_2 := \begin{bmatrix}	1& 1\\ 	1& -1	\end{bmatrix}$. A product of $\bH_s$ with a vector can be computed with $s \log_2{(s)}$ flops by using the fast Walsh-Hadamard transform.}}, and  $\bD \in  \mathbb{R}^{s\times s}$ is a random diagonal matrix with random entries such that $\mathbb{P} \left ( [\bD]_{i,i} =\pm 1 \right )=1/2$.
\begin{proposition} \label{thm:P-SRHT}
 Let $\varepsilon$ and $\delta$ be such that $0<\varepsilon<{1}$ and $0<\delta <1$. The P-SRHT distribution over $\mathbb{R}^{k\times n}$ with $k\geq {2( \varepsilon^{2} - \varepsilon^3/3)^{-1}} \left [\sqrt{d}+ \sqrt{8 \log(6 n/\delta)} \right ]^2 \log (3 d/\delta)$ is a $( \varepsilon, \delta, d)$ oblivious $\ell_2 \to \ell_2$ subspace embedding.
\begin{proof}
See appendix.
 \end{proof} \end{proposition}

\begin{remark} \label{rmk:P-SRHT-Gaussian}
 A product of P-SRHT and Gaussian (or Rademacher) matrices can lead to oblivious $\ell_2 \to \ell_2$ subspace embeddings that have better theoretical bounds for $k$ than P-SRHT but still {have low complexity of multiplication by a vector.}
\end{remark}

 We observe that the lower bounds in Propositions\nobreakspace \ref {thm:Rademacher} and\nobreakspace  \ref {thm:P-SRHT} are independent or only weakly (logarithmically) dependent on the dimension $n$ and the probability of failure $\delta$. In other words, $\bTheta$ with a moderate $k$ can be  guaranteed to satisfy~(\ref {eq:oblepsilon_embedding}) even for extremely large $n$ and small $\delta$. {Note that the theoretical bounds for $k$ shall be useful only for problems with rather high initial dimension, say with $n/r>10^4$. Furthermore, in our experiments we revealed that the presented theoretical bounds are pessimistic. {Another way for selecting the size for {the} random sketching matrix $\bTheta$ such that it is an $\varepsilon$-embedding for a given subspace $V$ is the adaptive procedure proposed in~\cite{balabanov2018}.}}
 
 The rescaled Rademacher distribution and P-SRHT provide database-friendly matrices, which are easy to operate with. The rescaled Rademacher distribution is attractive from the data structure point of view and it can be efficiently implemented using standard SQL primitives~\cite{achlioptas2003database}. {The P-SRHT has a hierarchical structure allowing multiplications by vectors with only $s  \log_2{(s)}$ flops, where $s$ is a power of $2$ and $n\leq s < 2n$, using the fast Walsh-Hadamard transform or even $2s \log_2(k + 1)$ flops using a more sophisticated procedure proposed in~\cite{ailon2009fast}}. In the algorithms P-SRHT distribution shall be preferred. However for multi-core computing, where the hierarchical structure of P-SRHT cannot be fully exploited, Gaussian or Rademacher matrices can be more expedient. Finally, we would like to point out that a random sequence needed for constructing a realization of Gaussian, Rademacher or P-SRHT distribution can be generated using a seeded random number generator. In this way, an embedding can be efficiently maintained with negligible communication (for parallel and distributed computing) and storage costs. 

The following proposition can be used for constructing oblivious $X \to \ell_2$ subspace embeddings for general inner product  $\langle \bR_X \cdot, \cdot \rangle$ from classical $\ell_2 \to \ell_2$ subspace embeddings.
\begin{proposition} \label{thm:buildepsilon_embedding}
 Let $\bQ \in \mathbb{K}^{s\times n}$ be any matrix such that $\bQ^{\mathrm{H}}\bQ= \bR_X$. If $\bOmega \in \mathbb{K}^{k\times s}$ is a $(\varepsilon, \delta, d)$ oblivious $\ell_2 \to \ell_2$ subspace embedding, then $\bTheta=\bOmega\bQ$ is a $(\varepsilon, \delta, d)$ oblivious $X \to \ell_2$ subspace embedding.
\begin{proof}
 See appendix.
 \end{proof} \end{proposition}
Note that {the} matrix $\bQ$ in Proposition\nobreakspace \ref {thm:buildepsilon_embedding} can be efficiently obtained block-wise (see Remark\nobreakspace \ref {rmk:Q_eval}). In addition, there is no need to evaluate $\bTheta=\bOmega\bQ$ explicitly.

\section{$\ell_2$-embeddings for projection-based MOR} \label{l2embeddingsMOR}
In this section we integrate the sketching technique in the context of {model order reduction methods from Section\nobreakspace \ref {MOR}. Let us define the following subspace of $U$:
\begin{equation}  \label{eq:Y_r}
 Y_r(\mu) :=  U_r+ \mathrm{span} \{ \bR_U^{-1} \br(\bx; \mu) : \bx \in U_r \},
\end{equation}
where $\br(\bx; \mu) = \bb(\mu)- \bA(\mu)\bx$, and identify its dual space with $Y_r(\mu)' :=\mathrm{span} \{ \bR_U \bx : \bx \in Y_r(\mu) \}$.}
{Furthermore, let $\bTheta \in \mathbb{K}^{k\times n}$ be a certain sketching matrix seen as an $U \to \ell_2$ subspace embedding.}
\subsection {Galerkin projection} \label{SKgalproj}
{We propose to use random sketching for estimating the Galerkin projection. For any $\bx \in U_r$ the residual $ \br(\bx; \mu)$ belongs to $Y_r(\mu)'$.} Consequently, taking into account Proposition\nobreakspace \ref {thm:skseminorm_ineq}, if $\bTheta$ is a  $U \to \ell_2$ $\varepsilon$-subspace embedding for $Y_r(\mu)$, then for all $\bx \in U_r$ the semi-norm $\| \br(\bx; \mu) \|_{U_r'}$ in~(\ref {eq:galproj2}) can be well approximated by $\| \br(\bx; \mu) \|^{\bTheta}_{U_r'}$. This leads to the sketched version of the Galerkin orthogonality condition:
\begin{equation} \label{eq:SKgalproj}
 \| \br(\bu_r(\mu);\mu) \|^{\bTheta}_{U_r'}=0.  
\end{equation}
The quality of projection $\bu_r(\mu)$ satisfying~(\ref {eq:SKgalproj}) can be characterized by the following coefficients:
\begin{subequations} \label{eq:skalpharbetar}
\begin{align} 
 &\alpha^{\bTheta}_r(\mu):= \underset{ \bx \in U_r \backslash \{ \bnull \}} \min \frac{\| \bA(\mu)\bx \|^{\bTheta}_{U_r'}}{\| \bx \|_U}, \label{eq:skalphar} \\ 
 &\beta^{\bTheta}_r(\mu):= \underset{ \bx \in \left ( \mathrm{span} \{ \bu(\mu) \}+ U_r \right ) \backslash \{ \bnull \}} \max \frac{ \| \bA(\mu) \bx \|^{\bTheta}_{U_r'}}{\| \bx \|_U}.
\end{align} 
\end{subequations}

\begin{proposition}[Cea's lemma for sketched Galerkin projection] \label{thm:SKquasi-opt}
 Let $\bu_r(\mu)$ satisfy~(\ref {eq:SKgalproj}). If $\alpha^{\bTheta}_r(\mu)>0$, then the following relation holds 
 \begin{equation} \label{eq:SKquasi-opt}
  \| \bu(\mu)- \bu_r(\mu) \|_{U} \leq (1+\frac{\beta^{\bTheta}_r(\mu)}{\alpha^{\bTheta}_r(\mu)}) \| \bu(\mu)- \bP_{U_r}\bu(\mu) \|_{U}.
 \end{equation}
\begin{proof}
See appendix.
 \end{proof} \end{proposition}

\begin{proposition} \label{thm:skcea}
 Let 
 \begin{equation*}
  a_r(\mu):= \underset{ \bw \in U_r \backslash \{ \bnull \}} \max \frac{ \| \bA(\mu)\bw \|_{U'} }{\| \bA(\mu)\bw \|_{U_r'}}.
 \end{equation*}
 If $\bTheta$ is a $U \to \ell_2$  $\varepsilon$-embedding for $Y_r(\mu)$, then 
 \begin{subequations}
 \begin{align} \label{eq:skalphabetabounds}
  &\alpha^{\bTheta}_r(\mu)\geq \frac{1}{\sqrt{1+\varepsilon}}(1-\varepsilon a_r(\mu)) \alpha_r(\mu), \\   
  &\beta^{\bTheta}_r(\mu)\leq \frac{1}{\sqrt{1-\varepsilon}}(\beta_r(\mu)+ \varepsilon \beta(\mu)).
 \end{align} 
 \end{subequations}
\begin{proof}
See appendix.
 \end{proof} \end{proposition}

There are two ways to select a random distribution for $\bTheta$ such that it is guaranteed to be a $U \to \ell_2$  $\varepsilon$-embedding for $Y_r(\mu)$ for all $\mu \in \mathcal{P}$, simultaneously, with probability at least $1-\delta$. 
A first way applies when $\mathcal{P}$ is of finite cardinality. We can choose $\bTheta$ such that it is a $( \varepsilon, \delta {\# \mathcal{P}}^{-1}, d)$ oblivious $U \to \ell_2$ subspace embedding, where $d: = \max_{\mu  \in \mathcal{P} } {\textup{ dim} (Y_r(\mu))}$ and apply a union bound for the probability of success. Since $d \leq 2r +1$,   $\bTheta$ can be selected of moderate size.
When  $\mathcal{P}$ is infinite, we make a standard assumption that $\bA(\mu)$ and $\bb(\mu)$ admit affine representations. 
It then follows directly from the definition of $Y_r(\mu)$  that $ \bigcup_{\mu  \in \mathcal{P}}  Y_r(\mu)$  is contained in a low-dimensional space $Y^*_r$. Let $d^*$ be the dimension of this space. By definition, if $\bTheta$ is a $( \varepsilon, \delta, d^*)$ oblivious $U \to \ell_2$ subspace embedding, then it is a $U \to \ell_2$ $\varepsilon$-embedding for $Y^*_r$, and hence for every  $Y_r(\mu)$, simultaneously, with probability at least $1-\delta$.

The lower bound for $\alpha^{\bTheta}_r(\mu)$ in~Proposition\nobreakspace \ref {thm:skcea} depends on the product $\varepsilon a_r(\mu)$. {In particular, to guarantee positivity of $\alpha^{\bTheta}_r(\mu)$ and ensure well-posedness of~(\ref{eq:SKgalproj}), condition $\varepsilon a_r(\mu) <1$ has to be satisfied.} The coefficient $a_r(\mu)$ is {bounded from above} by $\frac{\beta(\mu)}{\alpha_r(\mu)}$. Consequently, $a_r(\mu)$ for coercive well-conditioned operators is expected to be lower than for non-coercive ill-conditioned $\bA(\mu)$. The condition number and coercivity of $\bA(\mu)$, however, do not fully characterize $a_r(\mu)$. This coefficient rather reflects how well $U_r$ corresponds to its image $\{ \bA(\mu) \bx: \bx \in U_r \}$ through the map $\bA(\mu)$. 
For example, if the basis for $U_r$ is formed from eigenvectors of $\bA(\mu)$ then $a_r(\mu)=1$. We also would like to note that the performance of {the} random sketching technique depends on {the operator,} only when it is employed for estimating the Galerkin projection. The accuracy of estimation of the residual error and the goal-oriented correction depends on the quality of sketching matrix $\bTheta$ but not on $\bA(\mu)$. In addition, to make the performance of random sketching completely insensitive to the operator's properties, one can consider another type of projection (randomized minimal residual projection) for $\bu_r(\mu)$ as is discussed in~\cite{balabanov2018}.

The {coordinates} of the solution $\bu_r(\mu)$ of~(\ref {eq:SKgalproj}) can be found by solving 
\begin{equation} \label{eq:skreduced_system}
 \bA_r(\mu)\ba_r(\mu) = \bb_r(\mu), 
\end{equation} 
where $\bA_r(\mu):= \bU_r^{\mathrm{H}}\bTheta^{\mathrm{H}}\bTheta\bR_U^{-1}\bA(\mu)\bU_r \in \mathbb{K}^{r\times r}$ and $\bb_r(\mu):= \bU_r^{\mathrm{H}}\bTheta^{\mathrm{H}}\bTheta\bR_U^{-1}\bb(\mu) \in \mathbb{K}^{r}$. 
\begin{proposition} \label{thm:skalgstability}
 Let $\bTheta$ be a $U \to \ell_2$ $\varepsilon$-embedding for $U_r$, and let $\bU_r$ be orthogonal with respect to $\langle \cdot, \cdot  \rangle^{\bTheta}_U$. Then the condition number of $\bA_r(\mu)$ in~(\ref {eq:skreduced_system}) is bounded by $\sqrt{\frac{1+\varepsilon}{1-\varepsilon}} \frac{\beta^\bTheta_r(\mu)}{\alpha^\bTheta_r(\mu)}$.
\begin{proof}
See appendix.
\end{proof} 
\end{proposition}

\subsection{Error estimation} 
Let $\bu_r(\mu) \in U_r$ be an approximation of $\bu(\mu)$. Consider the following error estimator:
\begin{equation} \label{eq:skerrorind}
 \Delta^{\bTheta}(\bu_r(\mu);\mu):= \frac{\| \br(\bu_r(\mu);\mu) \|^{\bTheta}_{U'}}{\eta(\mu)},
\end{equation}
where $\eta(\mu)$ is defined by~(\ref {eq:eta}). Below we show that under certain conditions,  $\Delta^{\bTheta}(\bu_r(\mu);\mu)$ is guaranteed to be close to the classical error indicator $\Delta(\bu_r(\mu);\mu)$.
\begin{proposition} \label{thm:skerrorind}
 If $\bTheta$ is a $U \to \ell_2$ $\varepsilon$-embedding for $\mathrm{span} \{ \bR^{-1}_U\br(\bu_r(\mu);\mu) \}$, then 
 \begin{equation} \label{eq:skerroropt}
  \sqrt{1-\varepsilon}\Delta(\bu_r(\mu);\mu)\leq \Delta^{\bTheta}(\bu_r(\mu);\mu)\leq \sqrt{1+\varepsilon}\Delta(\bu_r(\mu);\mu).
 \end{equation}
\begin{proof}
 See appendix.
 \end{proof} \end{proposition}

\begin{corollary} \label{thm:skerrorind1}
 If $\bTheta$ is a $U \to \ell_2$ $\varepsilon$-embedding for $Y_r(\mu)$, then relation~(\ref {eq:skerroropt}) holds. 
\end{corollary}

\subsection{Primal-dual correction} \label{sk_pd_correction}
The sketching technique can be applied to the dual problem in exactly the same manner as to the primal problem. 

Let $\bu_r(\mu) \in U_r$ and $\bu_r^\mathrm{du}(\mu) \in  U^{\mathrm{du}}_r$ be approximations of $\bu(\mu)$ and $\bu^\mathrm{du}(\mu)$, respectively. 
The sketched version of the primal-dual correction~(\ref {eq:correction}) can be expressed as follows 
\begin{equation}  \label{eq:skcorrection}
	 s_r^{\mathrm{spd}}(\mu):= s_r(\mu)- \langle  \bu_r^\mathrm{du}(\mu), \bR_U^{-1} \br(\bu_r(\mu);\mu) \rangle^{\bTheta}_{U}.
\end{equation}

\begin{proposition} \label{thm:skerror_correction}
 If $\bTheta$ is $U \to \ell_2$ $\varepsilon$-embedding for $\mathrm{span}\{ \bu_r^\mathrm{du}(\mu), \bR^{-1}_U\br(\bu_r(\mu);\mu) \}$, then 
 \begin{equation} \label{eq:skerror_correction}
  |s(\mu)-  s_r^{\mathrm{spd}}(\mu)|\leq \frac{\| \br(\bu_r(\mu);\mu) \|_{U'}}{\eta(\mu)} ((1+\varepsilon)\| \br^{\mathrm{du}}(\bu_r^\mathrm{du}(\mu);\mu) \|_{U'}+ \varepsilon\| \bl(\mu) \|_{U'}).
 \end{equation}
\begin{proof}
See appendix.
 \end{proof} \end{proposition}

\begin{remark}
 We observe that the new version of primal-dual correction~(\ref {eq:skcorrection}) and its error bound~(\ref {eq:skerror_correction}) are no longer symmetric {in terms} of the primal and dual solutions. When the residual error of $\bu_r^\mathrm{du}(\mu)$ is smaller than the residual error of $\bu_r(\mu)$, it can be more beneficial to consider the dual problem as the primal one and vice versa. 
\end{remark}

\begin{remark} \label{rmk:compliant}
 Consider the so called ``compliant case'', i.e., $\bA(\mu)$ is self-adjoint, and $\bb(\mu)$ is equal to $\bl(\mu)$  up to a scaling factor. In such a case the same solution (up to a scaling factor) should be used for both the primal and the dual problems. If the approximation $\bu_r(\mu)$ of $\bu(\mu)$ is obtained with the classical Galerkin projection then the primal-dual correction is automatically included to the primal output quantity, i.e., $s_r(\mu) = s_r^{\mathrm{pd}}(\mu)$. {A} similar scenario can be observed for the sketched  Galerkin projection. If $\bu_r(\mu)$ satisfies~(\ref {eq:SKgalproj}) and the same $\bTheta$ is considered for both the projection and the inner product in~(\ref {eq:skcorrection}), then $s_r(\mu)= s^{\mathrm{spd}}_r(\mu)$. 
\end{remark}

It follows that if $\varepsilon$ is of the order of $\| \br^{\mathrm{du}}(\bu_r^\mathrm{du}(\mu);\mu) \|_{U'}/ \| \bl(\mu) \|_{U'}$, then the quadratic dependence in residual norm of the error bound is preserved. For relatively large $\varepsilon$, however, the error is expected to be proportional to $\varepsilon \| \br(\bu_r(\mu); \mu) \|_{U'}$. Note that $\varepsilon$ can decrease slowly with $k$ ({typically} $\varepsilon = \mathcal{O}(k^{-1/2})$, see~Propositions\nobreakspace \ref {thm:Rademacher} and\nobreakspace  \ref {thm:P-SRHT}). Consequently, preserving high precision of the primal-dual correction can require large sketching matrices. 

More accurate but yet efficient estimation of $s^{\mathrm{pd}}(\mu)$ can be obtained by introducing an approximation $\bw^\mathrm{du}_r(\mu)$ of $\bu_r^\mathrm{du}(\mu)$ such that the inner products with $\bw^\mathrm{du}_r(\mu)$ are efficiently computable. Such approximation does not have to be very precise. As it will become clear later, it is sufficient to have $\bw^\mathrm{du}_r(\mu)$ such that $\|\bu_r^\mathrm{du}(\mu) - \bw^\mathrm{du}_r(\mu) \|_U$ is of the order of $\varepsilon^{-1} \|\bu_r^\mathrm{du}(\mu) - \bu^\mathrm{du}(\mu)\|_U$. A possible choice is to let $\bw^\mathrm{du}_r(\mu)$ be the orthogonal projection of $\bu_r^\mathrm{du}(\mu)$ on a certain subspace $W^\mathrm{du}_r \subset U$, where $W^\mathrm{du}_r$ is such that it approximates well $\{ \bu_r^\mathrm{du} (\mu):~\mu \in \mathcal{P} \}$ but is much cheaper to operate with than $U^\mathrm{du}_r$, e.g., if it has a smaller dimension. One can simply take $W^\mathrm{du}_r = U^\mathrm{du}_i$ (the subspace spanned
by the first $i^\mathrm{du}$ basis vectors obtained during the generation of 
$U^\mathrm{du}_r$), for some small $i^\mathrm{du} < r^\mathrm{du}$.  A better approach consists in using a greedy algorithm or the POD method with a training set $\{ \bu_r^\mathrm{du}(\mu) : \mu \in \mathcal{P}_{\mathrm{train}} \}$. {We could also choose $W^\mathrm{du}_r$ as the subspace associated with a coarse-grid interpolation of the solution.} In this case, even if $W^\mathrm{du}_r$ has a high dimension, it can be operated with efficiently because its basis vectors are sparse. Strategies for the efficient construction of approximation spaces for $\bu_r^\mathrm{du}(\mu)$ (or $\bu_r(\mu)$) are provided in~\cite{balabanov2018}.  Now, let us assume that $\bw^\mathrm{du}_r(\mu)$ is given and consider the following estimation of $s_r^\mathrm{pd}(\mu)$:
\begin{equation}  \label{eq:skcorrection2}
	 s^{\mathrm{spd+}}_r(\mu):= s_r(\mu) - \langle \bw^\mathrm{du}_r(\mu), \br(\bu_r(\mu);\mu) \rangle - \langle \bu_r^\mathrm{du}(\mu) - \bw^\mathrm{du}_r(\mu), \bR_U^{-1}\br(\bu_r(\mu);\mu) \rangle^{\bTheta}_{U}.
\end{equation}
We notice that $s^{\mathrm{spd+}}_r(\mu)$ can be evaluated efficiently but, at the same time, it has better accuracy than $s_r^{\mathrm{spd}}(\mu)$ in~(\ref {eq:skerror_correction}). By similar consideration as in~Proposition\nobreakspace \ref {thm:skerror_correction} it can be shown that for preserving quadratic dependence in the error for $s^{\mathrm{spd+}}_r(\mu)$, it is sufficient to have $\varepsilon$ of the order of $\| \bu_r^\mathrm{du}(\mu) - \bu^\mathrm{du}(\mu) \|_{U'}/\| \bu_r^\mathrm{du}(\mu) - \bw^\mathrm{du}_r(\mu) \|_{U'}$.

Further, we assume that the accuracy of $s_r^{\mathrm{spd}}(\mu)$ is sufficiently good so that there is no need to consider a corrected estimation $s^{\mathrm{spd+}}_r(\mu)$. For other cases the methodology can be applied similarly.

\subsection{Computing the sketch} \label{sketch}
In this section we introduce the concept of a sketch of the reduced order model. A sketch contains all the information needed for estimating the output quantity and certifying this estimation. It can be efficiently computed in basically any computational environment.

We restrict ourselves to solving the primal problem. Similar considerations also apply for the dual problem and primal-dual correction. The $\bTheta$-sketch of a reduced model associated with a subspace $U_r$ is defined as
\begin{equation} \label{eq:sketchUr}
\left \{ \left \{ \bTheta \bx, \bTheta \bR_U^{-1} \br(\bx; \mu), \langle \bl(\mu), \bx \rangle  \right \}: ~~ \bx \in U_r \right \}
\end{equation}
In practice, each element of~(\ref {eq:sketchUr}) can be represented by the coordinates of $\bx$ associated with $\bU_r$, i.e., a vector $\ba_r \in \mathbb{K}^{r}$ such that $\bx=\bU_r \ba_r$, the sketched reduced basis matrix $\bU^{\bTheta}_r:=\bTheta \bU_r$ and the following small parameter-dependent matrices and vectors:
\begin{equation} \label{eq:sketch}
\bV^{\bTheta}_r(\mu):=\bTheta \bR_U^{-1}\bA(\mu) \bU_r,~~ {\bb^{\bTheta}}(\mu):=\bTheta \bR_U^{-1} \bb(\mu), ~~\bl_r(\mu)^{\mathrm{H}}:= \bl(\mu)^{\mathrm{H}} \bU_r. 
\end{equation}
Throughout the paper, matrix $\bU^{\bTheta}_r$ and the affine expansions of $\bV^{\bTheta}_r(\mu)$, $\bb^{\bTheta}(\mu)$ and $\bl_r(\mu)$ shall be referred to as the $\bTheta$-sketch of $\bU_r$. This object should not be confused with the $\bTheta$-sketch associated with a subspace $U_r$ defined by~(\ref{eq:sketchUr}). The $\bTheta$-sketch of $\bU_r$ shall be used for characterizing the elements of the $\bTheta$-sketch associated with $U_r$ similarly as $\bU_r$ is used for characterizing the vectors in $U_r$.

The affine expansions of $\bV^{\bTheta}_r(\mu)$, ${\bb^{\bTheta}}(\mu)$ and $\bl_r(\mu)$ can be obtained either by considering the affine expansions of $\bA(\mu)$, $\bb(\mu)$, and $\bl(\mu)$\footnote{{For instance, if $\bA(\mu)=\sum^{m_A}_{i=1} \phi_i(\mu) \bA_i$, then $ \bV_r^{\bTheta}(\mu) = \sum^{m_A}_{i=1} \phi_i(\mu) \left (\bTheta \bR_U^{-1}\bA_i \bU_r \right )$. Similar relations can also be derived for $\bb^{\bTheta}(\mu)$ and $\bl_r(\mu)$.}} or with {the} empirical interpolation method (EIM)~\cite{maday2007general}. Given the sketch, the affine expansions of the quantities (e.g., $\bA_r(\mu)$ in~(\ref {eq:skreduced_system})) needed for efficient evaluation of the output can be computed with negligible cost. Computation of the $\bTheta$-sketch determines the cost of the offline stage and it has to be performed depending on the computational environment. We assume that the affine factors of $\bl_r(\mu)$ are cheap to evaluate. Then the remaining computational cost is  mainly associated with the following three operations: computing the samples (snapshots) of the solution (i.e., solving the full order problem for several $\mu \in \mathcal{P}$), performing matrix-vector products with $\bR_U^{-1}$ and the affine factors of $\bA(\mu)$ (or $\bA(\mu)$ evaluated at the interpolation points for EIM), and evaluating matrix-vector products with $\bTheta$.

{The cost of obtaining the snapshots is assumed to be low compared to the cost of other offline computations such as evaluations of high dimensional inner and matrix-vector products.}
This is the case when the snapshots are computed beyond the main routine {using} highly optimised linear solver or a powerful server with limited budget. This is also the case when the snapshots are obtained on distributed machines with expensive communication costs.
{Solutions of linear systems of equations should have only a minor impact on the overall cost of an algorithm even when the basic metrics of efficiency, such as the complexity (number of floating point operations) and memory consumption, are considered. 
For large-scale problems solved in sequential or limited memory environments the computation of each snapshot should have log-linear (i.e., $\mathcal{O}(n (\log{n})^d)$, for some small $d$) complexity and memory requirements. Higher complexity or memory requirements are usually not acceptable {with} standard architectures.
In fact, in recent years there was an extensive development of methods for solving large-scale linear systems of equations~\cite{hackbusch2015book,bebendorf2008book,Elman2014} allowing computation of the snapshots with log-linear number of flops and bytes of memory (see for instance \cite{Xia2009,bjorn2011,Martinsson2009,Lee2013,Boman2008}). On the other hand, for classical model reduction, the evaluation of multiple inner products for the affine terms of reduced systems~(\ref {eq:reduced_system}) and the quantities for error estimation (see Section\nobreakspace\ref{efficient_res_norm}) require  $\mathcal{O}(n r^2 m^2_A + n m^2_b)$ flops, with $m_A$ and $m_b$ being the numbers of terms in affine expansions of $\bA(\mu)$ and $\bb(\mu)$, respectively, and $\mathcal{O}(n r)$ bytes of memory. We see that indeed the complexity and memory consumption of the offline stage can be highly dominated by the postprocessing of the snapshots but not their computation.}

{{The} matrices $\bR_U$ and $\bA(\mu)$ should be sparse or maintained in a hierarchical format~\cite{hackbusch2015book}, so that they can be multiplied by a vector using (log-)linear complexity and {storage consumption}.} {Multiplication of $\bR_U^{-1}$ by a vector should also be an inexpensive operation with the cost comparable to the cost of computing matrix-vector products with $\bR_U$.} For many problems it can be beneficial to precompute a factorization of $\bR_U$ and to use it for efficient multiplication of $\bR^{-1}_U$ by multiple vectors. {Note that for the typical $\bR_U$ (such as stiffness and mass matrices) originating from standard discretizations of partial differential equations in two spatial dimensions, a sparse Cholesky decomposition can be precomputed using $\mathcal{O}(n^{3/2})$ flops and then used for multiplying $\bR_U^{-1}$ by vectors with $\mathcal{O}(n \log n)$ flops.} For discretized PDEs in higher spatial dimensions, or problems where $\bR_U$ is dense, the classical Cholesky decomposition can be more burdensome to obtain and use. {For better efficiency, the matrix $\bR_U$ can be approximated by $\tilde{\bQ}^\mathrm{H} \tilde{\bQ}$ (with log-linear number of flops) using incomplete or hierarchical \cite{Bebendorf2007} Cholesky factorizations. Iterative Krylov methods with good preconditioning {are} an alternative way for computing products of $\bR_U^{-1}$ with vectors with log-linear complexity~\cite{Boman2008}. Note that although multiplication of $\bR_U^{-1}$ by a vector and computation of a snapshot both require solving high-dimensional systems of equations, the cost of the former operation should be considerably less than the cost of the later one due to good properties of $\bR_U$ (such as  positive-definiteness, symmetry, and parameter-independence providing ability of precomputing a decomposition).} 
In a streaming environment, where the snapshots are provided as data-streams, a special care has to be payed to the memory constraints. It can be important to maintain $\bR_U$ and the affine factors (or evaluations at EIM interpolation points) of $\bA(\mu)$ with a reduced storage consumption. For discretized PDEs, for example, the entries of these matrices {(if they are sparse)} can be generated subdomain-by-subdomain on the fly. In such a case the conjugate gradient method can be a good choice for evaluating products of $\bR^{-1}_U$ with vectors. In very extreme cases, e.g., where storage of even a single large vector is forbidden, $\bR_U$ can be approximated by a block matrix and inverted block-by-block on the fly.

Next we discuss an efficient implementation of $\bTheta$. We assume that $$\bTheta=\bOmega \bQ,$$ where $\bOmega \in \mathbb{K}^{k \times s}$ is a classical oblivious $\ell_2 \to \ell_2$ subspace embedding and $\bQ \in \mathbb{K}^{s\times n}$ is such that $\bQ^\mathrm{H}\bQ= \bR_U$ (see~Propositions\nobreakspace \ref {thm:Rademacher},  \ref {thm:P-SRHT} and\nobreakspace  \ref {thm:buildepsilon_embedding}). 
  
{The} matrix $\bQ$ can be expected to have a cost of multiplication by a vector comparable to $\bR_U$. If needed, this matrix can be generated block-wise (see Remark\nobreakspace \ref {rmk:Q_eval}) on the fly similarly to $\bR_U$. 

{For environments where the measure of efficiency is the number of flops, a sketching matrix $\bOmega$  with fast matrix-vector multiplications such as P-SRHT {is preferable}. The complexity of {a} matrix-vector product for P-SRHT is only $2s \log_2(k+1)$, with $s$ being the power of $2$ such that $n \leq s < 2n$~\cite{ailon2009fast,boutsidis2013improved}\footnote{{The straightforward implementation of P-SRHT using the fast Walsh-Hadamard transform results in $s \log_2{(s)}$ complexity of multiplication by a vector, which yields similar computational costs as the procedure from~\cite{ailon2009fast}.}}. Consequently, assuming that $\bA(\mu)$ is sparse, that multiplications of $\bQ$ and $\bR^{-1}_U$ by a vector take $\mathcal{O}(n(\log{n})^d)$ flops, and that $\bA(\mu)$ and $\bb(\mu)$ admit affine expansions with $m_A$ and $m_b$ terms respectively, the overall complexity of computation of a $\bTheta$-sketch of $\bU_r$, using {a} P-SRHT matrix as $\bOmega$, from the snapshots is only $$\mathcal{O}(n[ r m_A \log{k} +  m_b \log{k} + r m_A (\log{n})^d]).$$
This complexity can be much less than the complexity of construction of the classical reduced model from $\bU_r$, which is $\mathcal{O}(n [r^2 m^2_A + m^2_b + r m_A (\log{n})^d])$ (including the precomputation of quantities needed for online evaluation of the residual error).} 
The efficiency of an algorithm can be also measured {in terms of} the number of passes taken over the data. Such a situation may arise when there is a restriction on the accessible amount of fast memory. In this scenario, both structured and unstructured matrices may provide drastic reductions of the computational cost. Due to robustness and simplicity of implementation, we suggest using Gaussian or Rademacher matrices over the others. For these matrices a seeded random number generator has to be utilized. It allows accessing the entries of $\bOmega$ on the fly with negligible storage costs~\cite{halko2011finding}. In a streaming environment, multiplication of Gaussian or Rademacher matrices by a vector can be performed block-wise. 

Note that all aforementioned operations are well suited for parallelization. Regarding distributed computing, a sketch of each snapshot can be obtained on a separate machine with absolutely no communication. The cost of transferring the sketches to the master machine will depend on the number of rows of $\bTheta$ but not the size of the full order problem. 

Finally, let us comment on orthogonalization of $\bU_r$ with respect to {$\langle \cdot, \cdot \rangle_U^{\bTheta}$}. This procedure is particularly important for numerical stability of the reduced system of equations (see Proposition\nobreakspace \ref {thm:skalgstability}). In our applications we are interested in obtaining a sketch of the orthogonal matrix but not the matrix itself. In such a case, operating with large-scale matrices and vectors is not necessary. Let us assume to be given a sketch of $\bU_r$ associated with $\bTheta$. Let $\bT_r \in \mathbb{K}^{r \times r}$ be such that $ \bU^{\bTheta}_r \bT_r$ is  orthogonal  with respect to {$\langle \cdot, \cdot \rangle$}. Such a matrix can be obtained with a standard algorithm, e.g., QR factorization. It can be easily verified that $\bU^*_r := \bU_r \bT_r$ is  orthogonal with respect to {$\langle \cdot, \cdot \rangle_U^{\bTheta}$}. We have,
\begin{equation*}
 \bTheta\bU^*_r= \bU^{\bTheta}_r\bT_r,~~ \bTheta\bR_U^{-1}\bA(\mu)\bU^*_r= \bV^{\bTheta}_r(\mu)\bT_r, \textup{ and } \bl(\mu)^{\mathrm{H}}\bU^*_r= \bl_r(\mu)^{\mathrm{H}}\bT_r.
\end{equation*}
Therefore, the sketch of $\bU^*_r$ can be computed, simply, by multiplying $\bU^{\bTheta}_r$ and the affine factors of $\bV^{\bTheta}_r(\mu)$, and $\bl_r(\mu)^{\mathrm{H}}$, by $\bT_r$.

\subsection{Efficient evaluation of the residual norm} \label{efficient_res_norm}
Until now we discussed how random sketching can be used for reducing the offline cost of precomputing  factors of affine decompositions of the reduced operator and the reduced right-hand side. Let us now focus on the cost of the online stage. Often, the most expensive part of the online stage is the evaluation of the quantities needed for computing the residual norms for a posteriori error estimation due to many summands in their affine expansions. In addition, as was indicated in~\cite{casenave2014,buhr2014numerically}, the classical procedure for the evaluation of the residual norms can be sensitive to round-off errors. Here we provide a {less expensive} way of computing the residual norms, which simultaneously offers a better numerical stability. 

Let $\bu_r(\mu) \in U_r$ be an approximation of $\bu(\mu)$, and $\ba_r(\mu) \in \mathbb{K}^{r}$ be the coordinates of $\bu_r(\mu)$ associated with $\bU_r$, i.e.,  $\bu_r(\mu) = \bU_r \ba_r(\mu)$. The classical algorithm for evaluating the residual norm $\| \br(\bu_r(\mu);\mu) \|_{U'}$ for a large finite set of parameters $\mathcal{P}_{\mathrm{test}} \subseteq \mathcal{P}$ proceeds with expressing $\| \br(\bu_r(\mu);\mu) \|^2_{U'}$ in the following form~\cite{haasdonk2017reduced}
\begin{equation} \label{eq:compute_res}
 \| \br(\bu_r(\mu); \mu) \|^2_{U'}= \langle \ba_r(\mu), \bM(\mu)\ba_r(\mu) \rangle+ 2\mathrm{Re}(\langle \ba_r(\mu), \bbm(\mu) \rangle) +m(\mu), 
\end{equation}
where affine expansions of $\bM(\mu):= \bU_r^{\mathrm{H}}\bA(\mu)^{\mathrm{H}}\bR_U^{-1}\bA(\mu)\bU_r$, $\bbm(\mu):= \bU_r^{\mathrm{H}}\bA(\mu)^{\mathrm{H}}\bR_U^{-1} \bb(\mu)$ and $m(\mu):=\bb(\mu)^{\mathrm{H}}\bR_U^{-1}\bb(\mu)$ can be precomputed during the offline stage and used for efficient online evaluation of these quantities for each $\mu \in \mathcal{P}_{\mathrm{test}}$. {If $\bA(\mu)$ and $\bb(\mu)$ admit affine representations with $m_A$ and $m_b$ terms, respectively, then the associated affine expansions of $\bM(\mu)$, $\bbm(\mu)$ and $m(\mu)$ contain $\mathcal{O}(m^2_A), \mathcal{O}(m_A m_b), \mathcal{O}(m^2_b)$ terms respectively, therefore requiring $\mathcal{O}(r^2 m^2_A+m^2_b)$ flops for their online evaluations.}

An approximation of the residual norm can be obtained in a more efficient and numerically stable way  with {the} random sketching technique. Let us assume that $\bTheta \in \mathbb{K}^{k\times n} $ is a $U \to \ell_2$ embedding such that $\| \br(\bu_r(\mu);\mu) \|^{\bTheta}_{U'}$ approximates well $\| \br(\bu_r(\mu);\mu)  \|_{U'}$   (see Proposition\nobreakspace \ref {thm:skerrorind}). Let us also assume that the factors of affine decompositions of $\bV^{\bTheta}_r(\mu)$ and $ \bb^{\bTheta}(\mu)$ have been precomputed and are available. 
For each $\mu \in \mathcal{P}_{\mathrm{test}}$ an estimation of the residual norm can be provided by
\begin{equation} \label{eq:skcompres}
 \| \br(\bu_r(\mu);\mu) \|_{{U}'} \approx \| \br(\bu_r(\mu);\mu) \|^{\bTheta}_{{U}'}= \|\bV^{\bTheta}_r(\mu)\ba_r(\mu)- \bb^{\bTheta}(\mu) \|.
\end{equation}
We notice that ${\bV^{\bTheta}_r}(\mu)$ and  $\bb^{\bTheta}(\mu)$ have less  terms in their affine expansions than the quantities in~(\ref {eq:compute_res}). The sizes of ${\bV^{\bTheta}_r}(\mu)$ and $\bb^{\bTheta}(\mu)$, however, can be too large to provide any online cost reduction.
%\begin{remark}
% The sketched residual norm $\| \br(\bu_r(\mu); \mu)  \|_{U'}^{\bTheta}$ can be also evaluated in a classical way, i.e., by expressing  $(\| \br(\bu_r(\mu); \mu) \|^{\bTheta}_{U'})^2$ as 
% \begin{equation}
%  (\| \br(\bu_r(\mu); \mu) \|_{U'}^{\bTheta})^2= \langle \ba_r(\mu), %\bM^{\bTheta}(\mu)\ba_r(\mu) \rangle + 2\mathrm{Re}( \langle %\ba_r(\mu), \bbm^{\bTheta}(\mu) \rangle) +m^{\bTheta}(\mu), 
% \end{equation}
% where $\bM^\bTheta(\mu):= \bV^{\bTheta}_r(\mu)^{\mathrm{H}}\bV^{\bTheta}_r(\mu)$, $\bbm^\bTheta(\mu):= \bV^{\bTheta}_r(\mu)^{\mathrm{H}}\bb^{\bTheta}(\mu)$ and $m^\bTheta(\mu):= \bb^{\bTheta}(\mu)^{\mathrm{H}}\bb^{\bTheta}(\mu)$ are precomputed during the offline stage. 
%\end{remark}
In order to improve the efficiency, we introduce an additional  $(\varepsilon, \delta, 1)$ oblivious $\ell_2 \to \ell_2$ subspace embedding $\bGamma \in \mathbb{K}^{k'\times k}$. The theoretical bounds for the number of rows of Gaussian, Rademacher and P-SRHT matrices sufficient to satisfy the $(\varepsilon, \delta, 1)$ oblivious $\ell_2 \to \ell_2$ subspace embedding property can be obtained from~\cite[Lemmas~4.1 and 5.1]{achlioptas2003database}~and~Proposition\nobreakspace \ref {thm:P-SRHT}. They are presented in Table\nobreakspace \ref {tab:numrows}. Values are shown for $\varepsilon= 0.5$ and varying probabilities of failure $\delta$. We note that in order to account for the case $\mathbb{K} = \mathbb{C}$ we have to employ \cite[Lemmas~4.1 and 5.1]{achlioptas2003database} for the real part and the imaginary part of a vector, separately, with a union bound for the probability of success.
\begin{table}[tbhp]
 \caption{The number of rows of Gaussian (or Rademacher) and P-SRHT matrices sufficient to satisfy {the} $(1/2, \delta, 1)$ oblivious $\ell_2 \to \ell_2$ $\varepsilon$-subspace embedding property.}
 \label{tab:numrows}
 \centering
 \scalebox{0.9}{
 \begin{tabular}{|l|l|l|l|l|} \hline
  & $\delta = 10^{-3}$ & $\delta = 10^{-6}$ & $\delta = 10^{-12}$ & $\delta = 10^{-18}$ \\ \hline
  {Gaussian} & $200$ & $365$ & $697$ & $1029$ \\ [2pt] \hline
  {P-SRHT} & ${{96.4}(8\log{k}+69.6)}$ & ${{170}(8\log{k}+125)}$ & ${{313}(8\log{k}+236)}$ & ${{454}(8\log{k}+346)}$  \\ \hline
 \end{tabular}
 }
\end{table}

\begin{remark}
 In practice the  bounds provided in Table\nobreakspace \ref {tab:numrows} are  pessimistic (especially for P-SRHT) and much smaller $k'$ (say, $k' =100$) may provide desirable results. In addition, in our experiments any significant difference in performance between Gaussian matrices, Rademacher matrices and P-SRHT  has not been revealed.
\end{remark}

We observe that the number of rows of $\bGamma$ can be chosen independent (or weakly dependent) of the number of rows of $\bTheta$. Let $\bPhi:= \bGamma\bTheta$. By definition, for each $\mu \in \mathcal{P}_{\mathrm{test}}$
\begin{equation} \label{eq:res_concentration}
 \mathbb{P} \left ( \left | (\| \br(\bu_r(\mu);\mu) \|^{\bTheta}_{V'})^2- (\| \br(\bu_r(\mu);\mu) \|^{\bPhi}_{V'})^2 \right | \leq \varepsilon (\| \br(\bu_r(\mu);\mu) \|^{\bTheta}_{V'})^2\right ) \geq 1-\delta;
\end{equation}
which means that $\| \br(\bu_r(\mu);\mu) \|^{\bPhi}_{V'}$ is an {$\mathcal{O}(\varepsilon)$-}accurate approximation of $\| \br(\bu_r(\mu);\mu) \|_{V'}$ with high probability. The probability of success for all $\mu \in \mathcal{P}_{\mathrm{test}}$ simultaneously can be guaranteed with a union bound. In its turn, $\| \br(\bu_r(\mu);\mu) \|^{\bPhi}_{V'}$ can be computed from
\begin{equation} \label{eq:eff_comp_res}
 \| \br(\bu_r(\mu); \mu) \|^{\bPhi}_{V'}= \|\bV^{\bPhi}_r(\mu)\ba_r(\mu)- \bb^{\bPhi}(\mu) \|, 
\end{equation}
where $\bV^{\bPhi}_r(\mu):= \bGamma {\bV^{\bTheta}_r}(\mu)$ and $\bb^{\bPhi}(\mu):= \bGamma\bb^{\bTheta}(\mu)$. The efficient way of computing $\| \br(\bu_r(\mu);\mu) \|^{\bPhi}_{V'}$ for every $\mu \in \mathcal{P}_{\mathrm{test}}$ consists in two stages. Firstly, we generate $\bGamma$ and precompute affine expansions of $\bV^{\bPhi}_r(\mu)$ and $\bb^{\bPhi}(\mu)$ by multiplying each affine factor of $\bV^{\bTheta}_r(\mu)$ and $\bb^{\bTheta}(\mu)$ by $\bGamma$. {The cost of this stage is independent of $ \# \mathcal{P}_\mathrm{test}$ (and $n$, of course) and becomes negligible for $\mathcal{P}_\mathrm{test}$ of large size.}  In the second stage, for each parameter $\mu \in \mathcal{P}_{\mathrm{test}}$, $\| \br(\bu_r(\mu);\mu) \|^{\bPhi}_{V'}$ is evaluated from~(\ref {eq:eff_comp_res}) using precomputed affine expansions. 
{The quantities $\bV^{\bPhi}_r(\mu)$ and  $\bb^{\bPhi}(\mu)$ contain at most the same number of terms as $\bA(\mu)$ and $\bb(\mu)$ in their affine expansion. Consequently, if $\bA(\mu)$ and $\bb(\mu)$ are parameter-separable with $m_A$ and $m_b$ terms, respectively, then each evaluation of $\| \br(\bu_r(\mu);\mu) \|^{\bPhi}_{V'}$ from $\ba_r(\mu)$ requires only $\mathcal{O}(k' r m_A+k' m_b)$ flops, which can be much less than the $\mathcal{O}(r^2 m^2_A+m^2_b)$ flops required for evaluating~(\ref{eq:compute_res}). 
	Note that the classical computation of the residual norm by taking the square root of $\| \br(\bu_r(\mu); \mu) \|^2_{U'}$ evaluated using~(\ref{eq:compute_res}) can suffer from round-off errors. On the other hand, the evaluation of $\| \br(\bu_r(\mu); \mu) \|^{\bPhi}_{V'}$ using~(\ref{eq:eff_comp_res}) is less sensitive to round-off errors since here we proceed with direct evaluation of the (sketched) residual norm but not its square.}   
	
\begin{remark}
 {If $\mathcal{P}_{\mathrm{test}}$ is provided a priori, then the random matrix $\bGamma$ can be generated and multiplied by the affine factors of $\bV^{\bTheta}_r(\mu)$ and $\bb^{\bTheta}(\mu)$ during the offline stage.}
\end{remark}
\begin{remark}
 For algorithms where $\mathcal{P}_{\mathrm{test}}$ or $U_r$ are selected adaptively based on a criterion depending on the residual norm (e.g., the classical greedy algorithm outlined in Section\nobreakspace \ref {Greedy}), a new realization of $\bGamma$ has to be generated at  each iteration. If the same realization of $\bGamma$ is used for several iterations of the adaptive algorithm, care must be taken when characterizing the probability of success. This probability can decrease exponentially with the number of iterations, which requires to use considerably larger $\bGamma$. 	 Such option can be justified only for the cases when the cost of multiplying affine factors by $\bGamma$  greatly dominates the cost of the second stage, i.e., evaluating  $\| \br(\bu_r(\mu);\mu) \|^{\bPhi}_{V'}$ for all $\mu \in 
\mathcal{P}_{\mathrm{test}}$. 
\end{remark}

\section{Efficient reduced basis generation} \label{efficient_RB}
In this section we show how the sketching technique can be used for improving the generation of reduced approximation spaces with greedy algorithm for RB, or a POD. Let $\bTheta \in \mathbb{K}^{k\times n}$ be a $U \to \ell_2$ subspace embedding. 

\subsection{Greedy algorithm}
Recall that at each iteration of the greedy algorithm (see Section\nobreakspace\ref{Greedy}) the basis is enriched with a new sample (snapshot) $\bu(\mu^{i+1})$, selected based on error indicator $\widetilde{\Delta}(U_i; \mu)$. The standard choice is $\widetilde{\Delta}(U_i; \mu):= \Delta (\bu_i(\mu);\mu)$ where $\bu_i(\mu) \in U_i$ satisfies~(\ref {eq:galproj}).  Such error indicator, however, can lead to very expensive computations. The error indicator can be modified to $\widetilde{\Delta}(U_i; \mu):= \Delta^{\bTheta}(\bu_i(\mu);\mu)$, where $\bu_i(\mu) \in U_i$ is an approximation of $\bu(\mu)$ which does not necessarily satisfy~(\ref {eq:galproj}). Further, we restrict ourselves  to the case when  $\bu_i(\mu)$ is the sketched Galerkin projection~(\ref {eq:SKgalproj}). If there is no interest in reducing the cost of evaluating inner products but only reducing the cost of evaluating residual norms, it can be more relevant to consider the classical Galerkin projection~(\ref {eq:galproj}) instead of~(\ref {eq:SKgalproj}). 

A quasi-optimality guarantee for the greedy selection with $\widetilde{\Delta}(U_i; \mu):= \Delta^{\bTheta}(\bu_i(\mu);\mu)$ can be derived from Propositions\nobreakspace \ref {thm:SKquasi-opt} and\nobreakspace  \ref {thm:skcea} and\nobreakspace Corollary\nobreakspace \ref {thm:skerrorind1}. At iteration $i$ of the greedy algorithm, we need $\bTheta$ to be a $U \to \ell_2$ $\varepsilon$-subspace embedding for $Y_i(\mu)$ defined in~(\ref {eq:Y_r}) for all $\mu \in \mathcal{P}_{\mathrm{train}}$. One way to achieve this is to generate a new realization of an oblivious $U \to \ell_2$ subspace embedding $\bTheta$ at each iteration of the greedy algorithm. Such approach, however, will lead to extra complexities and storage costs compared to the case where the same realization is employed for the entire procedure. In this work, we shall consider algorithms where $\bTheta$ is generated only once. When it is known that the set $\bigcup_{\mu \in \mathcal{P}_{\mathrm{train}}} Y_r(\mu)$ belongs to a subspace $Y^*_m$ of moderate dimension (e.g., when we operate on a small training set), then $\bTheta$ can be chosen such that it is a $U \to \ell_2$ $\varepsilon$-subspace embedding for $Y^*_m$ with high probability. Otherwise, care must be taken when characterizing the probability of success because of the adaptive nature of the greedy algorithm. {In such cases, all possible outcomes for $U_r$ should be considered by using a union bound for the probability of success.}

\begin{proposition}\label{thm:sk_greedy_finite}
 Let $U_r \subseteq U $ be a subspace obtained with $r$ iterations of the greedy algorithm with error indicator depending on $\bTheta$. If $\bTheta$ is a $(\varepsilon, \allowbreak m^{-1}\binom{m}{r}^{-1}\delta, \allowbreak 2 r+ 1)$ oblivious $U \to \ell_2$ subspace embedding, then it is a $U \to \ell_2$ $\varepsilon$-subspace embedding for $Y_r(\mu)$ defined in~(\ref {eq:Y_r}), for all $\mu \in \mathcal{P}_{\mathrm{train}}$, with probability at least $1- \delta$.
\begin{proof}
See appendix.
 \end{proof} \end{proposition}

\begin{remark}
 Theoretical bounds for the number of rows needed to construct $( \varepsilon, \allowbreak m^{-1}  \binom{m}{r}^{-1} \delta, \allowbreak 2 r+ 1)$ oblivious $U \to \ell_2$ subspace embeddings using Gaussian, Rademacher or P-SRHT distributions can be obtained from Propositions\nobreakspace \ref {thm:Rademacher},  \ref {thm:P-SRHT} and\nobreakspace  \ref {thm:buildepsilon_embedding}. For Gaussian or Rademacher matrices they are proportional to $r$, while for  P-SRHT they are proportional to $r^2$. In practice, however, embeddings built with P-SRHT, Gaussian or Rademacher distributions perform equally well. 
\end{remark}

%\begin{remark}
% If necessary, the cost of evaluating inner products can be further reduced by letting $\bTheta:= \bGamma\bTheta_0$, where $\bTheta_0 \in \mathbb{K}^{k_0 \times n}$ is a $U \to \ell_2$ $\varepsilon$-subspace embedding for $Y_i(\mu)$ for all $\mu \in \mathcal{P}_{\mathrm{train}}$ with high probability, and $\bGamma \in \mathbb{K}^{k\times k_0}$ with $k< k_0$, is an $(\varepsilon,  \delta, 2 i + 1)$ oblivious $\ell_2 \to \ell_2$ subspace embedding. We can use the same realization of $\bTheta_0$ for the entire algorithm, but generate a new realization of $\bGamma$ at each iteration. Observe that the sketch associated with $\bTheta_0$ can be precomputed and used for efficient evaluation of the sketch associated with each realization of $\bTheta$.
%\end{remark}

Evaluating $\| \br(\bu_r(\mu);\mu) \|^{\bTheta}_{V'}$ for very large training sets can be much more expensive than other costs. The complexity of this step can be reduced using the procedure explained in~Section\nobreakspace \ref {efficient_res_norm}. The {efficient sketched greedy algorithm} is summarized in Algorithm\nobreakspace \ref {alg:sk_greedy_online}. 
\begin{algorithm} \caption{efficient sketched greedy algorithm} \label{alg:sk_greedy_online}
 \begin{algorithmic}
  \STATE{\textbf{Given:} $\mathcal{P}_{\mathrm{train}}$, $\bA(\mu)$, $\bb(\mu)$, $\bl(\mu)$, $\bTheta$, $\tau$.}
  \STATE{\textbf{Output}: $U_{r}$}
  \STATE{1. Set $i:=0$, $U_{0}= \{ \bnull \}$, and pick $\mu^{1} \in \mathcal{P}_{\mathrm{train}}$.}
  \WHILE{$\underset{\mu \in \mathcal{P}_{\mathrm{train}}}{\max} {\widetilde{\Delta}(U_i;\mu)}\geq \tau$}
   \STATE{2. Set $i:=i+1$.}
   \STATE{3. Evaluate $\bu(\mu^{i})$ and set $U_{i}:= U_{i-1}+\mathrm{span}(\bu(\mu^{i}))$.}
   \STATE{4. Update  affine factors of $\bA_i(\mu)$, $\bb_i(\mu)$, {$\bV^{\bTheta}_i(\mu)$} and $\bb^{\bTheta}(\mu)$.}
   \STATE{5. Generate $\bGamma$ and evaluate affine factors of {$\bV^{\bPhi}_i(\mu)$} and $\bb^{\bPhi}(\mu)$.}
   \STATE{6. Set $\widetilde{\Delta}(U_i;\mu):= \Delta^{\bPhi}(\bu_i(\mu);\mu)$.}
   \STATE{7. Use~(\ref {eq:eff_comp_res}) to find $\mu^{i+1}:=\underset{\mu \in \mathcal{P}_{\mathrm{train}}}{\mathrm{argmax}~}{\widetilde{\Delta}(U_i;\mu)}$.}
  \ENDWHILE
 \end{algorithmic}
\end{algorithm}
From Propositions\nobreakspace \ref {thm:SKquasi-opt} and\nobreakspace  \ref {thm:skcea}, Corollary\nobreakspace \ref {thm:skerrorind1} and (\ref {eq:res_concentration}), we can prove the quasi-optimality of the greedy selection in Algorithm\nobreakspace \ref {alg:sk_greedy_online} with high probability.

\subsection{Proper Orthogonal Decomposition} \label{sk_pod} 

Now we introduce the sketched version of POD. We first note that random sketching is a popular technique for obtaining low-rank approximations of large matrices~\cite{woodruff2014sketching}. It can be easily combined with Proposition\nobreakspace \ref {thm:approx_pod} and\nobreakspace Algorithm\nobreakspace \ref {alg:approx_pod} for finding POD vectors. For large-scale problems, however, evaluating and storing POD vectors can be too expensive or even unfeasible, e.g., in a streaming or a distributed environment. We here propose a POD where evaluation of the full vectors is not necessary. We give a special attention to  distributed computing. The computations involved in our version of POD can be distributed among separate machines with a communication cost independent of the dimension of the full order problem.

We observe that a complete reduced order model can be constructed from a sketch (see Section\nobreakspace \ref {l2embeddingsMOR}). Assume that we are given the sketch of a matrix $\bU_m$ containing $m$ solutions samples associated with $\bTheta$, i.e.,
\begin{equation*}
 \bU^{\bTheta}_m:= \bTheta\bU_m,~~ \bV^{\bTheta}_m(\mu):= \bTheta\bR_U^{-1}\bA(\mu)\bU_m,~~ \bl_m(\mu)^\mathrm{H}:= \bl(\mu)^{\mathrm{H}}\bU_m,~~ \bb^{\bTheta}(\mu):= \bTheta\bR_U^{-1}\bb(\mu).
\end{equation*}
Recall that sketching a set of vectors can be efficiently performed basically in any modern computational environment, e.g., a distributed environment with expensive communication cost (see~Section\nobreakspace \ref {sketch}). Instead of computing a full matrix of reduced basis vectors, $\bU_r \in \mathbb{K}^{n \times r}$,  as in classical methods, we look for a small matrix $\bT_r \in \mathbb{K}^{m\times r}$ such that $\bU_r= \bU_m\bT_r$. Given $\bT_r$, the sketch of $\bU_r$ can be {computed} without operating with the whole $\bU_m$ but only with its sketch: 
\begin{equation*}
\bTheta \bU_r= \bU^{\bTheta}_m\bT_r,~~ \bTheta \bR_U^{-1}\bA(\mu)\bU_r= \bV^{\bTheta}_m(\mu)\bT_r, \textup{ and } \bl(\mu)^{\mathrm{H}}\bU_r= \bl_m(\mu)^{\mathrm{H}}\bT_r.
\end{equation*}
Further we propose an efficient way for obtaining $\bT_r$ such that the quality of $U_r:= \mathrm{span}(\bU_r)$ is close to optimal. 

For each $r\leq \mathrm{rank}(\bU^\bTheta_m)$, let $U_r$ be an $r$-dimensional subspace obtained with the method of snapshots associated with norm $\| \cdot \|^{\bTheta}_U $, presented below. 
\begin{definition}[Sketched method of snapshots]\label{thm:sk_pod}
	Consider the following eigenvalue problem
	\begin{equation} \label{eq:sk_eig_gram}
	\bG \bt = \lambda \bt 
	\end{equation}
	where  $\bG:= (\bU^\bTheta_m)^\mathrm{H}\bU^\bTheta_m$. Let $l= \mathrm{rank}(\bU^\bTheta_m)\geq r$ and let $\{ (\lambda_i, \bt_i) \}_{i=1}^{l}$ be the solutions to~(\ref {eq:sk_eig_gram}) ordered such that $ \lambda_1  \ge \hdots \ge \lambda_l$. Define 
	\begin{equation} \label{eq:sk_podbasis}
	U_r:=\mathrm{range} (\bU_m\bT_r), 
	\end{equation}
	where $\bT_r:= [\bt_1, ..., \bt_r]$.
\end{definition}

For given $V \subseteq U_m$, let $\bP^\bTheta_{V}:U_m \rightarrow V$ denote an orthogonal projection on $V$ with respect to $\| \cdot \|^\bTheta_{U}$, i.e.,
\begin{equation} \label{eq:P*}
\forall \bx \in U_m,~ \bP^\bTheta_{V}\bx= \arg\min_{\bw \in V} \| \bx-\bw \|^\bTheta_{U},
\end{equation}
and define the following error indicator:
\begin{equation}\label{eq:sk_poderror}
\Delta^\mathrm{POD}(V) := \frac{1}{m} \sum^{m}_{i=1} \left (\| \bu (\mu^i)- \bP^\bTheta_{V}\bu(\mu^i)\|^{\bTheta}_{U} \right )^2.
\end{equation} 

\begin{proposition} \label{thm:sk_poderror}
	Let $\{ \lambda_i \}_{i=1}^{l}$ be the set of eigenvalues from Definition\nobreakspace\ref{thm:sk_pod}. Then 
	\begin{equation} \label{eq:lambda}
	\Delta^\mathrm{POD}(U_r):=  \frac{1}{m} \sum^{l}_{i=r+1} \lambda_i.
	\end{equation}
	Moreover, for all $V_r \subseteq U_m$ with $\mathrm{dim}(V_r)\leq r$,
	\begin{equation}
	\Delta^\mathrm{POD}(U_r)\leq \Delta^\mathrm{POD}(V_r).
	\end{equation}
\begin{proof}
See appendix.
 \end{proof} \end{proposition}

Observe that {{the} matrix $\bT_r$ (characterizing $U_r$) can be much cheaper to obtain than the basis vectors for $U^*_r = POD_r(\bU_m, \| \cdot \|_U )$.}  For this, we need to operate only with the sketched matrix $\bU^\bTheta_m$ but not with the full snapshot matrix $\bU_m$. Nevertheless, the quality of $U_r$ can be guaranteed to be close {to the quality of $U^*_r$.} 

\begin{theorem} \label{thm:sk_podopt}
	Let $Y \subseteq U_m$ be a subspace of $U_m$ with $\mathrm{dim}(Y)\geq r$, and let 
	\begin{equation*}
	\Delta_Y= \frac{1}{m} \sum^{m}_{i=1} \| \bu(\mu^i)- \bP_{Y}\bu(\mu^i) \|^2_{U}.
	\end{equation*}
	If $\bTheta$ is a $U \to \ell_2$ $\varepsilon$-subspace embedding for $Y$ and every subspace in $\left \{ \mathrm{span} (\bu(\mu^i)- \bP_{Y}\bu(\mu^i)) \right \}_{i=1}^{m}$ and $\left \{ \mathrm{span}(\bu(\mu^i)- {\bP_{U^*_r}}\bu(\mu^i)) \right \}_{i=1}^{m}$, then 
	\begin{equation} \label{eq:sk_podoptY}
	\begin{split}
	\frac{1}{m} \sum^{m}_{i=1} \| \bu(\mu^i)- \bP_{U_r}\bu(\mu^i) \|&_{U}^2 \leq \frac{2}{1-\varepsilon}\Delta^\mathrm{POD}(U_r)+ (\frac{2(1+\varepsilon)}{1-\varepsilon}+1)\Delta_Y \\
	&\leq \frac{2(1+\varepsilon)}{1-\varepsilon} \frac{1}{m} \sum^{m}_{i=1} \| \bu(\mu^i)- \bP_{U^*_r}\bu(\mu^i) \|^2_{U}+ (\frac{2(1+\varepsilon)}{1-\varepsilon}+1)\Delta_Y.
	\end{split}
	\end{equation}
	Moreover, if $\bTheta$ is $U \to \ell_2$ $\varepsilon$-subspace embedding for $U_{m}$, then
	\begin{equation} \label{eq:sk_podoptUm}
	\begin{split}
	\frac{1}{m}\sum^{m}_{i=1} \| \bu(\mu^i)- \bP_{U_r}\bu(\mu^i) \|_{U}^2 &\leq \frac{1}{1-\varepsilon}\Delta^\mathrm{POD}(U_r)\leq \frac{1+\varepsilon}{1-\varepsilon} \frac{1}{m}\sum^{m}_{i=1} \| \bu(\mu^i) - \bP_{U^*_r} \bu(\mu^i) \|_{U}^2.
	\end{split}
	\end{equation}
	\begin{proof}
		See appendix. 
	\end{proof}
\end{theorem}
{By an union bound argument and the definition of an oblivious embedding, the hypothesis in the first part of Theorem\nobreakspace \ref {thm:sk_podopt} can be satisfied with probability at least $1-3\delta$ if $\bTheta$ is a $( \varepsilon, \delta, \mathrm{dim}(Y))$ and $( \varepsilon, \delta/m, 1)$  oblivious $U \to \ell_2$ embedding.} A subspace $Y$ can be taken as $U^*_{r}$, or a larger subspace making $\Delta_Y$ as small as possible. It is important to note that even if $U_r$ is quasi-optimal, there is no guarantee that $\bTheta$ is a $U \to \ell_2$ $\varepsilon$-subspace embedding for $U_r$ unless it is a $U \to \ell_2$ $\varepsilon$-subspace embedding for the whole $U_m$. Such guarantee can be unfeasible to achieve for large training sets. One possible solution is to maintain two sketches of $\bU_m$: one for {the method of snapshots}, and one for Galerkin projections and residual norms. Another way (following considerations similar to~\cite{halko2011finding}) is to replace $\bU_m$ by its low-rank approximation $\widetilde{\bU}_m = \bP^\bTheta_{W}\bU_m$, with $W = \mathrm{span}(\bU_m \bOmega^*)$, where $\bOmega^*$ is a small random matrix (e.g., Gaussian matrix). The latter procedure can be also used for improving the efficiency of the algorithm when $m$ is large. Finally, if $\bTheta$ is a $U \to \ell_2$ $\varepsilon$-subspace embedding for every subspace in $\{ \mathrm{span} (\bu(\mu^i)- \bP^\bTheta_{U_r}\bu(\mu^i)) \}_{i=1}^{m}$ then the error indicator $\Delta^\mathrm{POD}(U_r)$ is quasi-optimal. However, if only the first hypothesis of Theorem\nobreakspace \ref{thm:sk_podopt} is satisfied then the quality of $\Delta^\mathrm{POD}(U_r)$ will depend on $\Delta_Y$. In such a case the error can be {certified} using $\Delta^{\mathrm{POD}}(\cdot)$ defined with a new realization of $\bTheta$.

\section{Numerical examples} \label{Numerical}
In this section the approach is validated numerically
and compared against classical methods. For simplicity in all our experiments, we chose a coefficient  $\eta(\mu)=1$ in~Equations\nobreakspace \textup {(\ref {eq:errorind})} and\nobreakspace  \textup {(\ref {eq:skerrorind})} for the error estimation. The experiments revealed that the theoretical bounds for $k$ in~Propositions\nobreakspace \ref {thm:Rademacher} and\nobreakspace  \ref {thm:P-SRHT} and\nobreakspace Table\nobreakspace \ref {tab:numrows} are pessimistic.
In practice, much smaller random matrices still provide good estimation of the output. In addition, we did not detect any significant difference in performance between Rademacher matrices, Gaussian matrices and P-SRHT, even though the theoretical bounds for P-SRHT are worse.
Finally, the results obtained with Rademacher matrices are not presented. They are similar to those for Gaussian matrices and P-SRHT. 

\subsection{3D thermal block}
We use a 3D version of the thermal block benchmark from~\cite{haasdonk2017reduced}.  This problem describes a  heat transfer phenomenon through a domain $\Omega:= [0,1]^3$  made of an assembly of  blocks, each composed of a different material. The boundary value problem for modeling the thermal block is as follows 
\begin{equation} \label{eq:BVP1}
 \left \{
 \begin{array}{rll}
  -\boldsymbol{\nabla} \cdot (\kappa \boldsymbol{\nabla} T)&= 0,~~ & \textup{in }  \Omega \\
  T &=0,~~ & \textup{on }  \Gamma_{D} \\
  \boldsymbol{n} \cdot (\kappa \boldsymbol{\nabla} T)&=0,~~ &  \textup{on }  \Gamma_{N,1}\\
  \boldsymbol{n} \cdot (\kappa \boldsymbol{\nabla} T)&=1,~~ & \textup{on } \Gamma_{N,2},
 \end{array}
 \right.
\end{equation}
where $T$ is the temperature field, $\boldsymbol{n}$ is the outward normal vector to the boundary, $\kappa$ is the thermal conductivity, and $\Gamma_{D}$, $\Gamma_{N,1}$, $\Gamma_{N,2}$ are parts of the boundary defined by $\Gamma_{D} := \{(x,y,z) \in \partial\Omega: y=1\}$, $\Gamma_{N,2} := \{(x,y,z) \in \partial\Omega: y=0\}$ and $\Gamma_{N,1} :=\partial\Omega \backslash (\Gamma_D \cup \Gamma_{N,2})$. $\Omega$ is partitioned into $2\times 2\times 2$ subblocks $\Omega_i$ of equal size. A different thermal conductivity $\kappa_i$ is assigned to each  $\Omega_i$, i.e., 
$ \kappa(x)= \kappa_i$, $x \in \Omega_i.$ 
We are interested in estimating the mean temperature in $\Omega_1 := [0, \frac{1}{2}]^3$ for each $\mu:= (\kappa_1,..., \kappa_{8}) \in \mathcal{P} := [\frac{1}{10}, 10]^{8}$. The $\kappa_i$ are independent random variables with log-uniform distribution over $ [\frac{1}{10}, 10]$. 

Problem~(\ref {eq:BVP1}) was discretized using the classical finite element method with approximately $n=120000$ degrees of freedom. A function $w$ in the finite element approximation space is identified with a vector $\bw \in U$. The space $U$ is equipped with an inner product compatible with the 
$H^1_0$ inner product, i.e., $\| \bw \|_{U}:= \| \boldsymbol{\nabla}w \|_{L_2}$. The training set $\mathcal{P}_{\mathrm{train}}$ and the test set $\mathcal{P}_{\mathrm{test}}$ were taken as $10000$ and $1000$ independent samples{,} respectively. The factorization of $\bR_U$ was precomputed only once and used for efficient multiplication of $\bR^{-1}_U$  by multiple vectors. The sketching matrix $\bTheta$ was constructed with ~Proposition\nobreakspace \ref {thm:buildepsilon_embedding}, i.e., $\bTheta:= \bOmega\bQ$, where $\bOmega \in \mathbb{R}^{k\times s}$ is a classical oblivious $\ell_2 \to \ell_2$ subspace embedding and $\bQ \in \mathbb{R}^{s\times n}$ is such that $\bQ^\mathrm{T}\bQ=\bR_U$. Furthermore, $\bQ$ was taken as the transposed Cholesky factor of $\bR_U$. Different distributions and sizes of {the} matrix $\bOmega$ were considered. The same realizations of $\bOmega$ were used for all parameters and greedy iterations within each experiment. A seeded random number generator was used for memory-efficient operations on random matrices. For P-SRHT, a fast implementation of the {fast} Walsh-Hadamard transform was employed for multiplying the Walsh-Hadamard matrix by a vector in $s \log_2{(s)}$ time. In Algorithm\nobreakspace \ref {alg:sk_greedy_online}, we used  $\bPhi:= \bGamma\bTheta$, where $\bGamma \in \mathbb{R}^{k'\times k}$ is a Gaussian matrix and $k'=100$. The same realizations of $\bGamma$ were used for all the parameters but it was regenerated at each greedy iteration. 

\emph{Galerkin projection and primal-dual correction.}
Let us investigate how the quality of the solution depends on the distribution and size of $\bOmega$. We first generated sufficiently accurate reduced subspaces $U_r$ and $U^{\mathrm{du}}_r$ for the primal and the  dual problems. The subspaces were spanned by snapshots evaluated at some points in $\mathcal{P}_{\mathrm{train}}$. The interpolation points were obtained by $r=100$ iterations of the  {efficient sketched greedy algorithm} (Algorithm\nobreakspace \ref {alg:sk_greedy_online}) with P-SRHT and $k =1000$ rows. Thereafter, $\bu(\mu)$ was approximated by a projection $\bu_r(\mu) \in U_r$. The classical Galerkin projection~(\ref {eq:galproj}) and its sketched version~(\ref {eq:SKgalproj}) with different distributions and sizes of $\bOmega$ were considered. The quality of a parameter-dependent projection is measured by $e_\mathcal{P}:=\max_{\mu \in \mathcal{P}_{\mathrm{test}}} \|\bu(\mu) - \bu_r(\mu)\|_{U} / \max_{\mu \in \mathcal{P}_{\mathrm{test}}} \|\bu(\mu)\|_{U}$ and $\Delta_\mathcal{P}:=\max_{\mu \in \mathcal{P}_{\mathrm{test}}} \| \br(\bu_r(\mu); \mu) \|_{U'}/\max_{\mu \in \mathcal{P}_{\mathrm{test}}} \|\bb(\mu)\|_{U'}$.  For each random projection 20 samples of $e_\mathcal{P}$ and $\Delta_\mathcal{P}$ were evaluated. Figure\nobreakspace \ref {fig:Ex1_1} describes how $e_\mathcal{P}$ and $\Delta_\mathcal{P}$ depend on the number of rows $k$\footnote{{The $p$-quantile of a random variable $X$ is  defined as $\inf \{ t : \mathbb{P}(X \leq t) \geq p \}$ and can be estimated by replacing the cumulative distribution function $\mathbb{P}(X \leq t)$ by its empirical estimation. Here we use 20 samples for this estimation.}}. We observe that the error associated with the sketched Galerkin projection is large when $k$ is close to $r$, but as $k$ increases, it asymptotically approaches the error of the classical Galerkin projection. The residual errors of the classical and the sketched projections become almost identical already for $k=500$ while the exact errors become close for $k=1000$. We also observe that for the aforementioned $k$ there is practically no deviation of $\Delta_\mathcal{P}$ and only a little deviation of $e_\mathcal{P}$. 

{
	Note that the theoretical bounds for $k$ to preserve the quasi-optimality constants of the classical Galerkin projection can be derived using Propositions\nobreakspace \ref{thm:Rademacher} and \ref{thm:P-SRHT} combined with Proposition\nobreakspace \ref{thm:SKquasi-opt} and a union bound for the probability of success. As was noted in Section\nobreakspace\ref{obleddings}, however, the theoretical bounds for $k$ in Propositions\nobreakspace \ref{thm:Rademacher} and \ref{thm:P-SRHT} shall be useful only for large problems with, say $n/r>10^4$, which means they should not be applicable here. Indeed, we see that for ensuring that 
	$$\mathbb{P}(\forall \mu \in \mathcal{P}_{\mathrm{test}}: \varepsilon a_r(\mu) <1) >1-10^{-6}, $$
	using the theoretical bounds, we need impractical values $k \geq {39280}$ for Gaussian matrices and $k =n \approx 100000$ for P-SRHT. 	
	In practice, the value for $k$ can be determined using the adaptive procedure proposed in~\cite{balabanov2018}.}	
%. 
\begin{figure}[h!]
 \centering
 \begin{subfigure}[b]{.4\textwidth}
  \centering
  \includegraphics[width=\textwidth]{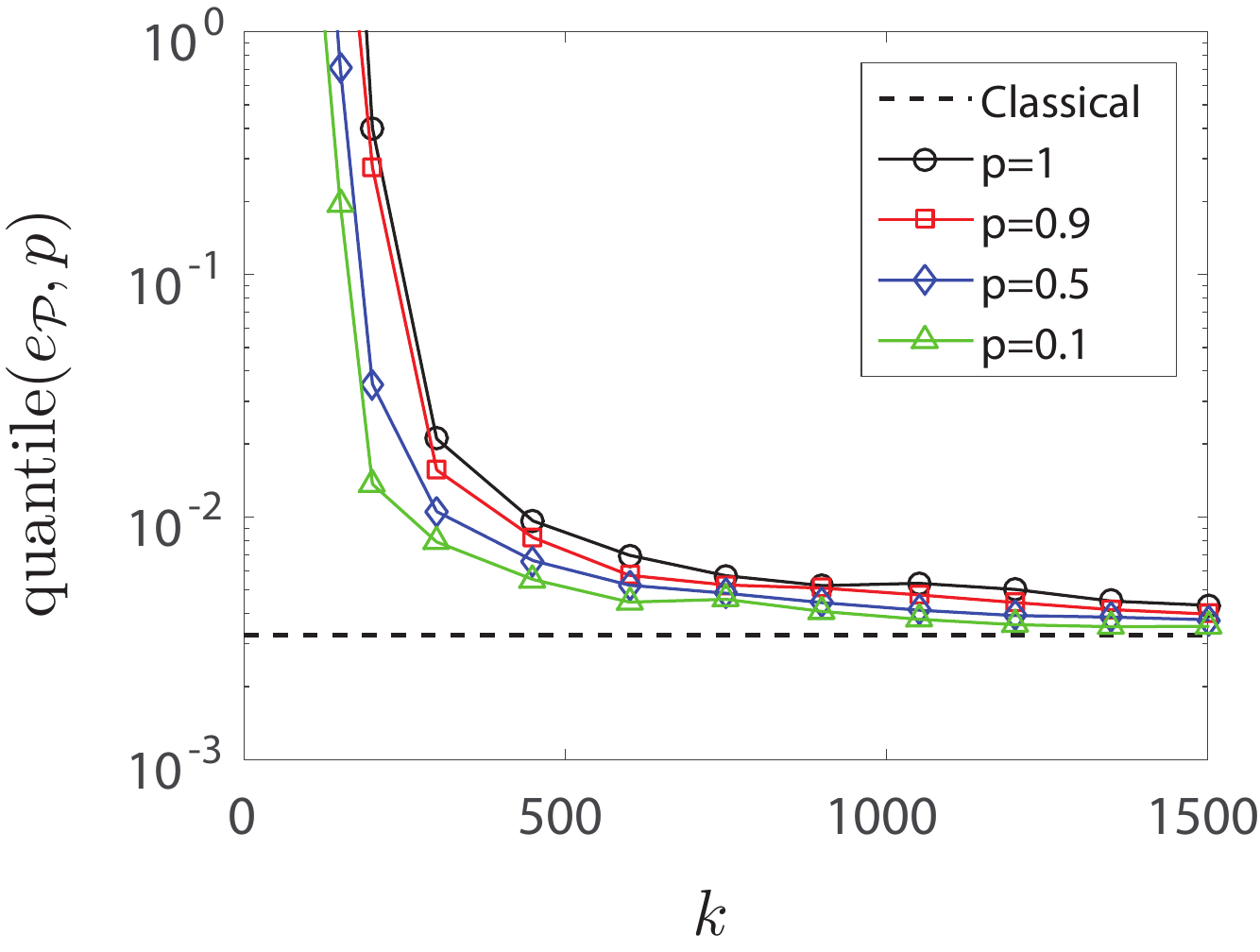}
  \caption{}
  \label{fig:Ex1_1a}
 \end{subfigure} \hspace{.01\textwidth}
 \begin{subfigure}[b]{.4\textwidth}
  \centering
  \includegraphics[width=\textwidth]{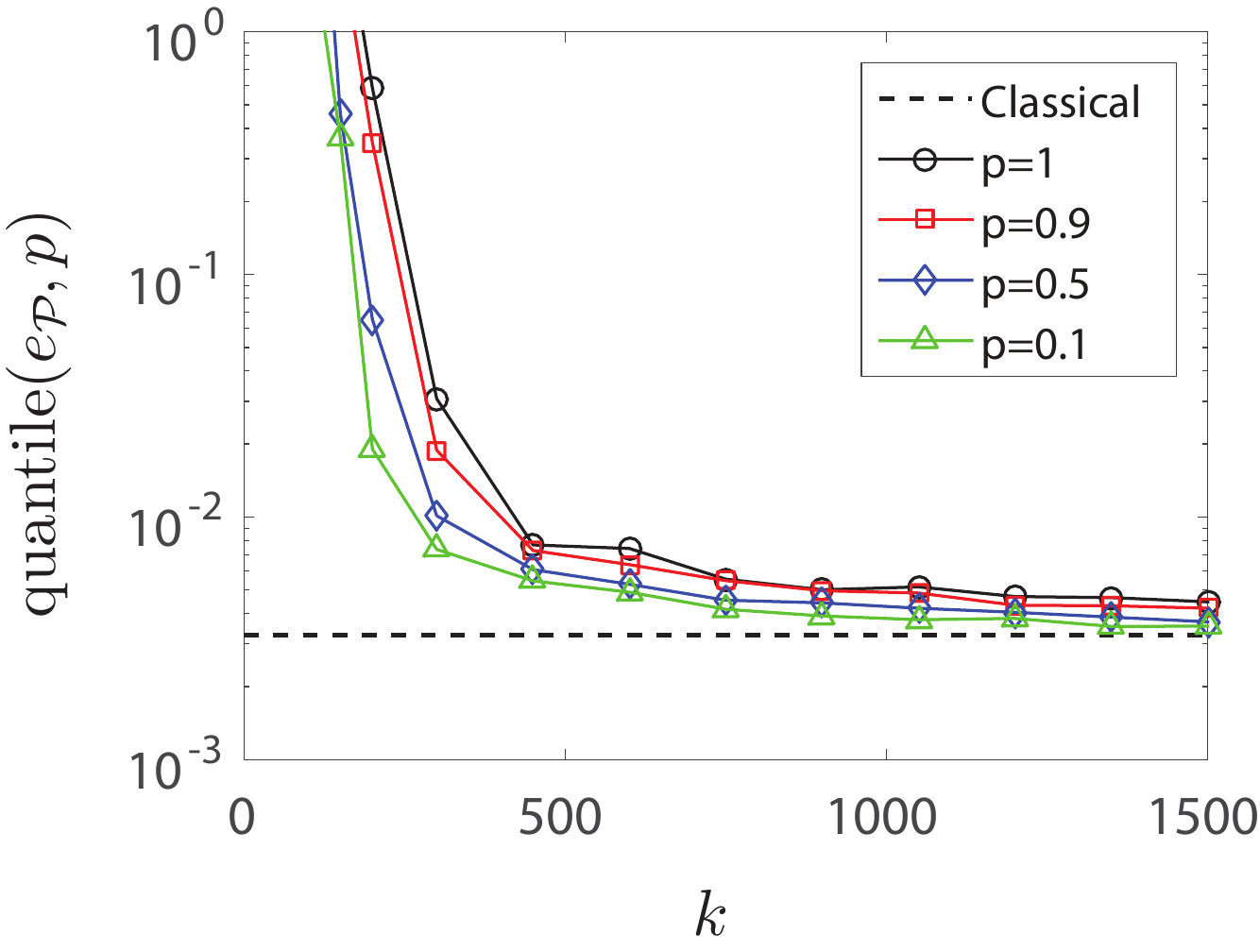}
  \caption{}
  \label{fig:Ex1_1b}
   \end{subfigure}
   
  \begin{subfigure}[b]{.4\textwidth}
  \centering
  \includegraphics[width=\textwidth]{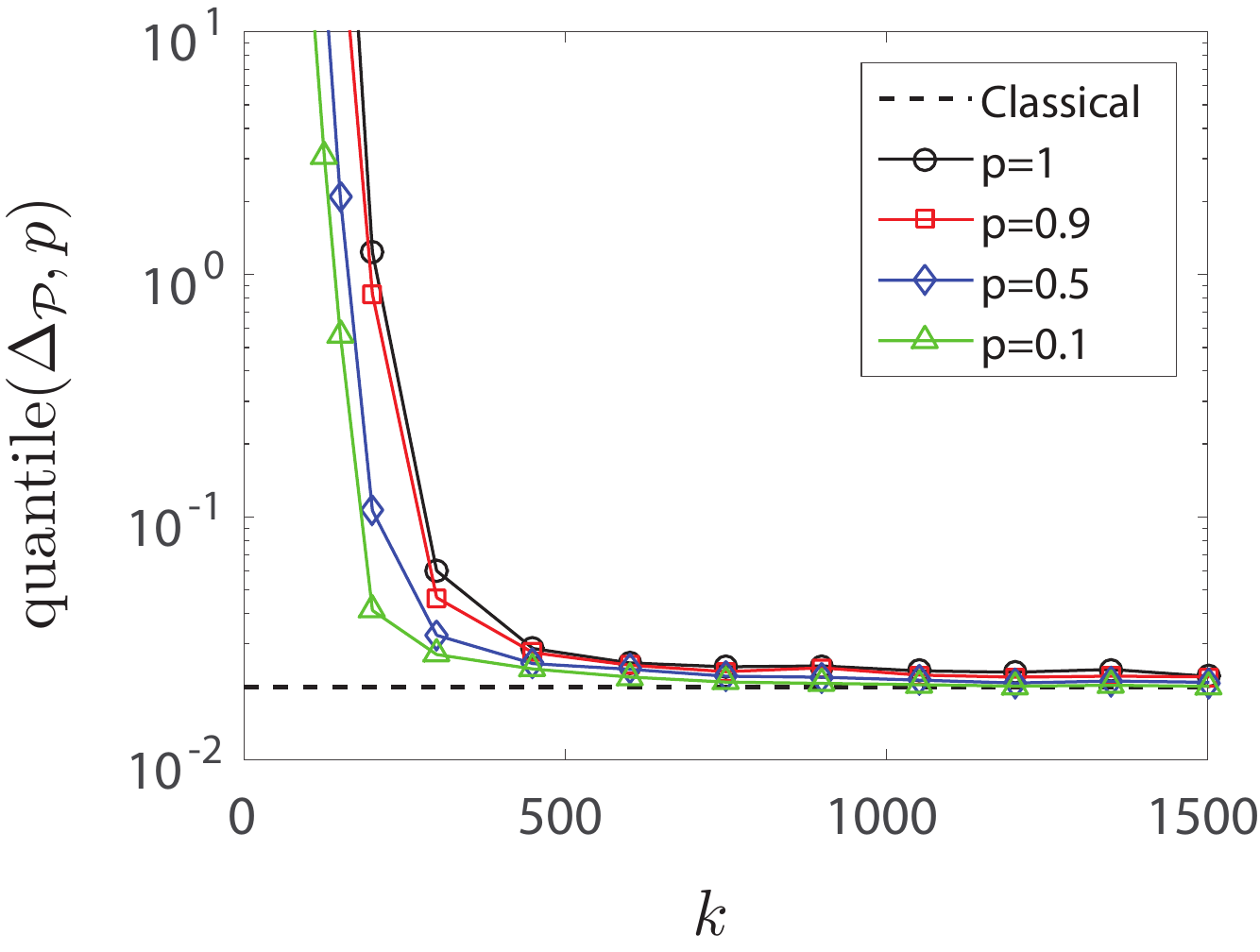}
  \caption{}
  \label{fig:Ex1_1c}
 \end{subfigure} \hspace{.01\textwidth}
 \begin{subfigure}[b]{.4\textwidth}
  \centering
  \includegraphics[width=\textwidth]{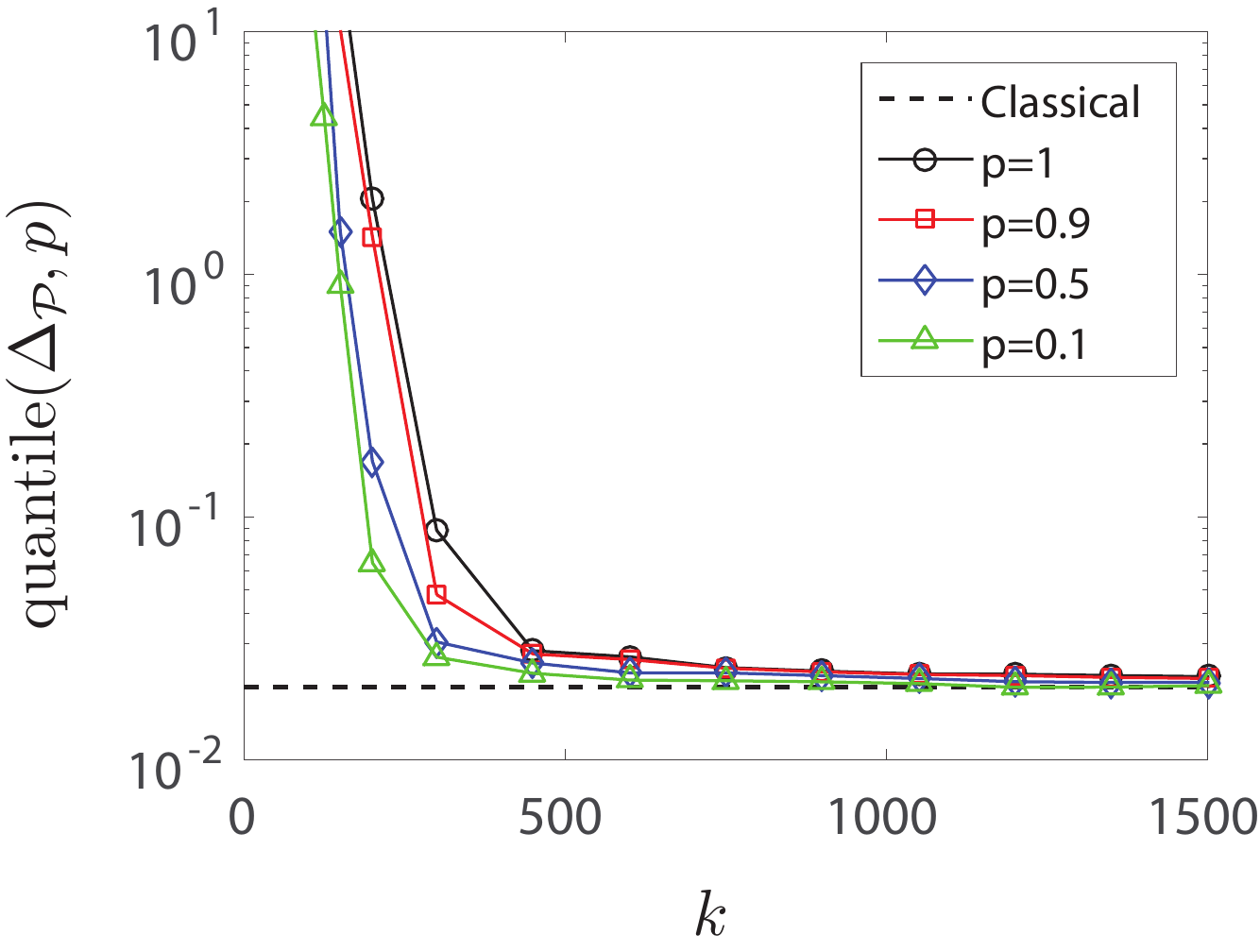}
  \caption{}
  \label{fig:Ex1_1d}
 \end{subfigure}
 \caption{Errors $e_\mathcal{P}$ and $\Delta_\mathcal{P}$ of the classical Galerkin projection and quantiles of probabilities $p=1, 0.9, 0.5$ and $0.1$ over 20 samples of $e_\mathcal{P}$ and $\Delta_\mathcal{P}$ of the randomized Galerkin projection versus the number of rows of $\bOmega$. (a) The exact error $e_\mathcal{P}$ with rescaled Gaussian distribution as $\bOmega$. (b) The exact error $e_\mathcal{P}$ with P-SRHT matrix as $\bOmega$. (c) The residual error $\Delta_\mathcal{P}$ with rescaled Gaussian distribution as $\bOmega$. (d) The residual error $\Delta_\mathcal{P}$ with P-SRHT matrix as $\bOmega$.}
 \label{fig:Ex1_1}
\end{figure}

Thereafter, we let $\bu_r(\mu) \in U_r$ and $\bu_r^\mathrm{du}(\mu) \in U^{\mathrm{du}}_r$ be the sketched Galerkin projections, where $\bOmega$ was taken as P-SRHT with $k =500$ rows. For the fixed $\bu_r(\mu)$ and $\bu_r^\mathrm{du}(\mu)$ the classical primal-dual correction $s^{\mathrm{pd}}_r(\mu)$~(\ref {eq:correction}), and the sketched primal-dual correction $s^{\mathrm{spd}}_r(\mu)$~(\ref {eq:skcorrection}) were evaluated using different sizes and distributions of $\bOmega$. In addition, the approach introduced in~Section\nobreakspace \ref {sk_pd_correction} for improving the accuracy of the sketched correction was employed. For $\bw^\mathrm{du}_r(\mu)$ we chose the orthogonal projection of $\bu_r^\mathrm{du}(\mu)$ on $W^\mathrm{du}_r:=U^{\mathrm{du}}_i$ with $i^\mathrm{du}=30$ (the subspace spanned by the first $i^\mathrm{du}=30$ basis vectors obtained during the generation of $U^{\mathrm{du}}_r$). With such $\bw_r^\mathrm{du}(\mu)$ the improved correction $s^{\mathrm{spd+}}_r(\mu)$ defined by~(\ref {eq:skcorrection2}) was computed. It has to be mentioned that $s^{\mathrm{spd+}}_r(\mu)$ yielded additional computations. They, however, are cheaper than the computations required for constructing the classical reduced systems and evaluating the classical output quantities in about $10$ times in terms of complexity and $6.67$ times in terms of memory. We define the error by $d_\mathcal{P}:=\max_{\mu \in \mathcal{P}_{\mathrm{test}}} |s(\mu) - \widetilde{s}_r(\mu) | / \max_{\mu \in \mathcal{P}_{\mathrm{test}}} |s(\mu)|$, where $\widetilde{s}_r(\mu)=s^{\mathrm{pd}}_r(\mu), s^{\mathrm{spd}}_r(\mu)$ or $s^{\mathrm{spd+}}_r(\mu)$. For each random correction we computed $20$ samples of $d_\mathcal{P}$. The errors on the output quantities versus the numbers of rows of $\bTheta$ are presented in~Figure\nobreakspace \ref {fig:Ex1_2}. We see that the error of $s^{\mathrm{spd}}_r(\mu)$ is proportional to $k^{-1/2}$. It can be explained by the fact that for considered sizes of random matrices, $\varepsilon$ is large compared to the residual error of the dual solution. As was noted in~Section\nobreakspace \ref {sk_pd_correction} in such a case the error bound for $s^{\mathrm{spd}}_r(\mu)$ is equal to $\mathcal{O}(\varepsilon \| \br(\bu_r(\mu);\mu)) \|_{U'})$. By~Propositions\nobreakspace \ref {thm:Rademacher} and\nobreakspace  \ref {thm:P-SRHT} it follows that $\varepsilon = \mathcal{O}(k^{-1/2})$, which explains the behavior of the error in~Figure\nobreakspace \ref {fig:Ex1_2}. Note that the convergence of $s^{\mathrm{spd}}_r(\mu)$ is not expected to be reached even for $k$ close to the dimension of the discrete problem. For large enough problems, however, the quality of the classical output will be always attained with $k \ll n$. In general, the error of the sketched primal-dual correction does not depend (or weakly depends for P-SRHT) on the dimension of the full order problem, but only on the accuracies of the approximate solutions $\bu_r(\mu)$ and $\bu_r^\mathrm{du}(\mu)$. 
On the other hand, we see that $s^{\mathrm{spd+}}_r(\mu)$ reaches the accuracy of the classical primal-dual correction for moderate $k$.   

\begin{figure}[h!]
 \centering
 \begin{subfigure}{.4\textwidth}
  \centering
  \includegraphics[width=\textwidth]{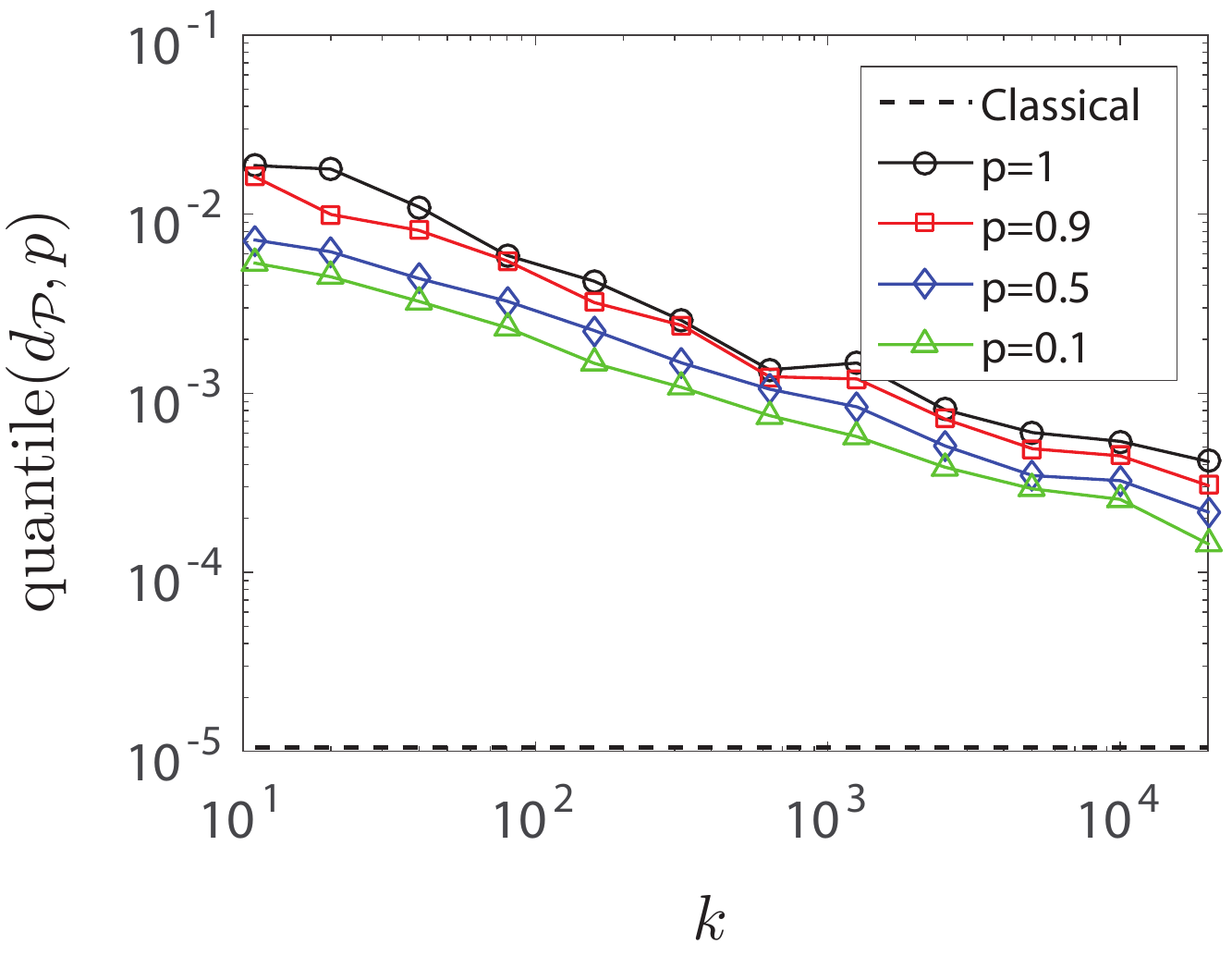}
  \caption{}
  \label{fig:Ex1_2a}
 \end{subfigure} \hspace{.01\textwidth}
 \begin{subfigure}{.4\textwidth}
  \centering
  \includegraphics[width=\textwidth]{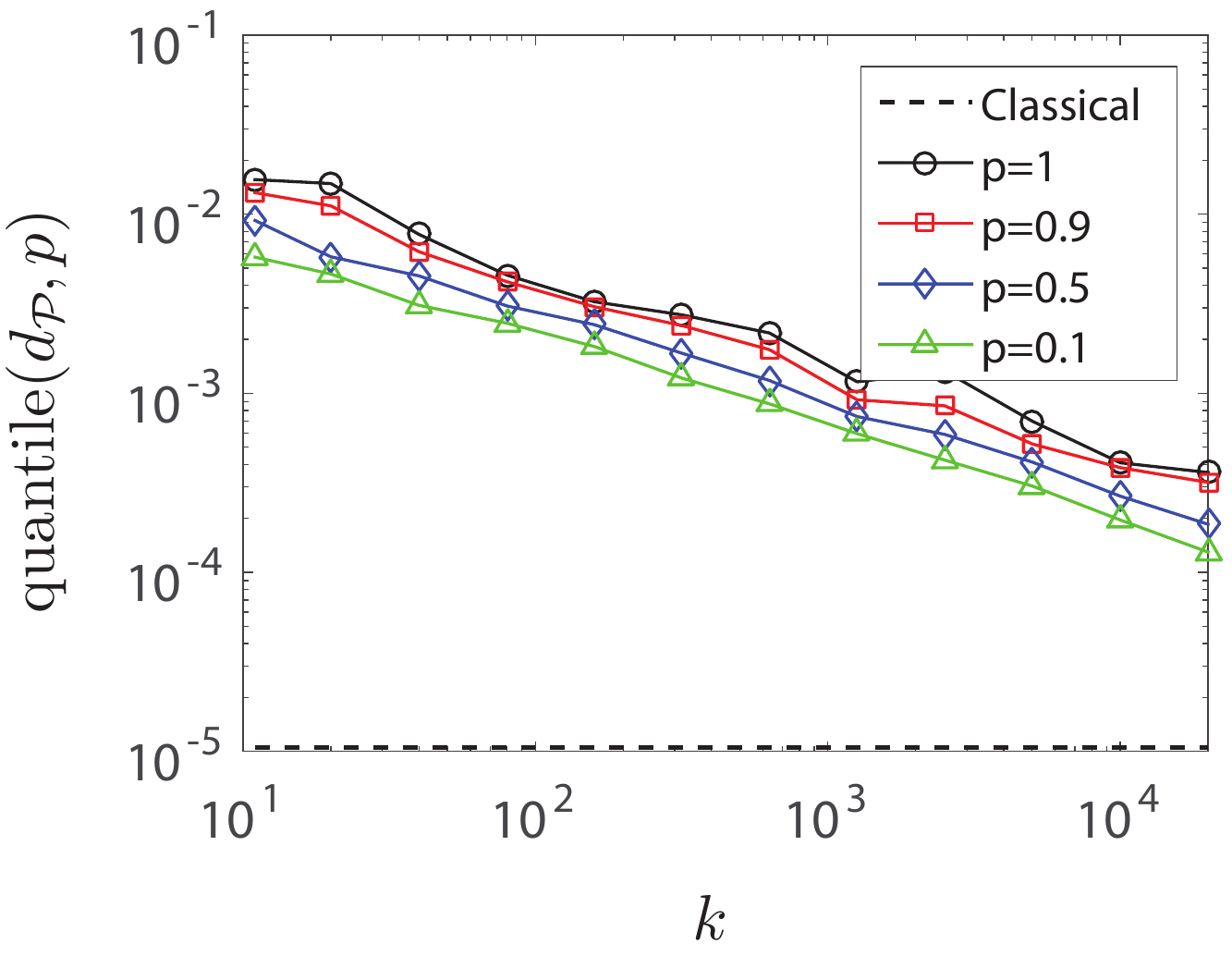}
  \caption{}
  \label{fig:Ex1_2b}
 \end{subfigure}
 \begin{subfigure}{.4\textwidth}
  \centering
  \includegraphics[width=\textwidth]{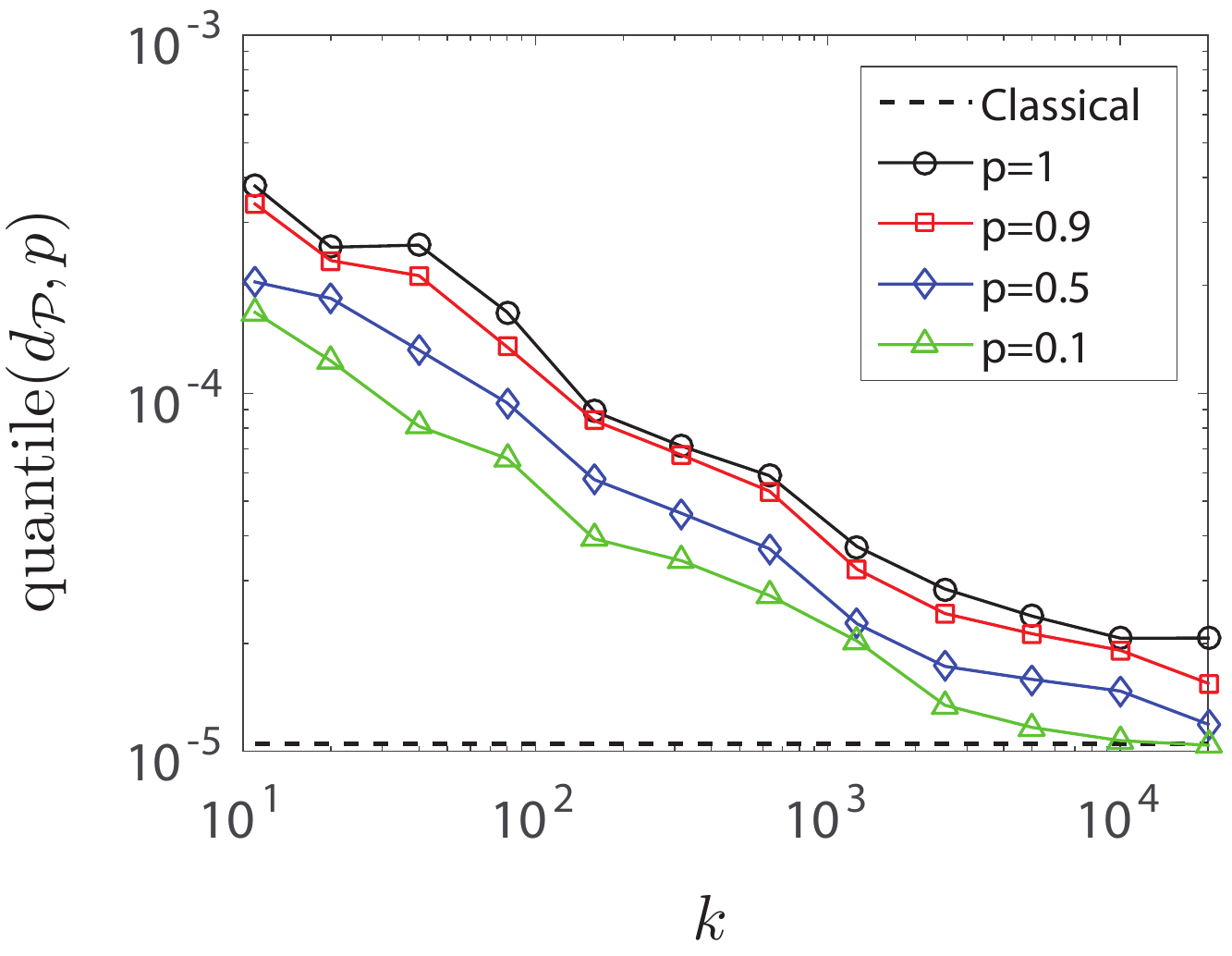}
  \caption{}
  \label{fig:Ex1_2c}
 \end{subfigure} \hspace{.01\textwidth}
 \begin{subfigure}{.4\textwidth}
  \centering
  \includegraphics[width=\textwidth]{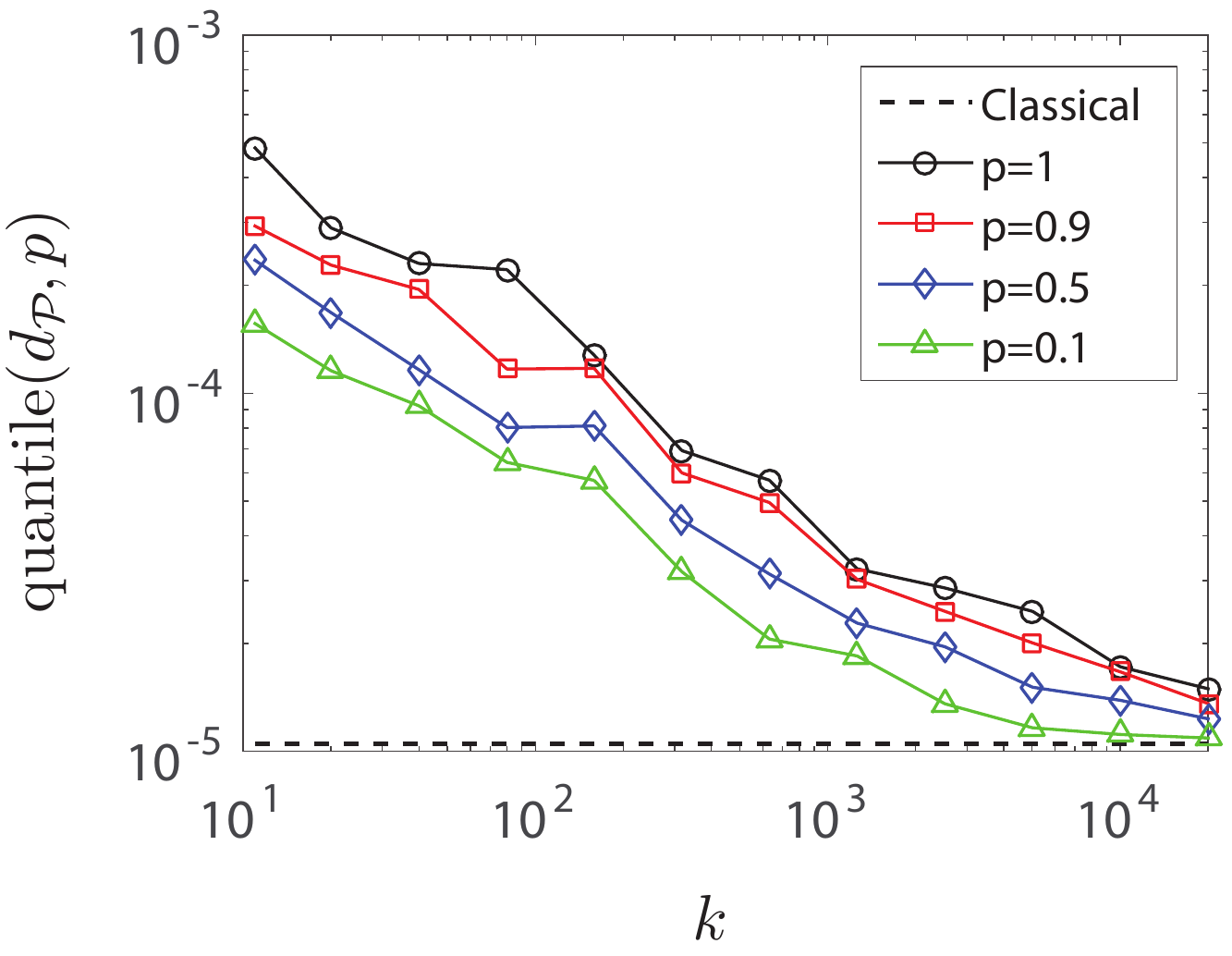}
  \caption{}
  \label{fig:Ex1_2d}
 \end{subfigure}
 \caption{The error $d_\mathcal{P}$ of the classical primal-dual correction and quantiles of probabilities $p=1, 0.9, 0.5$ and $0.1$ over 20 samples of $d_\mathcal{P}$ of the randomized primal-dual corrections with fixed $\bu_r(\mu)$ and $\bu_r^\mathrm{du}(\mu)$ versus the number of rows of $\bOmega$. (a) The errors of $s^{\mathrm{pd}}_r(\mu)$ and $s^{\mathrm{spd}}_r(\mu)$ with Gaussian matrix as $\bOmega$. (b) The errors of $s^{\mathrm{pd}}_r(\mu)$ and $s^{\mathrm{spd}}_r(\mu)$ with P-SRHT distribution as $\bOmega$. (c) The errors of $s^{\mathrm{pd}}_r(\mu)$ and $s^{\mathrm{spd+}}_r(\mu)$ with Gaussian matrix as $\bOmega$ and $W^\mathrm{du}_r:=U^{\mathrm{du}}_i$, $i^\mathrm{du}=30$.  (d) The errors of $s^{\mathrm{pd}}_r(\mu)$ and $s^{\mathrm{spd+}}_r(\mu)$ with P-SRHT distribution as $\bOmega$ and $W^\mathrm{du}_r:=U^{\mathrm{du}}_i$, $i^\mathrm{du}=30$.}
 \label{fig:Ex1_2}
\end{figure}

Further we focus only on the primal problem noting that similar results were observed also for the dual one.

\emph{Error estimation.} 
We let $U_r$ and $\bu_r(\mu)$ be the subspace and the approximate solution from the previous experiment. The classical error indicator $\Delta(\bu_r(\mu); \mu)$ and the sketched error indicator $\Delta^\bTheta (\bu_r(\mu); \mu)$ were evaluated for every $\mu \in \mathcal{P}_{\mathrm{test}}$. For $\Delta^\bTheta(\bu_r(\mu);\mu)$ different distributions and sizes of $\bOmega$ were considered. The quality of $\Delta^\bTheta(\bu_r(\mu);\mu)$ as estimator for $\Delta(\bu_r(\mu);\mu)$ can be characterized by $e^{\mathrm{ind}}_{\mathrm{\mathcal{P}}}:= \max_{\mu \in \mathcal{P}_{\mathrm{test}}} |\Delta(\bu_r(\mu);\mu)-\Delta^\bTheta(\bu_r(\mu);\mu)| / \max_{\mu \in \mathcal{P}_{\mathrm{test}}} \Delta(\bu_r(\mu);\mu)$. For each $\bOmega$, $20$ samples of $e^{\mathrm{ind}}_{\mathrm{\mathcal{P}}}$ were evaluated. Figure\nobreakspace \ref {fig:Ex1_3b} shows how $e^{\mathrm{ind}}_{\mathrm{\mathcal{P}}}$ depends on $k$. The convergence of the error is proportional to $k^{-1/2}$, similarly as for the primal-dual correction. In practice, however, $\Delta^\bTheta(\bu_r(\mu);\mu)$ does not have to be so accurate as the approximation of the quantity of interest. For many problems, estimating $\Delta(\bu_r(\mu);\mu)$ with relative error less than {1/2} is already good enough. Consequently, $\Delta^\bTheta(\bu_r(\mu);\mu)$ employing $\bOmega$ with $k=100$ or even $k=10$ rows can be readily used as a reliable error estimator. Note that $\mathcal{P}_{\mathrm{test}}$ and $U_r$ were formed independently of $\bOmega$. Otherwise, a larger $\bOmega$ should be considered with an additional embedding $\bGamma$ as explained in Section\nobreakspace \ref {efficient_res_norm}.

\begin{figure}[h!]
 \centering
 \begin{subfigure}{.4\textwidth}
  \centering
  \includegraphics[width=\textwidth]{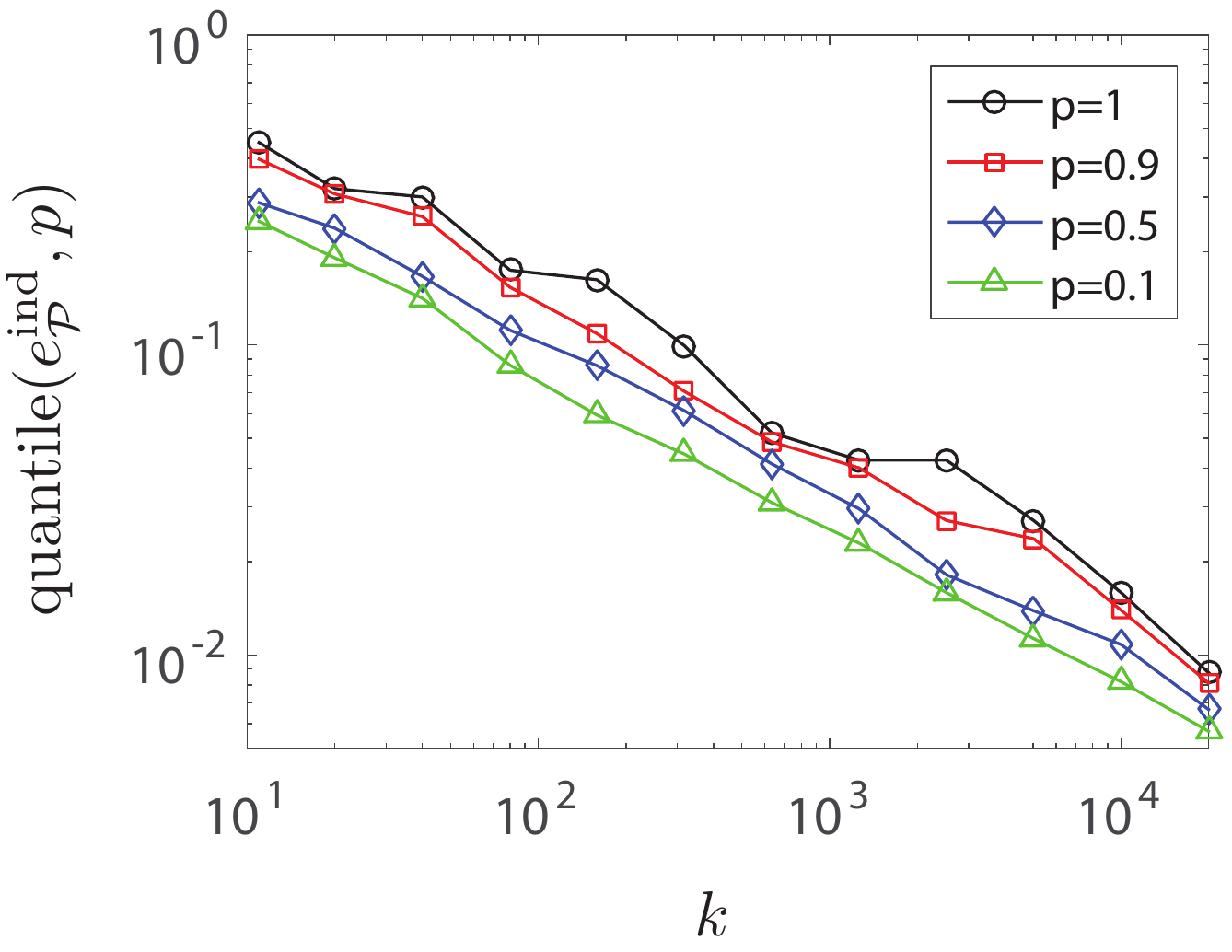}
  \caption{}
  \label{fig:Ex1_3a}
 \end{subfigure} \hspace{.01\textwidth}
 \begin{subfigure}{.4\textwidth}
  \centering
  \includegraphics[width=\textwidth]{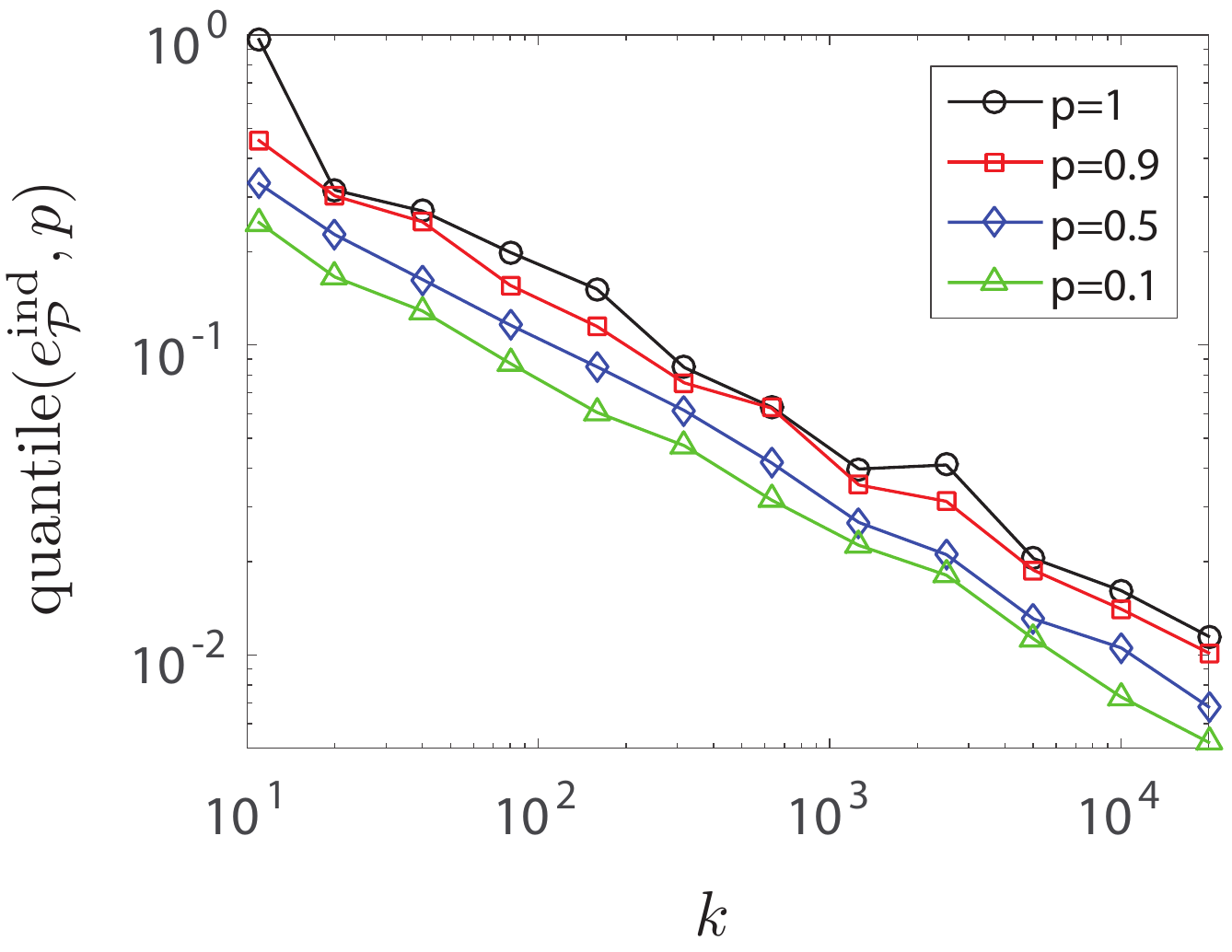}
  \caption{}
  \label{fig:Ex1_3b}
 \end{subfigure}
 \caption{Quantiles of probabilities $p=1, 0.9, 0.5$ and $0.1$ over 20 samples of the error $e^{\mathrm{ind}}_\mathcal{P}$ of $\Delta^\bTheta(\bu_r(\mu);\mu)$ as estimator of $\Delta(\bu_r(\mu);\mu)$. (a) The error of $\Delta^\bTheta(\bu_r(\mu);\mu)$ with Gaussian distribution. (b) The error of $\Delta^\bTheta(\bu_r(\mu);\mu)$ with P-SRHT distribution.}
\label{fig:Ex1_3}
\end{figure}

To validate the claim that our  approach (see~Section\nobreakspace \ref {efficient_res_norm}) for error estimation provides more numerical stability than the classical one, we performed the following experiment. For fixed $\mu \in \mathcal{P}$ such that $\bu(\mu) \in U_r$ we picked several vectors $\bu^*_i \in U_r$ at different distances of $\bu(\mu)$. For each such $\bu^*_i$ we evaluated  $\Delta(\bu_i^*;\mu)$ and $\Delta^\bTheta(\bu_i^*;\mu)$. The classical error indicator $\Delta(\bu_i^*;\mu)$ was evaluated using the traditional procedure, i.e., expressing $\| \br(\bu_i^*;\mu) \|^2_{U'}$ in the form~(\ref {eq:compute_res}), while $\Delta^\bTheta(\bu_i^*;\mu)$ was evaluated  with relation~(\ref {eq:skcompres}). The sketching matrix $\bOmega$ was generated from the P-SRHT or the rescaled Gaussian distribution with $k=100$ rows. Note that $\mu$ and $\bu_i^*$  were chosen independently of $\bOmega$ so there is no point to use larger $\bOmega$ with additional embedding $\bGamma$ (see~Section\nobreakspace \ref {efficient_res_norm}). Figure\nobreakspace \ref {fig:Ex1_4} clearly reveals the failure of the classical error indicator at $\Delta(\bu_i^*;\mu)/\| \bb(\mu) \|_{U'} \approx 10^{-7}$. On the contrary, the indicators computed with random sketching technique remain reliable even for $\Delta(\bu_i^*;\mu)/\| \bb(\mu) \|_{U'}$ close to the machine precision. 
\begin{figure}[h!]
 \centering
 \includegraphics[width=0.5\textwidth]{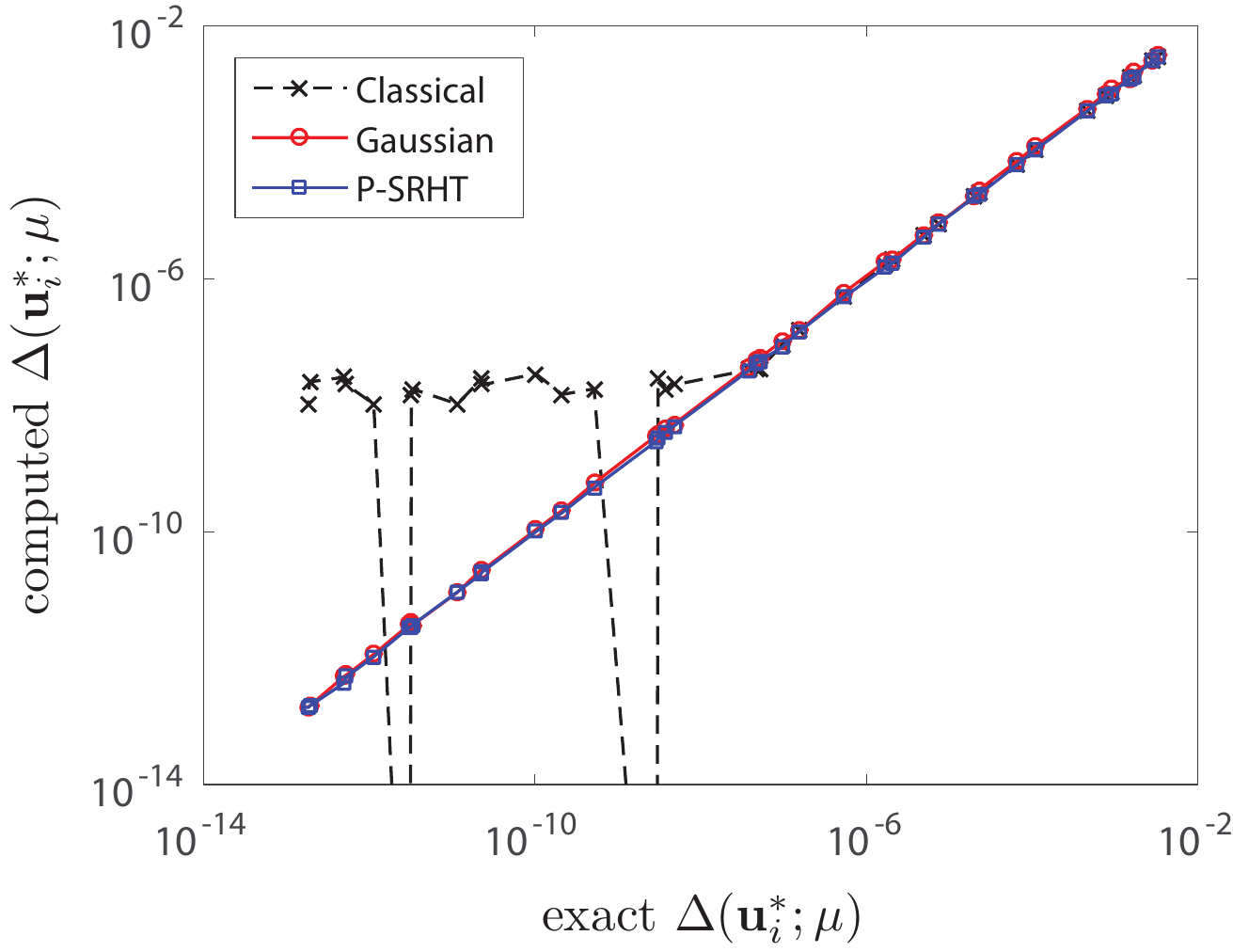}
 \caption{Error indicator $\Delta(\bu_i^*;\mu)$ (rescaled by $\| \bb(\mu) \|_{U'}$) computed with the classical procedure and its estimator $\Delta^\bTheta(\bu_i^*;\mu )$ computed with relation~(\ref {eq:skcompres}) employing P-SRHT or Gaussian distribution with $k=100$ rows versus the exact value of $\Delta(\bu_i^*;\mu)$ (rescaled by $\| \bb(\mu) \|_{U'}$).}
 \label{fig:Ex1_4}
\end{figure}

\emph{{Efficient sketched greedy algorithm}.} Further, we validate the performance of the  {efficient sketched greedy algorithm} (Algorithm\nobreakspace \ref {alg:sk_greedy_online}). For this we generated a subspace $U_r$ of dimension $r=100$ using the classical greedy algorithm (depicted in Section\nobreakspace \ref {Greedy}) and its randomized version (Algorithm\nobreakspace \ref {alg:sk_greedy_online}) employing $\bOmega$ of different types and sizes. In~Algorithm\nobreakspace \ref {alg:sk_greedy_online}, $\bGamma$ was generated from a Gaussian distribution with $k'=100$ rows. The error at $i$-th iteration is identified with $\Delta_\mathcal{P}:=\max_{\mu \in \mathcal{P}_{\mathrm{train}}} \| \br(\bu_i(\mu); \mu) \|_{U'}/\max_{\mu \in \mathcal{P}_{\mathrm{train}}} \|\bb(\mu)\|_{U'}$.  The {convergence} is depicted in Figure\nobreakspace \ref {fig:Ex1_6}. For the  {efficient sketched greedy algorithm} with $k=250$ and $k=500$ a slight difference in performance is detected compared to the classical algorithm. The difference is more evident for $k=250$ at higher iterations. The {behaviors} of~the classical algorithm and~Algorithm\nobreakspace \ref {alg:sk_greedy_online} with $k=1000$ are almost identical. 

\begin{figure}[h!] 
 \centering
 \begin{subfigure}{.4\textwidth}
  \centering
  \includegraphics[width=\textwidth]{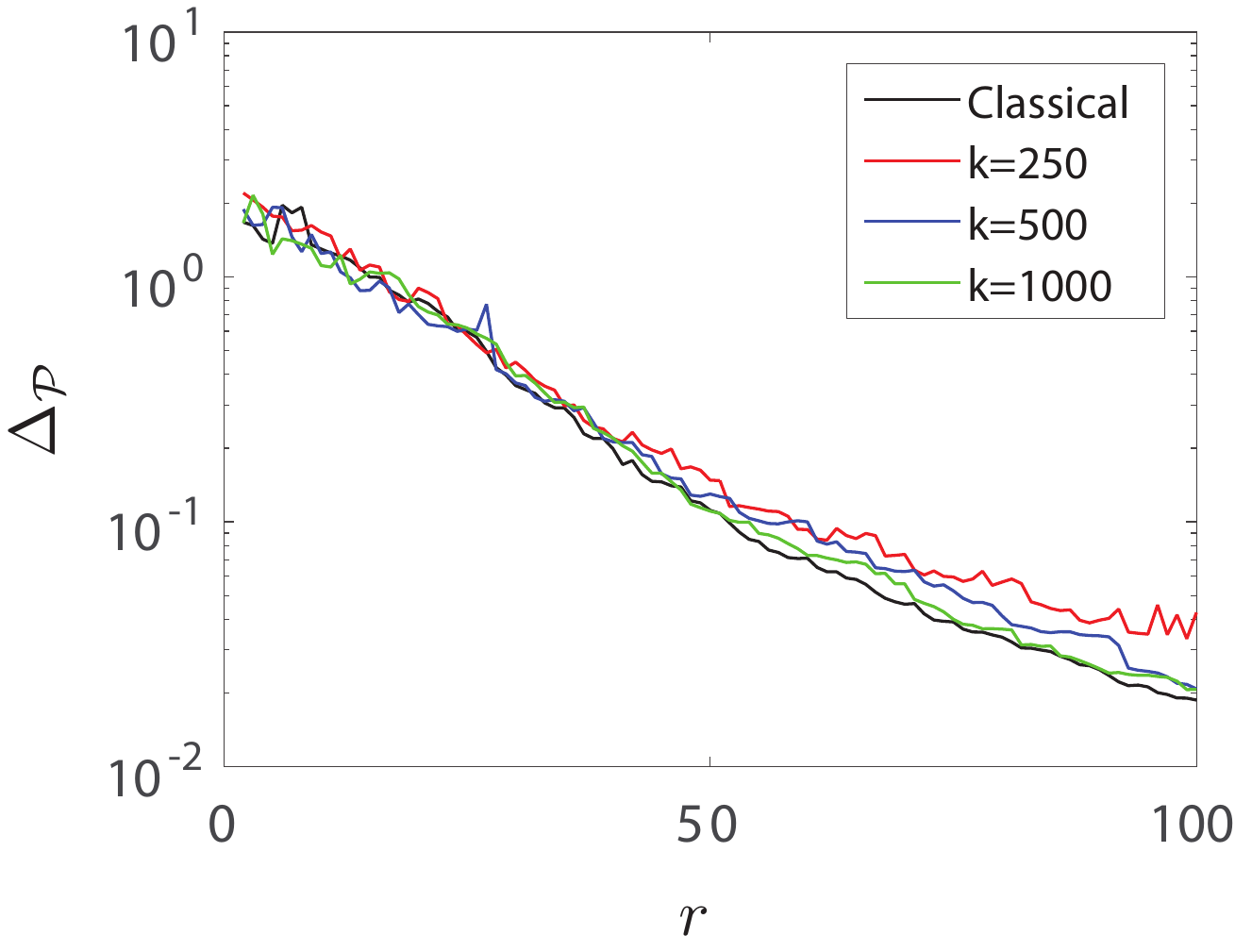}
  \caption{}
  \label{fig:Ex1_6a}
 \end{subfigure} \hspace{.01\textwidth}
 \begin{subfigure}{.4\textwidth}
  \centering
  \includegraphics[width=\textwidth]{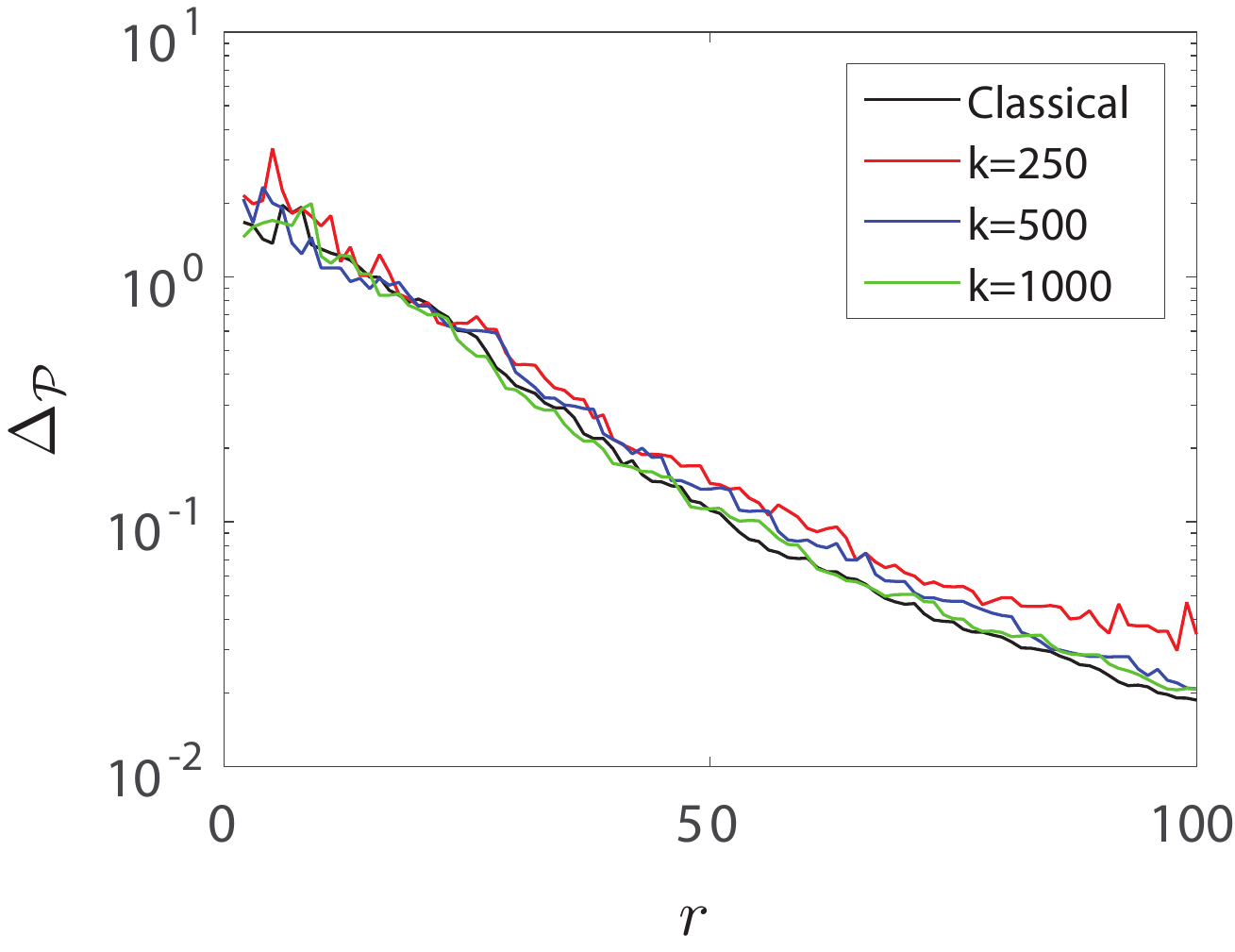}
  \caption{}
  \label{fig:Ex1_6b}
 \end{subfigure}
 \caption{{Convergence} of the classical greedy algorithm (depicted in Section\nobreakspace \ref {Greedy}) and its efficient randomized version (Algorithm\nobreakspace \ref {alg:sk_greedy_online}) using $\bOmega$ drawn from (a) Gaussian distribution or (b) P-SRHT distribution.}
 \label{fig:Ex1_6}
\end{figure}

\emph{Efficient Proper Orthogonal Decomposition.}
We finish with validation of the efficient randomized version of POD. For this experiment only $m=1000$ points from $\mathcal{P}_{\mathrm{train}}$ were considered as the training set. The POD bases were obtained with the classical method of snapshots, i.e.,~Algorithm\nobreakspace \ref {alg:approx_pod} where $\bB_r$ was computed from SVD of $\bQ \bU_m$,  or the randomized version of POD introduced in~Section\nobreakspace \ref {sk_pod}. The same $\bOmega$ was used for both the basis generation and the error estimation with $\Delta^{\mathrm{POD}}(U_r)$, defined in~(\ref {eq:sk_poderror}). From~Figure\nobreakspace \ref {fig:Ex1_5a} we observe that for large enough $k$ the quality of {the} POD basis formed with the new efficient algorithm is close to the quality of the the basis obtained with the classical method. Construction of $r=100$ basis vectors using $\bOmega$ with only $k=500$ rows provides almost optimal error.  As expected, the error indicator  $\Delta^{\mathrm{POD}}(U_r)$ is close to the exact error for large enough $k$, but it represents the error poorly for small $k$. Furthermore, $\Delta^{\mathrm{POD}}(U_r)$ is always smaller than the true error and is increasing monotonically with $k$. Figure\nobreakspace \ref {fig:Ex1_5b} depicts how the errors of the classical and randomized (with $k=500$) POD bases depend on the dimension of $U_r$. We see that the qualities of the basis and the error indicator obtained with the new version of POD remain close to the optimal ones up to dimension $r=150$. However, as $r$ becomes larger the quasi-optimality of the randomized POD degrades {so that for $r \geq 150$ the sketching size $k=500$ becomes insufficient.} 

\begin{figure}[h!]
 \centering
 \begin{subfigure}{.4\textwidth}
  \centering
  \includegraphics[width=\textwidth]{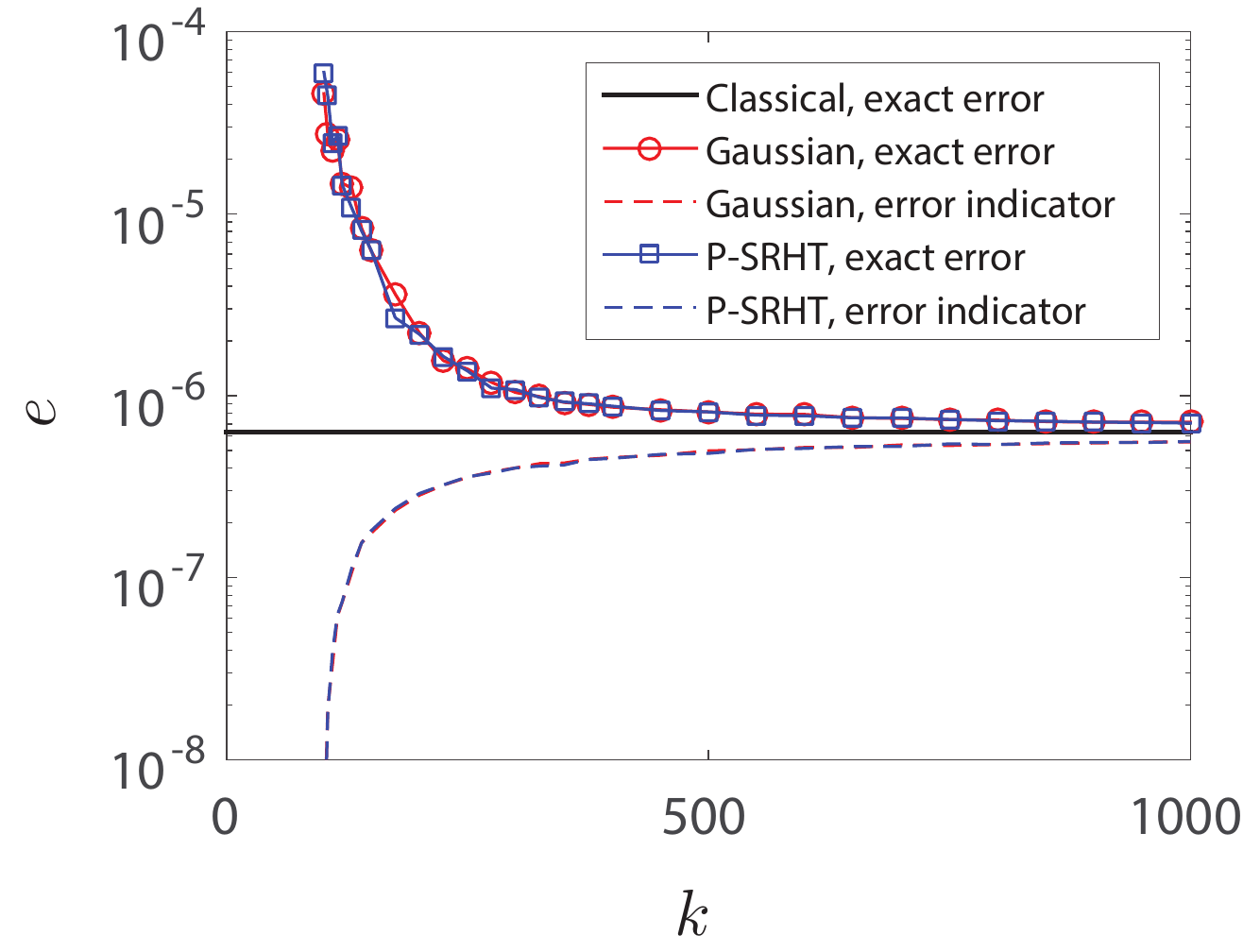}
  \caption{}
  \label{fig:Ex1_5a}
 \end{subfigure} \hspace{.01\textwidth}
 \begin{subfigure}{.4\textwidth}
  \centering
  \includegraphics[width=\textwidth]{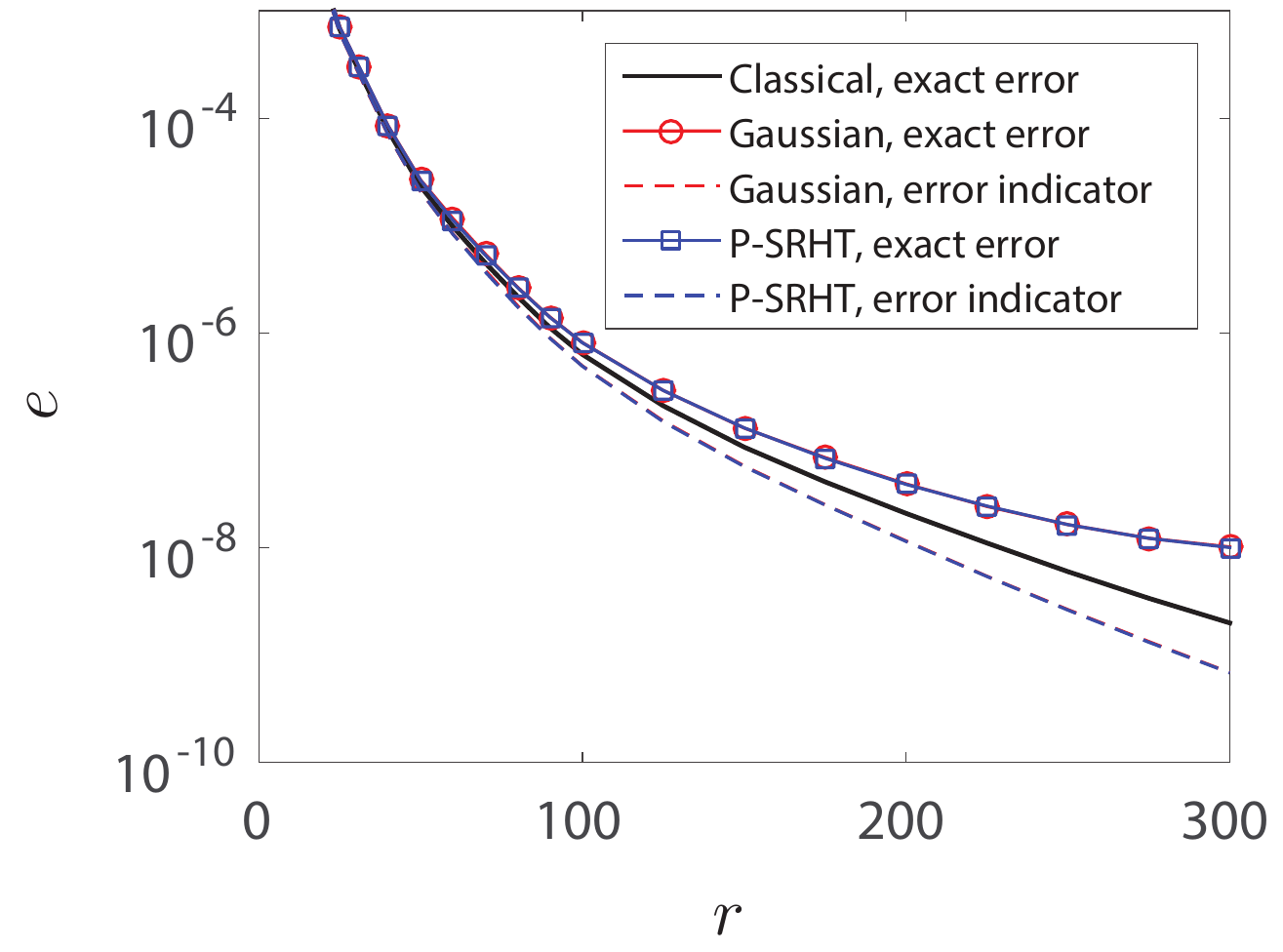}
  \caption{}
  \label{fig:Ex1_5b}
 \end{subfigure}
 \caption{Error $e=\frac{1}{m} \sum^{m}_{i=1} \| \bu(\mu^i)- \bP_{U_r}\bu(\mu^i) \|_{U}^2 /  ( \frac{1}{m} \sum^{m}_{i=1} \| \bu(\mu^i) \|_{U}^2 )$ and error indicator $e= \Delta^{\mathrm{POD}}(U_r)/  (\frac{1}{m}\sum^{m}_{i=1} \| \bu(\mu^i) \|_{U}^2 )$ associated with $U_r$ computed with traditional POD and its efficient randomized version introduced in~Section\nobreakspace \ref {sk_pod}. (a) Errors and indicators versus the number of rows of $\bOmega$ for $r=100$. (b) Errors and indicators versus the dimension of $U_r$  for $k=500$.}
\label{fig:Ex1_5}
\end{figure}

\subsection{Multi-layered acoustic cloak}
In the previous numerical example we considered a problem with strongly coercive well-conditioned operator. But as was discussed in~Section\nobreakspace \ref {SKgalproj}, random sketching with a fixed number of rows is expected to perform worse for approximating the Galerkin projection with non-coercive ill-conditioned  $\bA(\mu)$. Further, we would like to validate the methodology on such a problem.  The benchmark consists in a scattering problem of a 2D wave with perfect scatterer covered in a multi-layered cloak. For this experiment we solve the following Helmholtz equation with first order absorbing boundary conditions 
\begin{equation} \label{eq:BVP2}
 \left \{
 \begin{array}{rll}
  \Delta u + \kappa^2 u &= 0,~~  & \textup{in } \Omega \\
  i \kappa u + \frac{\partial u}{\partial \boldsymbol{n}}  &=0,~~ & \textup{on } \Gamma_{out} \\
  i \kappa u + \frac{\partial u}{\partial \boldsymbol{n}}  &=2i \kappa,~~ & \textup{on } \Gamma_{in}\\
  \frac{\partial u}{\partial \boldsymbol{n}} &= 0,~~ & \textup{on } \Gamma_{s},
 \end{array}
 \right.
\end{equation}
where $u$ is the solution field (primal unknown), $\kappa$ is the wave number and the geometry of the problem is defined in~Figure\nobreakspace \ref {fig:Ex2_intial_problem_a}. The background has a fixed wave number $\kappa=\kappa_0:=50$. The cloak consists of 10 layers of equal thicknesses enumerated in the order corresponding to the distance to the scatterer. The $i$-th layer is composed of a material with wave number $\kappa=\kappa_i$. The quantity of interest is the average of the solution field on $\Gamma_{in}$. The aim is to estimate the quantity of interest for each parameter $\mu:=(\kappa_1, ..., \kappa_{10}) \in [\kappa_0, \sqrt{2} \kappa_0]^{10} := \mathcal{P}$. The $\kappa_i$ are considered as independent random variables with log-uniform distribution over $[\kappa_0, \sqrt{2} \kappa_0]$.
The solution for a randomly chosen $\mu \in \mathcal{P}$ is illustrated in~Figure\nobreakspace \ref {fig:Ex2_intial_problem_b}.
\begin{figure}[h!]
 \centering
 \begin{minipage}[c]{.4\textwidth}
  \centering
  \includegraphics[width=.96\textwidth]{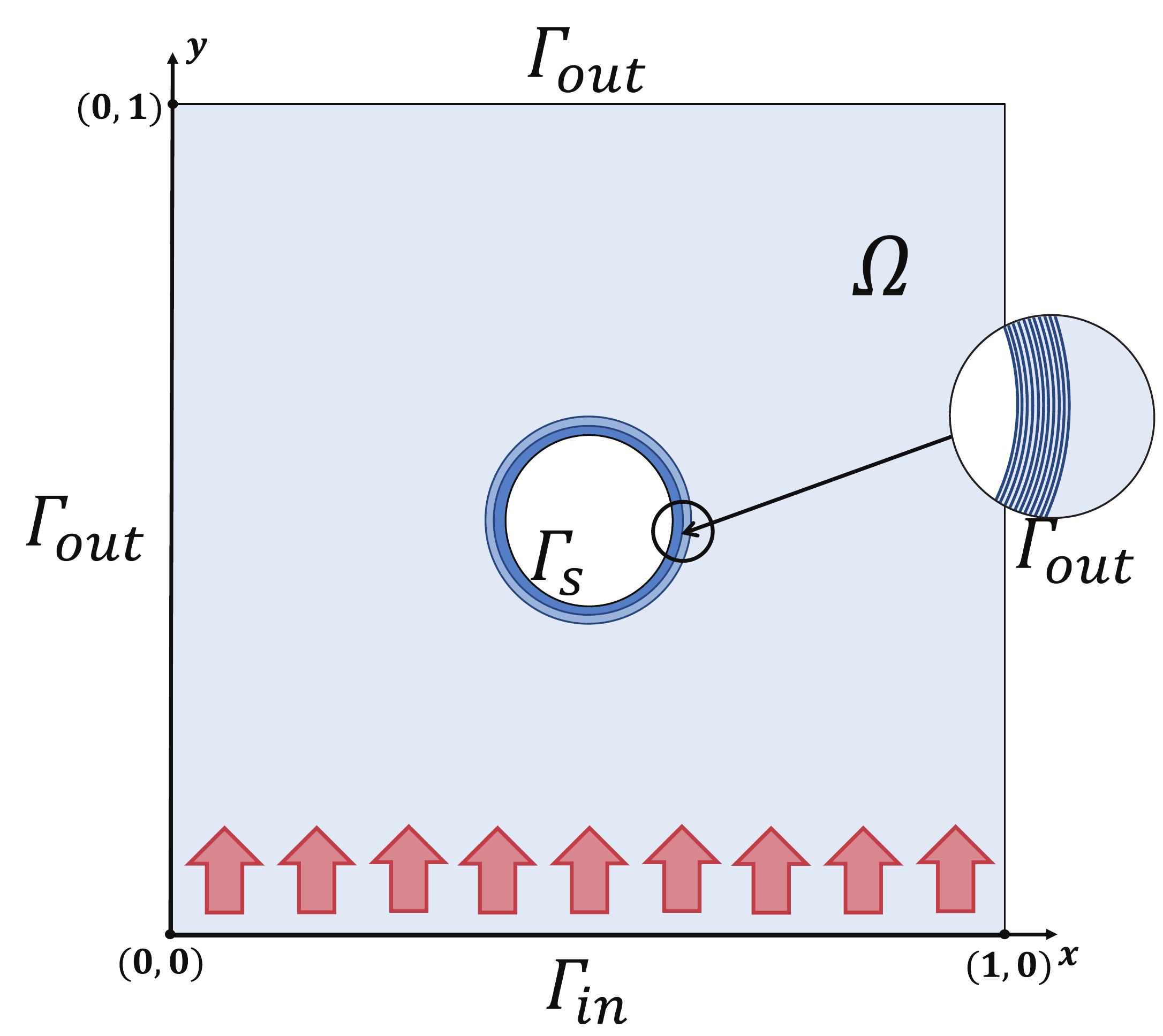}
 \end{minipage}%
 \begin{minipage}[c]{.4\textwidth}
  \centering
  \includegraphics[width=.84\textwidth]{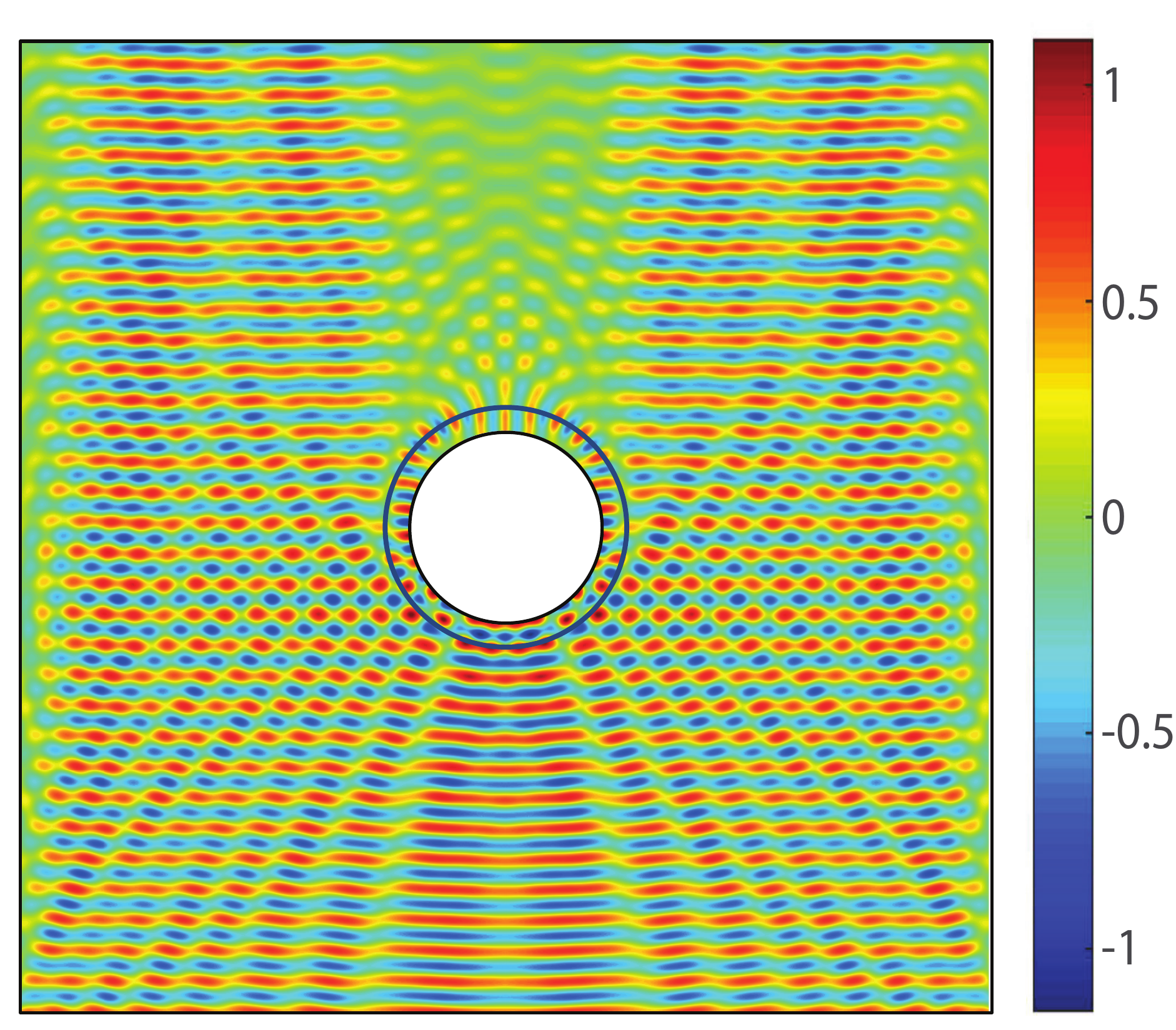}
 \end{minipage} \\ 
 \begin{minipage}[b]{.4\textwidth}
  \centering
  \subcaption{Geometry}
  \label{fig:Ex2_intial_problem_a}
 \end{minipage}%
 \begin{minipage}[b]{.4\textwidth}
  \subcaption{Solution at random $\mu$}
  \label{fig:Ex2_intial_problem_b}
 \end{minipage}
 \caption{(a) Geometry of acoustic cloak benchmark. (b) The real component of $u$  for randomly picked parameter $\mu=\allowbreak (66.86,\allowbreak 54.21,\allowbreak 61.56,\allowbreak 64.45,\allowbreak 66.15,\allowbreak 58.42,\allowbreak 54.90,\allowbreak  63.79,\allowbreak 58.44,\allowbreak 63.09)$.}
 \label{fig:Ex2_intial_problem}
\end{figure}

The problem has a symmetry with respect to the vertical axis $x = 0.5$. Consequently, only half of the domain has to be considered for discretization. The discretization was performed {using quadratic triangular finite elements with} approximately {17} complex degrees of freedom per wavelength, i.e., around $200000$ complex degrees of freedom in total. A function $w$ in the approximation space is identified with a vector $\bw \in U$. The solution space $U$ is equipped with an inner product compatible with the $H^1$ inner product, i.e., 
\begin{equation*}
 \|\bw \|_{U}^2:= \| \boldsymbol{\nabla}w \|^2_{L_2}+ \kappa_0^2 \| w \|^2_{L_2} .
\end{equation*}
 Further, $20000$ and $1000$  independent samples were considered as the training set $\mathcal{P}_{\mathrm{train}}$ and the test set $\mathcal{P}_{\mathrm{test}}$, respectively. {The} sketching matrix $\bTheta$ was constructed as in the thermal block benchmark, i.e., $\bTheta:= \bOmega\bQ$, where $\bOmega \in \mathbb{R}^{k\times s}$ is either a Gaussian matrix or P-SRHT and $\bQ \in \mathbb{R}^{s\times n}$ is the transposed Cholesky factor of $\bR_U$. In addition, we used  $\bPhi:= \bGamma\bTheta$, where $\bGamma \in \mathbb{R}^{k'\times k}$ is a Gaussian matrix and $k'=200$.

Below we present validation of the Galerkin projection and the greedy algorithm only. The performance of our methodology for error estimation and POD does not depend on the operator and is similar to the performance observed in the previous numerical example. 

\emph{Galerkin projection.} 
A subspace $U_r$ was generated with $r=150$ iterations of the randomized greedy algorithm (Algorithm\nobreakspace \ref {alg:sk_greedy_online}) with a $\bOmega$ drawn from the  P-SRHT distribution with $k =20000$ rows. Such $U_r$  was then used for validation of the Galerkin projection. We evaluated multiple approximations of $\bu(\mu)$ using either the classical projection (\ref {eq:galproj}) or its randomized version (\ref {eq:SKgalproj}). Different $\bOmega$ were considered for (\ref {eq:SKgalproj}). As before, the approximation and residual errors are respectively defined by
$e_\mathcal{P}:=\max_{\mu \in \mathcal{P}_{\mathrm{test}}} \|\bu(\mu) - \bu_r(\mu)\|_{U} / \max_{\mu \in \mathcal{P}_{\mathrm{test}}} \|\bu(\mu)\|_{U}$ and $\Delta_\mathcal{P}:=\max_{\mu \in \mathcal{P}_{\mathrm{test}}} \| \br(\bu_r(\mu); \mu) \|_{U'} /\max_{\mu \in \mathcal{P}_{\mathrm{test}}} \|\bb(\mu)\|_{U'}$. For each type and size of $\bOmega$, 20 samples of $e_\mathcal{P}$ and $\Delta_\mathcal{P}$ were evaluated. The errors are presented in Figure\nobreakspace \ref {fig:Ex2_1}.  This experiment reveals that indeed the performance of random sketching is worse than in the thermal block benchmark (see Figure\nobreakspace \ref {fig:Ex1_1}). For $k=1000$ the error of the randomized version of Galerkin projection is much larger than the error of the classical projection. Whereas for the same value of $k$ in the thermal block benchmark practically no difference between the qualities of the classical projection and its sketched version was observed. It can be explained by the fact that the quality of randomized Galerkin projection depends on the coefficient $a_r(\mu)$ defined in Proposition\nobreakspace \ref {thm:skcea}, which in its turn depends on the operator. {In both numerical examples the coefficient $a_r(\mu)$ was measured over $\mathcal{P}_{\mathrm{test}}$. We observed that {here} $\max_{\mu \in \mathcal{P}_{\mathrm{test}}} a_r(\mu) = 28.3$, while in the thermal block benchmark $\max_{\mu \in \mathcal{P}_{\mathrm{test}}} a_r(\mu) = 2.65$.}  In addition, here we work on the complex field instead of the real field and consider slightly larger reduced subspaces, which {could} also have an impact on the accuracy of random sketching.  Reduction of performance, however, is not that severe and already starting from  $k=15000$ the sketched version of Galerkin projection has an error close to the classical one. Such size of~$\bOmega$ is still very small compared to the dimension of the discrete problem and provides drastic reduction of the computational cost. Let us also note that one could obtain a good approximation of $\bu(\mu)$ from the sketch with $k \ll 15000$ by considering another type of projection (a randomized minimal residual projection) proposed in~\cite{balabanov2018}.
\begin{figure}[h!]
	\centering
 \begin{subfigure}[b]{.4\textwidth}
  \centering
  \includegraphics[width=\textwidth]{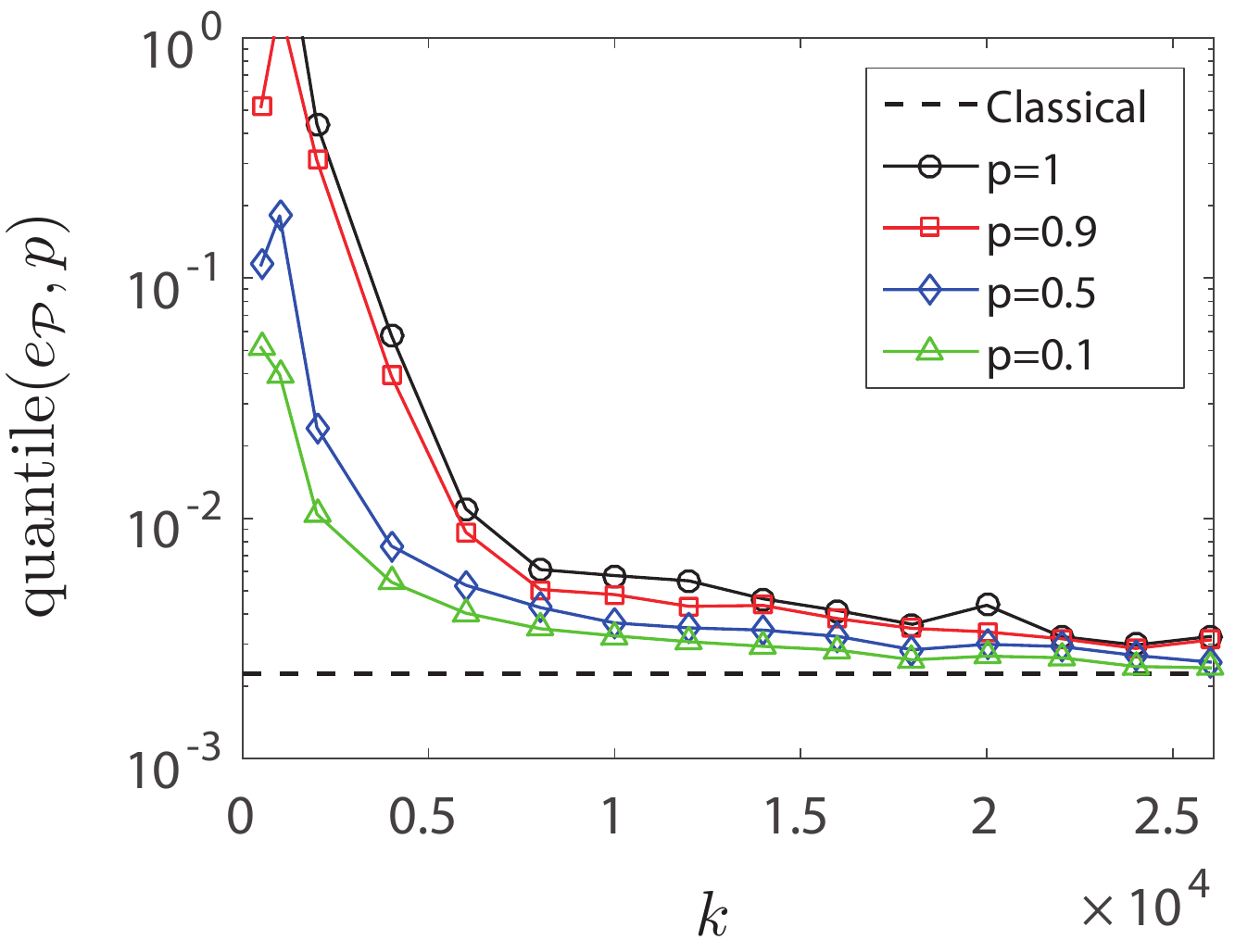}
  \caption{}
  \label{fig:Ex2_1a}
 \end{subfigure} \hspace{.01\textwidth}
 \begin{subfigure}[b]{.4\textwidth}
  \centering
  \includegraphics[width=\textwidth]{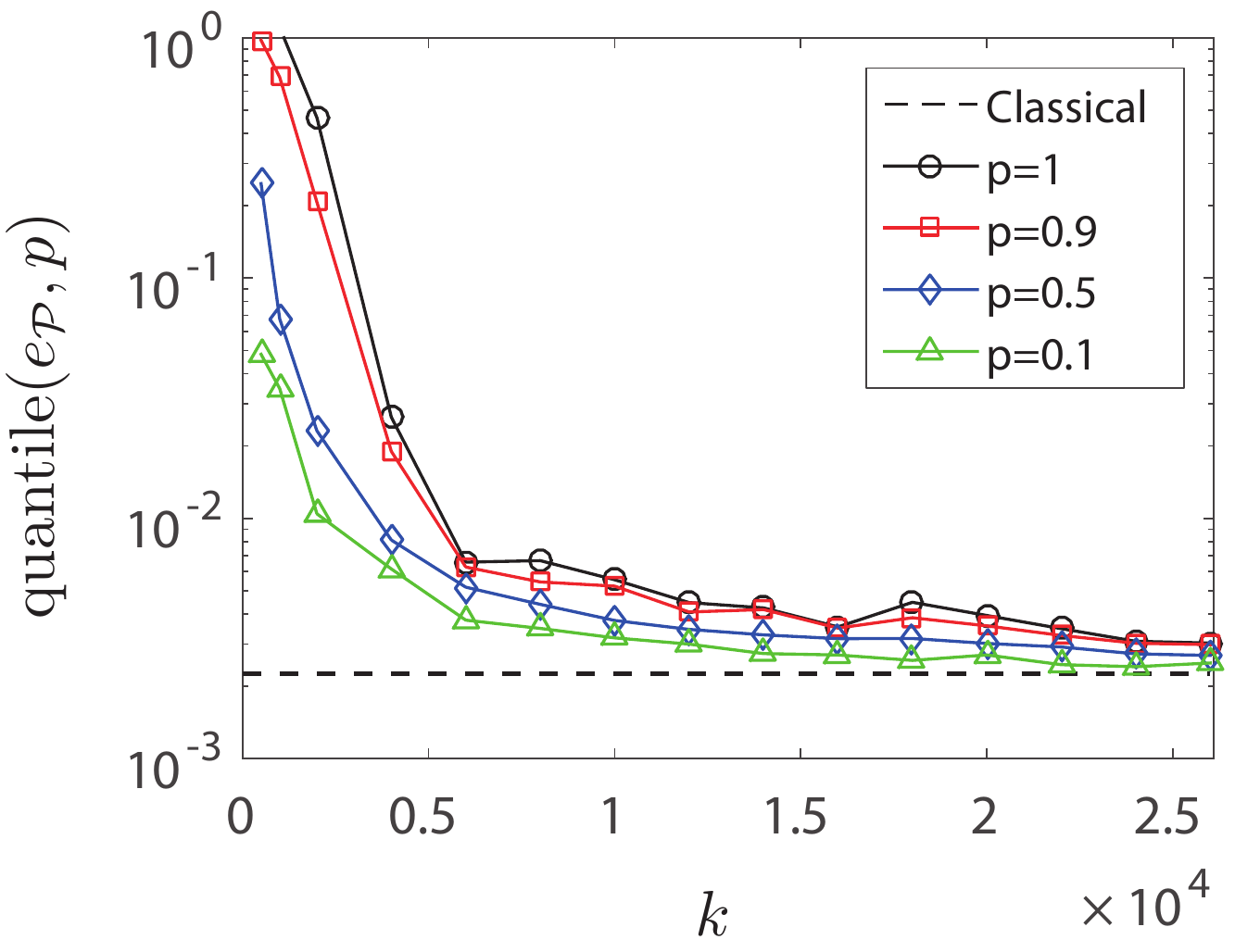}
  \caption{}
  \label{fig:Ex2_1b}
   \end{subfigure}
   
  \begin{subfigure}[b]{.4\textwidth}
  \centering
  \includegraphics[width=\textwidth]{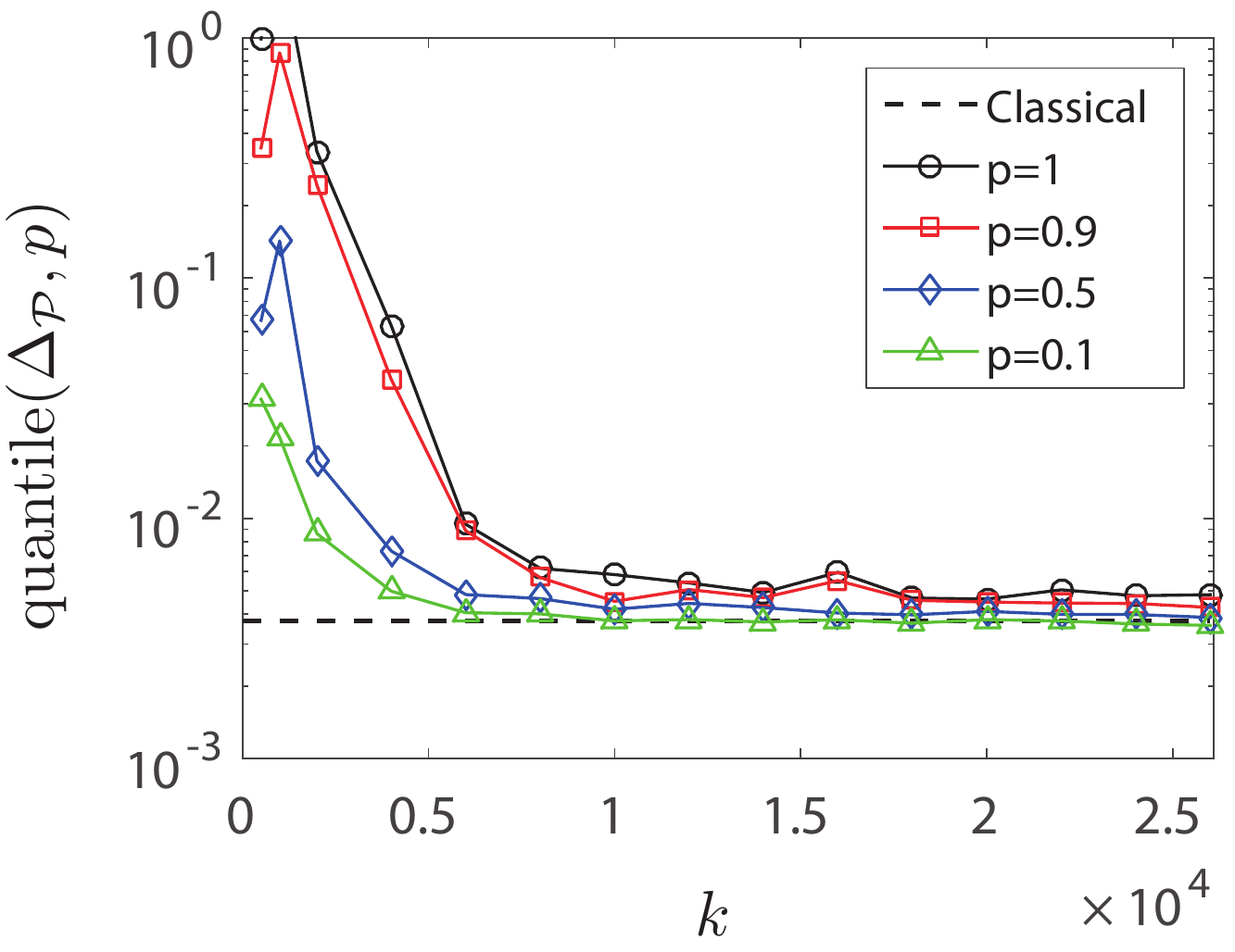}
  \caption{}
  \label{fig:Ex2_1c}
 \end{subfigure} \hspace{.01\textwidth}
 \begin{subfigure}[b]{.4\textwidth}
  \centering
  \includegraphics[width=\textwidth]{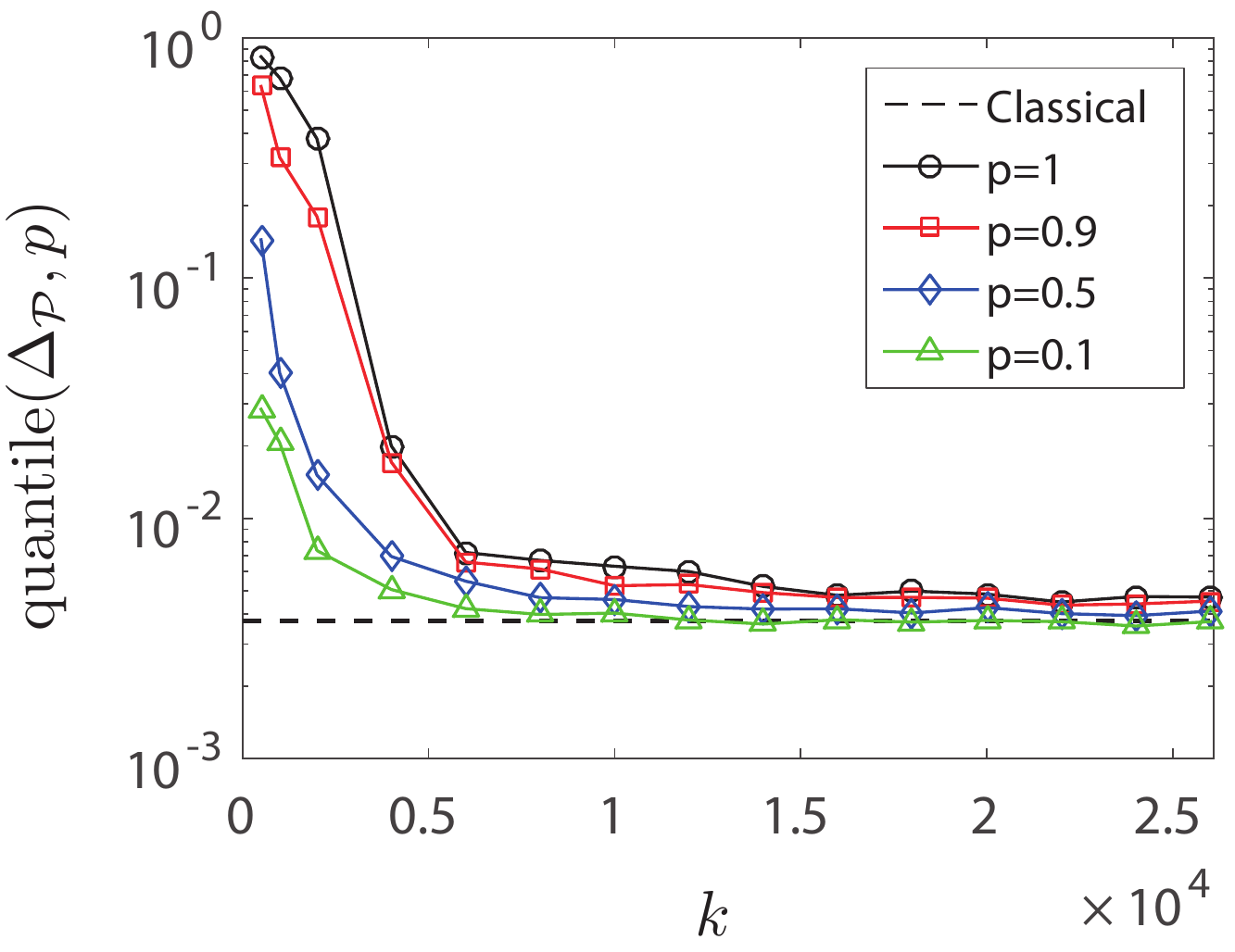}
  \caption{}
  \label{fig:Ex2_1d}
 \end{subfigure}
 \caption{Error $e_\mathcal{P}$ and residual error $\Delta_\mathcal{P}$ of the classical Galerkin projection and quantiles of probabilities $p=1, 0.9, 0.5$ and $0.1$ over 20 samples of $e_\mathcal{P}$ and $\Delta_\mathcal{P}$ of the randomized Galerkin projection versus the number of rows of $\bOmega$. (a) Exact error $e_\mathcal{P}$, with rescaled Gaussian distribution as $\bOmega$. (b) Exact error $e_\mathcal{P}$, with P-SRHT matrix as $\bOmega$. (c) Residual error $\Delta_\mathcal{P}$, with rescaled Gaussian distribution as $\bOmega$. (d) Residual error $\Delta_\mathcal{P}$, with P-SRHT matrix as $\bOmega$.}
   \label{fig:Ex2_1}
 \end{figure}

Let us further note that we are in the so called ``compliant case'' (see Remark\nobreakspace \ref {rmk:compliant}). Thus, for the classical Galerkin projection we have $s_r(\mu) =s^{\mathrm{pd}}_r(\mu)$ and for the sketched Galerkin projection, $s_r(\mu)=s^{\mathrm{spd}}_r(\mu)$. The output quantity $s_r(\mu)$ was computed with the classical Galerkin projection and with the randomized Galerkin projection employing different $\bOmega$. For each $\bOmega$ we also computed the improved sketched correction $s^{\mathrm{spd+}}_r(\mu)$ (see~Section\nobreakspace \ref {sk_pd_correction}) using $W^\mathrm{du}_r:=U^{\mathrm{du}}_i$ with $i^\mathrm{du}=30$. It required inexpensive additional computations which are in about $5$ times cheaper (in terms of both complexity and memory) than the computations involved in the classical method. The error on the output quantity is measured by $d_\mathcal{P}:=\max_{\mu \in \mathcal{P}_{\mathrm{test}}} |s(\mu) - \widetilde{s}_r(\mu) | / \max_{\mu \in \mathcal{P}_{\mathrm{test}}} |s(\mu)|$, where $\widetilde{s}_r(\mu)=s_r(\mu)$ or $s^{\mathrm{spd+}}_r(\mu)$.        
For each random distribution type $20$ samples of $d_\mathcal{P}$ were evaluated. Figure\nobreakspace \ref {fig:Ex2_2} describes how the error of the output quantity depends on $k$. For small $k$ the error is large because of the poor quality of the projection and lack of precision when approximating the inner product for $s^{\mathrm{pd}}_r(\mu)$ in~(\ref {eq:correction}) by the one in~(\ref {eq:skcorrection}). But starting from $k=15000$ we see that the quality of $s_r(\mu)$ obtained with the random sketching technique becomes close to the quality of the output computed with the classical Galerkin projection. For $k \geq 15000$ the randomized Galerkin projection has practically the same accuracy as the classical one. Therefore, for such values of $k$ the error depends mainly on the precision of the approximate inner product for $s^{\mathrm{pd}}_r(\mu)$. Unlike in the thermal block problem (see Figure\nobreakspace \ref {fig:Ex1_2}), in this experiment  the quality of the classical method is attained by $s_r(\mu)=s^{\mathrm{spd}}_r(\mu)$ with $k \ll n$. Consequently, the benefit of employing the improved correction $s^{\mathrm{spd+}}_r(\mu)$ here is not as evident as in the previous numerical example. This experiment only proves that  the error associated with approximation of the inner product for $s^{\mathrm{pd}}_r(\mu)$ does not depend on the condition number and the dimension of the operator. 

\begin{figure}[h!]
 \centering
 \begin{subfigure}{.4\textwidth}
  \centering
  \includegraphics[width=\textwidth]{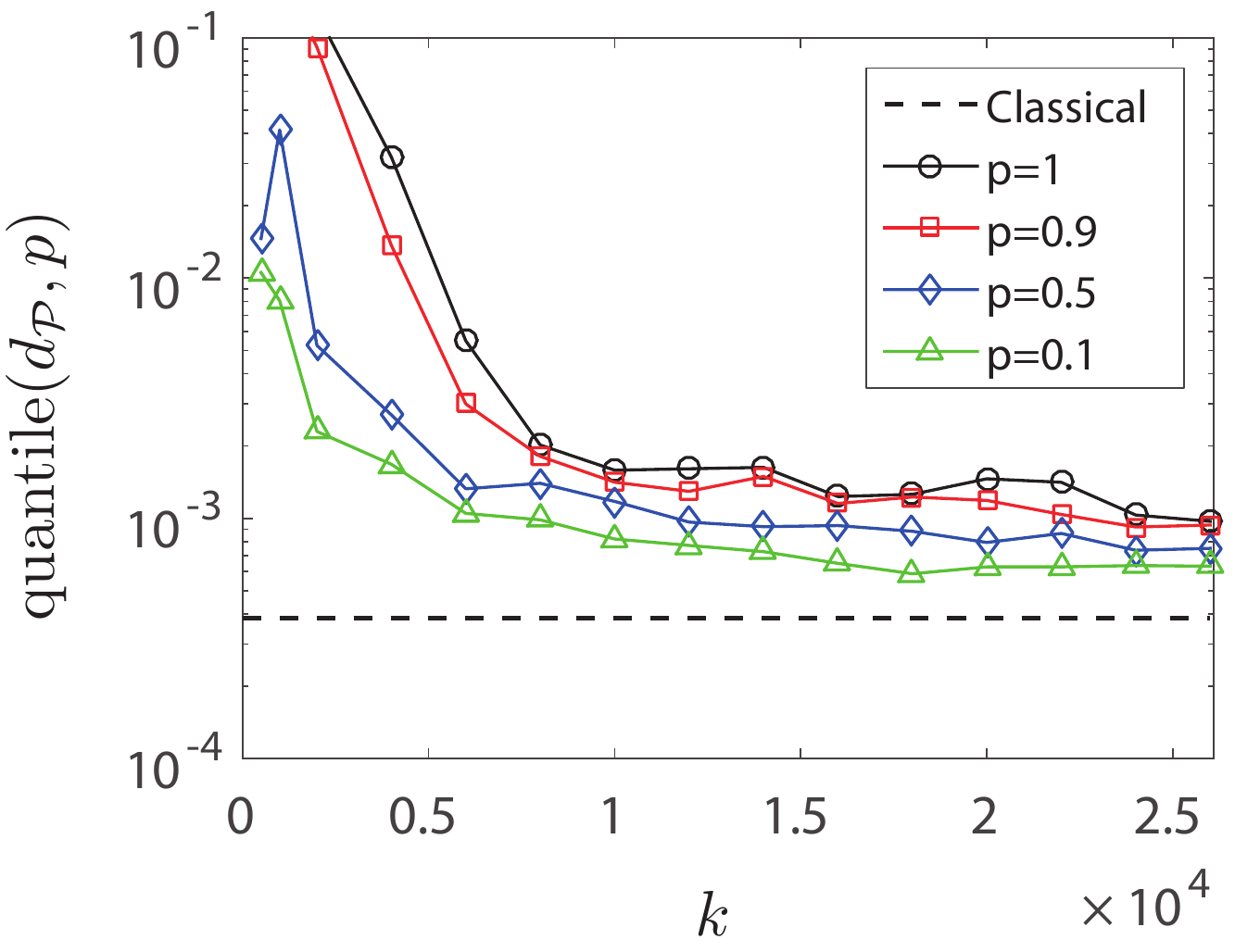}
  \caption{}
  \label{fig:Ex2_2a}
 \end{subfigure} \hspace{.01\textwidth}
 \begin{subfigure}{.4\textwidth}
  \centering
  \includegraphics[width=\textwidth]{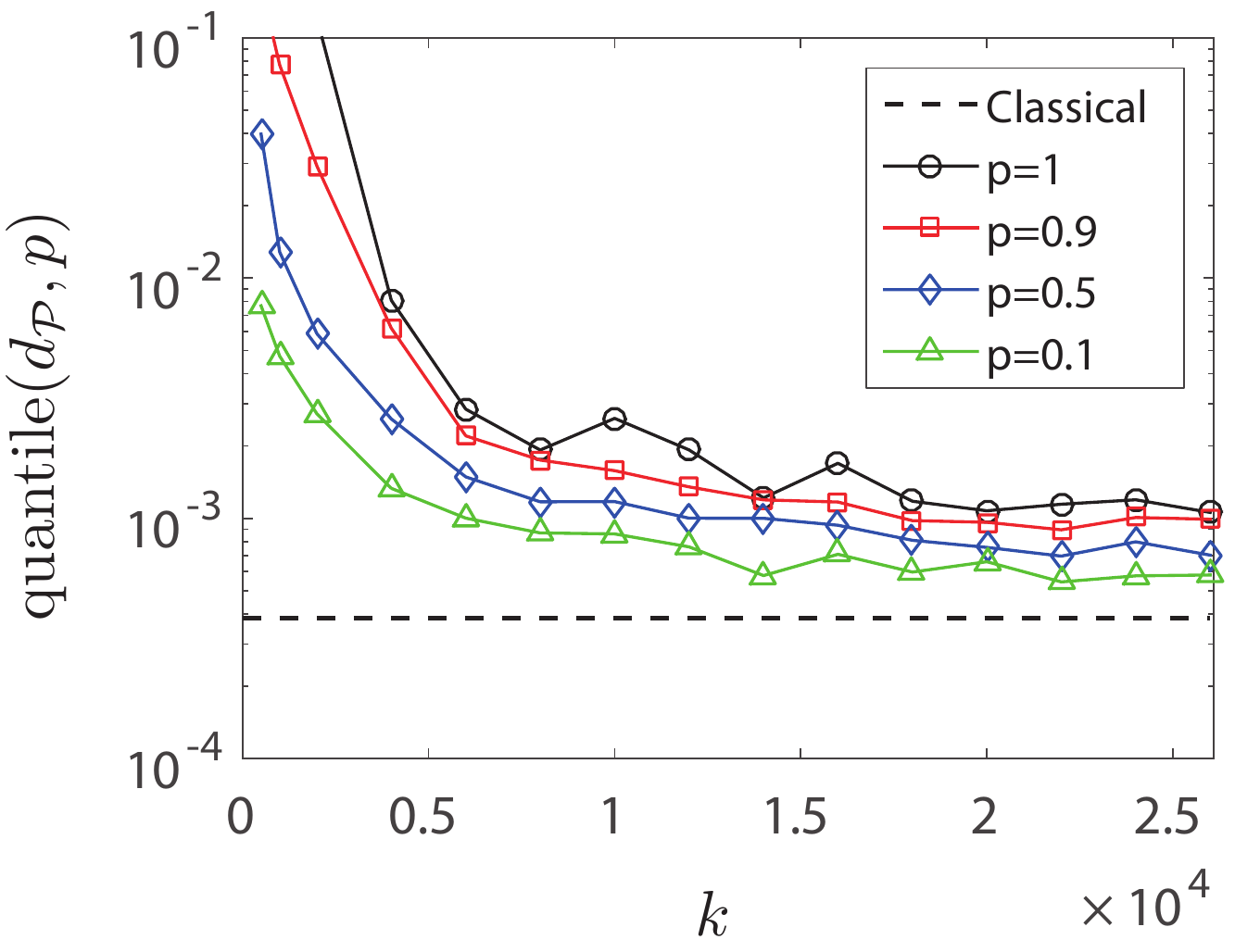}
  \caption{}
  \label{fig:Ex2_2b}
 \end{subfigure}
 \begin{subfigure}{.4\textwidth}
  \centering
  \includegraphics[width=\textwidth]{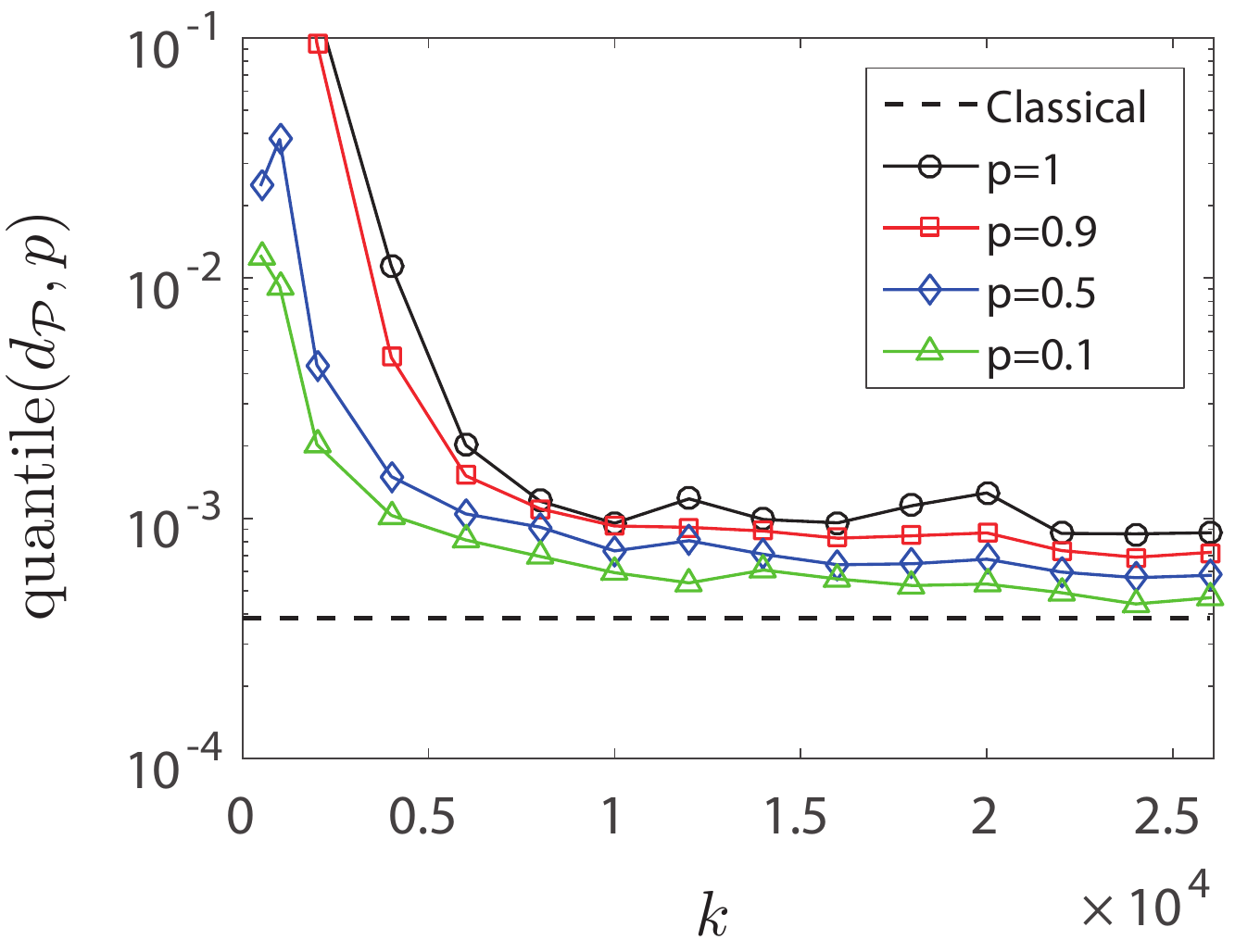}
  \caption{}
  \label{fig:Ex2_2c}
 \end{subfigure} \hspace{.01\textwidth}
 \begin{subfigure}{.4\textwidth}
  \centering
  \includegraphics[width=\textwidth]{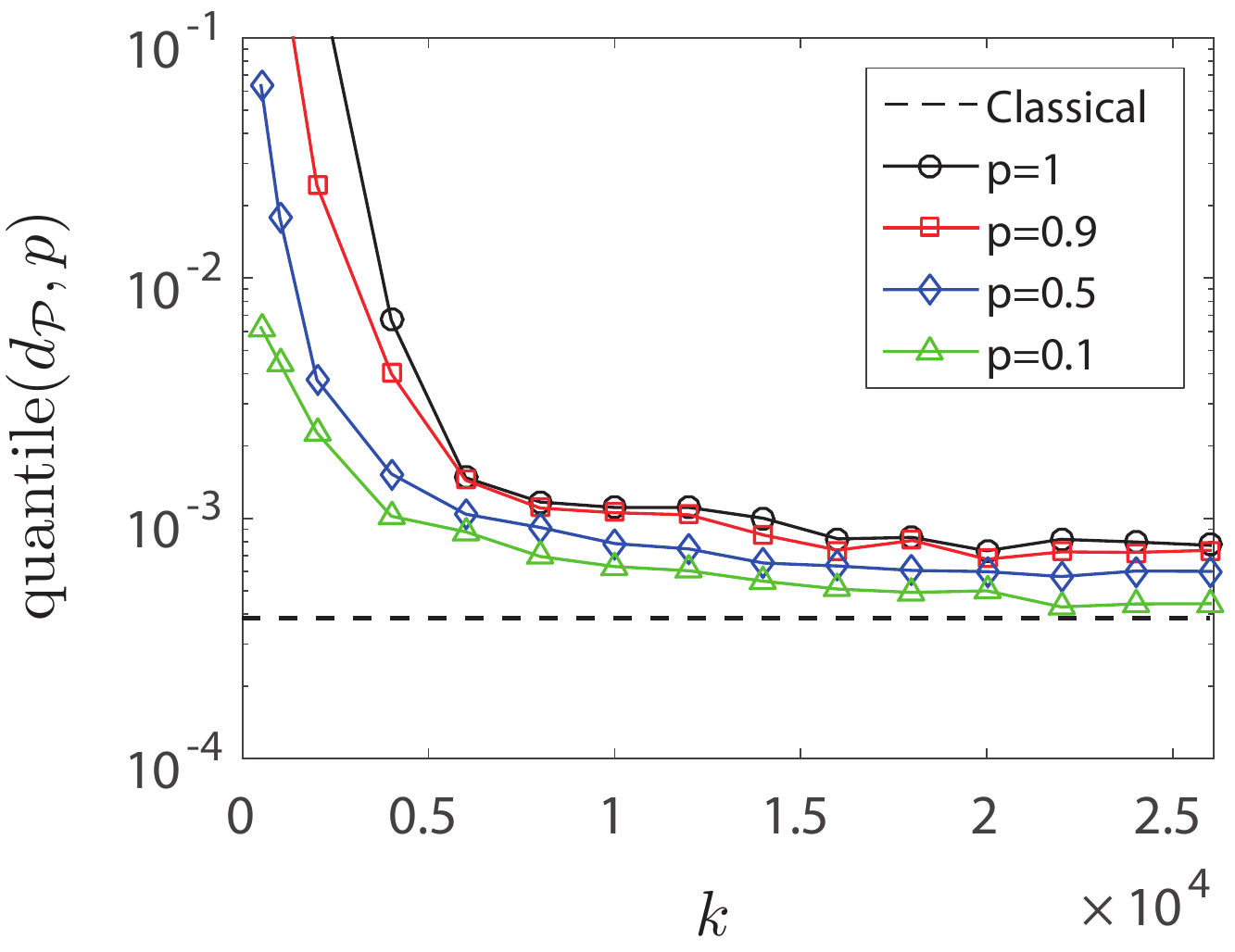}
  \caption{}
  \label{fig:Ex2_2d}
 \end{subfigure}
 \caption{The error $d_\mathcal{P}$ of the classical output quantity and quantiles of probabilities $p=1, 0.9, 0.5$ and $0.1$ over 20 samples of $d_\mathcal{P}$ of the output quantities computed with random sketching versus the number of rows of $\bOmega$. (a) The errors of the classical $s_r(\mu)$ and the randomized $s_r(\mu)$ with Gaussian matrix as $\bOmega$. (b) The errors of the classical $s_r(\mu)$ and the randomized $s_r(\mu)$ with P-SRHT distribution as $\bOmega$. (c) The errors of the classical $s_r(\mu)$ and $s^{\mathrm{spd+}}_r(\mu)$ with Gaussian matrix as $\bOmega$ and $W^\mathrm{du}_r:=U^{\mathrm{du}}_i$, $i^\mathrm{du}=30$.  (d) The errors of the classical $s_r(\mu)$ and $s^{\mathrm{spd+}}_r(\mu)$ with P-SRHT distribution as $\bOmega$ and $W^\mathrm{du}_r:=U^{\mathrm{du}}_i$, $i^\mathrm{du}=30$.}
 \label{fig:Ex2_2}
\end{figure}
 
\emph{Randomized greedy algorithm.} Finally, we performed $r=150$ iterations of the classical greedy algorithm (see Section\nobreakspace \ref {Greedy}) and its randomized version (Algorithm\nobreakspace \ref {alg:sk_greedy_online}) using different distributions and sizes for $\bOmega$, and a Gaussian random matrix with $k'=200$ rows for $\bGamma$. As in the thermal block benchmark, the error at $i$-th iteration is measured by $\Delta_\mathcal{P}:=\max_{\mu \in \mathcal{P}_{\mathrm{train}}} \| \br(\bu_i(\mu); \mu) \|_{U'}/\max_{\mu \in \mathcal{P}_{\mathrm{train}}} \|\bb(\mu)\|_{U'}$.  For $k=1000$ we reveal poor performance of~Algorithm\nobreakspace \ref {alg:sk_greedy_online} (see~Figure\nobreakspace \ref {fig:Ex2_3}). It can be explained by the fact that for such size of $\bOmega$ the randomized Galerkin projection has low accuracy. For $k=20000$, however, the {convergence} of the classical greedy algorithm is fully preserved.
\begin{figure}[h]
 \centering
 \begin{subfigure}{.4\textwidth}
  \centering  
  \includegraphics[width=\textwidth]{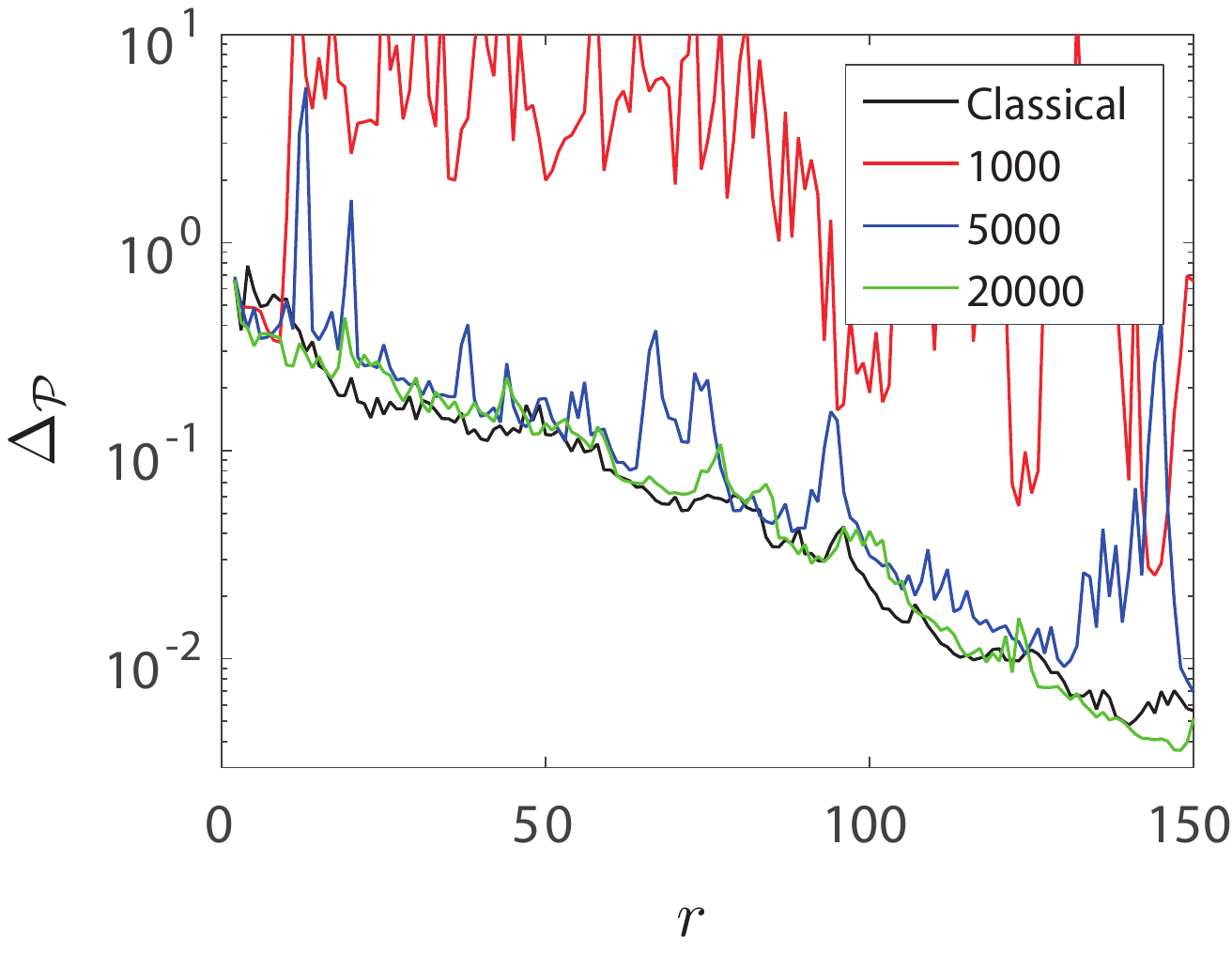}
  \caption{}
  \label{fig:Ex2_3a}
 \end{subfigure} \hspace{.03\textwidth}
 \begin{subfigure}{.4\textwidth}
  \centering
  \includegraphics[width=\textwidth]{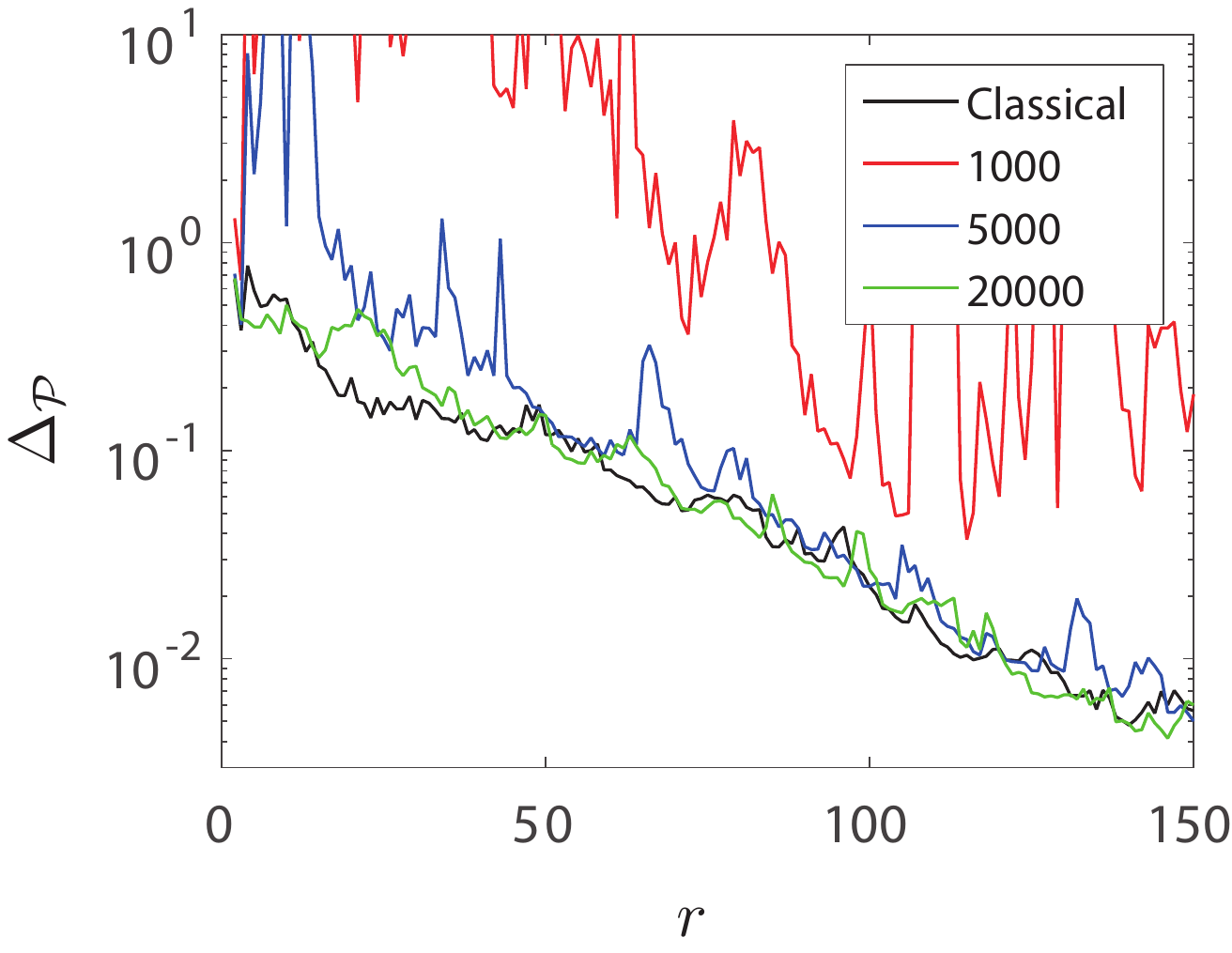}
  \caption{}
  \label{fig:Ex2_3b}
 \end{subfigure}
 \caption{{Convergence} of the classical greedy algorithm (see Section\nobreakspace \ref {Greedy}) and its efficient randomized version (Algorithm\nobreakspace \ref {alg:sk_greedy_online}) using $\bOmega$ drawn from (a) Gaussian distribution or (b) P-SRHT distribution.}
\label{fig:Ex2_3}
\end{figure}

{\emph{Comparison of computational costs.} }
Even though the size of $\bOmega$ has to be considered larger than for the thermal block problem, our methodology still yields considerable reduction of the computational costs compared to the classical approach.  The implementation was carried out in Matlab\textsuperscript{\textregistered} {R2015a} with an external $\verb!C++!$ function for the fast Walsh-Hadamard transform (see e.g. \url{https://github.com/sheljohn/WalshHadamard}). Our codes were not designed for a specific problem but rather for a generic multi-query MOR. The algorithms were executed on an Intel\textsuperscript{\textregistered} Core{\texttrademark} i7-7700HQ 2.8GHz CPU, with 16.0GB RAM memory.

Let us start with validation of the computational cost reduction of the greedy algorithm. In~Table\nobreakspace \ref {tab:runtimes} we provide the runtimes of the classical greedy algorithm and~Algorithm\nobreakspace \ref {alg:sk_greedy_online} employing $\bOmega$ drawn from P-SRHT distribution with $k=20000$ rows. In~Table\nobreakspace \ref {tab:runtimes} the computations are divided into three  basic categories: computing the snapshots (samples of the solution), precomputing the affine expansions for the online solver, and finding $\mu^{i+1} \in \mathcal{P}_\mathrm{train}$ which maximizes the error indicator with a provisional online solver. The first category includes evaluation of $\bA(\mu)$ and $\bb(\mu)$ using their affine expansions and solving the systems with a built in~Matlab\textsuperscript{\textregistered} linear solver. The second category consists of evaluating the random sketch in~Algorithm\nobreakspace \ref {alg:sk_greedy_online}; evaluating high-dimensional matrix-vector products and inner products for the Galerkin projection; evaluating high-dimensional matrix-vector products and inner products for the error estimation; and the remaining computations, such as precomputing a decomposition of $\bR_U$, memory allocations, orthogonalization of the basis, etc. In its turn, the third category of computations includes generating $\bGamma$ and evaluating the affine factors of $\bV^{\bPhi}_i(\mu)$ and $\bb^{\bPhi}(\mu)$ from the affine factors of $\bV^{\bTheta}_i(\mu)$ and $\bb^{\bTheta}(\mu)$ at each iteration of~Algorithm\nobreakspace \ref {alg:sk_greedy_online}; evaluating the reduced systems from the precomputed affine expansions and solving them with a built in~Matlab\textsuperscript{\textregistered} linear solver, for all $\mu \in \mathcal{P}_\mathrm{train}$, at each iteration; evaluating the residual terms from the affine expansions and using them to evaluate the residual errors of the Galerkin projections, for all $\mu \in \mathcal{P}_\mathrm{train}$, at each iteration. 

We observe that evaluating the snapshots occupied only $6\%$ of the overall runtime of the classical greedy algorithm. The other $94 \%$ could be subject to reduction with {the} random sketching technique. Due to operating on a large training set, the cost of solving (including estimation of the error) reduced order models on $\mathcal{P}_{\mathrm{train}}$ has a considerable impact on the runtimes of both  classical and randomized algorithms. This cost, however, is independent of the dimension of the full system of equations and will become negligible for larger problems. Nevertheless, {for $r=150$} the randomized procedure for error estimation (see~Section\nobreakspace \ref {efficient_res_norm}) yielded reduction of the aforementioned cost in about $2$ times. As expected, in the classical method the most expensive computations are numerous evaluations of high-dimensional matrix-vector and inner products. For large problems these computations can become a bottleneck of an algorithm. Their cost reduction by random sketching is drastic. We observe that for the classical algorithm the corresponding runtime grows quadratically with $r$ while for the randomized algorithm it grows only linearly. The cost of this step for $r=150$ iterations of the greedy algorithm was divided $15$. In addition, random sketching helped to reduce memory consumption. The memory required by $r=150$ iterations of the greedy algorithm has been reduced from $6.1$GB {(including storage of affine factors of $\bA(\mu)\bU_i$)} to only $1$GB, from which $0.4$GB is meant for the initialization, i.e., defining the discrete problem, precomputing the decomposition of $\bR_U$, etc.
\begin{table}[tbhp]
 \caption{The CPU times in seconds taken by each type of computations in the classical greedy algorithm (see Section\nobreakspace \ref {Greedy})~and the randomized greedy algorithm (Algorithm\nobreakspace \ref {alg:sk_greedy_online}). }
 \label{tab:runtimes}
 \centering
 \scalebox{0.88}{
 \begin{tabular}{|c|l|r|r|r|r|r|r|r|} \hline
    \multirow{2}{*}{Category}  &  \multirow{2}{*}{Computations} & \multicolumn{3}{c|}{Classical} & \multicolumn{3}{c|}{Randomized}  \\  \cline{3-8}
   &    & $r=50$  & $r=100$ & $r=150$ & $r=50$ & $r=100$ & $r=150$ \\  [2pt]  \hline 
 snapshots   &  & $143$ & $286$ & $430$ & $143$ & $287$ & $430$ \\  [2pt] \hline
 \multirow{5}{*}{
 \begin{tabular}{c c c} high-dimensional \\ matrix-vector \& \\ inner products \end{tabular}}  & sketch & $-$ & $-$ & $-$ & $54$ & $113$ & $177$ \\   \cline{2-8}
  & Galerkin & $59$ & $234$ & $525$ & $3$ & $14$ & $31$ \\  \cline{2-8}   
  & error    & $405$ & $1560$ & $3444$ & $-$ & $-$ & $-$ \\  \cline{2-8} 
  & remaining    & $27$ & $196$ & $236$ & $7$ & $28$ & $67$ \\ \cline{2-8}  
  & total    & $491$ & $1899$ & $4205$ & $64$ & $154$ & $275$ \\  \hline 
  \multirow{4}{*}{\begin{tabular}{c c} provisional \\ online solver  \end{tabular}}     & sketch & $-$ & $-$ & $-$ & $56$ & $127$ & $216$ \\ \cline{2-8}   
  & Galerkin& $46$ & $268$ & $779$ & $50$ & $272$ & $783$ \\ \cline{2-8}
  & error   & $45$ & $522$ & $2022$ & $43$ & $146$ & $407$ \\   \cline{2-8}
  & total   & $91$ & $790$ & $2801$ & $149$ & $545$ & $1406$ \\   \hline 
 \end{tabular}
 }
 \end{table}

The improvement of the efficiency of the online stage can be validated by comparing the CPU times of the provisional online solver in the greedy algorithms. Table\nobreakspace \ref {tab:runtimes} presents the CPU times taken by the provisional online solver at {the} $i$-th iteration of the classical and the sketched greedy algorithms, where the solver is used for efficient computation of the reduced models associated with {an} $i$-dimensional approximation space $U_i$ for all parameter's values from the training set. These computations consist of evaluating the reduced systems from the affine expansions and their solutions with the Matlab\textsuperscript{\textregistered} linear solver, and computing residual-based error estimates using~(\ref{eq:compute_res}) for the classical method or~(\ref{eq:eff_comp_res}) for the estimation with random sketching. Moreover, the sketched online stage also involves generation of $\bGamma$ and computing $\bV^{\bPhi}_i(\mu)$ and $\bb^{\bPhi}(\mu)$ from the affine factors of $\bV^{\bTheta}_i(\mu)$ and $\bb^{\bTheta}(\mu)$. Note that random sketching can reduce the online complexity (and improve the stability) associated with residual-based error estimation. The online cost of computation of a solution, however, remains the same for both the classical and the sketched methods. Table\nobreakspace \ref {tab:onlineruntimes} reveals that for this benchmark the speedups in the online stage are achieved for $i \geq 50$. The computational cost of the error estimation using the classical approach grows quadratically with $i$, while using the randomized procedure, it grows only linearly. For $i=150$ we report a reduction of the runtime for error estimation by a factor $5$ and a reduction of the total runtime by a factor $2.6$.

\begin{table}[tbhp]
	\caption{The CPU times in seconds taken by each type of computations of  the classical and the efficient sketched provisional online solvers during the $i$th iteration of the greedy algorithms. }
	\label{tab:onlineruntimes}
	\centering
		\begin{tabular}{|c|r|r|r|r|r|r|} \hline
			 \multirow{2}{*}{Computations} & \multicolumn{3}{c|}{Classical} & \multicolumn{3}{c|}{Randomized}  \\  \cline{2-7}
			&$i=50$  & $i=100$ & $i=150$ & $i=50$ & $i=100$ & $i=150$  \\  [2pt]  \hline 
			 sketch & $-$ & $-$ & $-$ & $1.3$ & $1.5$ & $2$ \\ \cline{1-7}   
			 Galerkin& $2$ & $7$ & $13.5$ & $2.3$ & $7$ & $14$ \\ \cline{1-7}
			 error   & $2.8$ & $18$ & $45.2$ & $1.3$ & $3.1$ & $7$ \\   \cline{1-7}
			 total   & $4.8$ & $24.9$ & $58.7$ & $4.8$ & $11.6$ & $22.8$ \\   \hline 
		\end{tabular}
\end{table}   

The benefit of using random sketching methods for POD is validated in the context of distributed or limited-memory environments, where the snapshots are computed on distributed workstations or when the storage of snapshots requires too much RAM. For these scenarios the efficiency is characterized  by the amount of communication or storage needed for constructing a reduced model. Let us recall that the classical POD requires maintaining and operating with the full basis matrix $\bU_m$, while the sketched POD requires the {precomputation} of a $\bTheta$-sketch of $\bU_m$ and then constructs a reduced model from the sketch. In particular, for distributed computing a random sketch of each snapshot should be computed on a separate machine and then efficiently transfered to the master workstation for post-processing. For this experiment, Gaussian matrices of different sizes were tested for $\bOmega$. A seeded random number generator was used for maintaining $\bOmega$ with negligible computational costs. In~Table\nobreakspace \ref{tab:storage} we provide the amount of storage needed to maintain a sketch of a single snapshot, which also reflects the required communication for its transfer to the master workstation in the distributed computational environment. We observe that random sketching methods yielded computational costs reductions when $k \leq 17000$. {It follows that for $k=10000$ a $\bTheta$-sketch of a snapshot consumes $1.7$ times less memory than the full snapshot. Yet, for $m=\# \mathcal{P}_\mathrm{train} \leq 10000$ and $r \leq 150$, the sketched method of snapshots (see Definition~\ref{thm:sk_pod}) using $\bOmega$ of size $k=10000$ provides almost optimal approximation of the training set of snapshots with an error which is only at most $1.1$ times higher than the error associated with the classical POD approximation. A Gaussian matrix of size $k=10000$, for $r \leq 150$, also yields with high probability very accurate estimation (up to a factor of $1.1$) of the residual error and sufficiently accurate  estimation of the Galerkin projection (increasing the residual error by at most a factor of $1.66$).} For coercive and well-conditioned problems such as the thermal-block benchmark, it can be sufficient to use much smaller sketching matrices than in the present benchmark, say with $k=2500$ rows. Moreover, this value for $k$ should be pertinent also for ill-conditioned problems, including the considered acoustic cloak benchmark, when the minimal residual methods are used alternatively to the Galerkin methods~\cite{balabanov2018}. From~Table\nobreakspace \ref{tab:storage} it follows that a random sketch of dimension $k=2500$ is $6.8$ times cheaper to maintain than a full snapshot vector.  It has to be mentioned that when the sketch is computed from the affine expansion of $\bA(\mu)$ with $m_A$ terms (here $m_A=11$), its maintenance/transfer costs are proportional to $k m_A$ and are independent of the dimension of the initial system of equations. Consequently, for problems with larger $n/m_A$ a better cost reduction is expected.

\begin{table}[tbhp]
	\caption{The amount of data in megabytes required to maintain/transfer a single snapshot or its $\bTheta$-sketch for post-processing.}
	\label{tab:storage}
	\centering
	\begin{tabular}{|c|c|c|c|c|c|c|} \hline
		full snapshot& $k=2500$  & $k=5000$ & $k=10000$ & $k=15000$ &  $k=17000$ & $k=20000$ \\  [2pt]  \hline
		$1.64$ & $0.24$  & $0.48$ & $0.96$ & $1.44$ & $1.63$& $1.92$  \\ [2pt]  \hline  
	\end{tabular}
\end{table}

\section{Conclusions and future work} \label{Conclusions}  
In this paper we proposed a methodology for reducing the cost of classical projection-based MOR methods such as RB method and POD. The computational cost of constructing a reduced order model is essentially reduced to evaluating the samples (snapshots) of the solution on the training set, which in its turn can be efficiently performed with state-of-the-art routine on a powerful server or distributed machines. Our approach can be beneficial in any computational environment. It improves efficiency of classical MOR methods in terms of complexity (number of flops), memory consumption, scalability, communication cost between distributed machines, etc. Unlike classical methods, our method does not require maintaining and operating with high-dimensional vectors. Instead, the reduced order model is constructed from a random sketch (a set of random projections), with a negligible computational cost. A new framework was introduced in order to adapt {the} random sketching technique to the context of MOR. We interpret random sketching as a random estimation of inner products between high-dimensional vectors. The projections are obtained with random matrices (called oblivious subspace embeddings), which are efficient to store and to multiply by. We introduced oblivious subspace embeddings for a general inner product defined by a self-adjoint positive definite matrix.  Thereafter, we introduced randomized versions of Galerkin projection, residual based error estimation, and primal-dual correction. The conditions for preserving the quality of the output of the classical method were provided. In addition, we discussed computational aspects for an efficient evaluation of a random sketch in different computational environments, and introduced a new procedure for estimating the residual norm. This procedure is not only efficient but also is less sensitive to round-off errors than the classical approach. Finally, we proposed randomized versions of POD and greedy algorithm for RB. Again, in both algorithms, standard operations are performed only on the sketch but not on high-dimensional vectors. 

The methodology has been validated in a series of numerical experiments. We observed that indeed random sketching can provide a drastic reduction of the computational cost. The experiments revealed that the theoretical bounds for the sizes of random matrices are pessimistic. In practice, it can be pertinent to use much smaller matrices. In such a case it is important to provide a posteriori certification of the solution. In addition, it can be helpful to have an indicator of the accuracy of random sketching, which can be used for an adaptive selection of {the random matrices'} sizes. The aforementioned issues are addressed in~\cite{balabanov2018}. It was also observed that the performance of random sketching for estimating the Galerkin projection depends on the operator's properties (more precisely on the constant $a_r(\mu)$ defined in  Proposition\nobreakspace \ref {thm:skcea}). Consequently, the accuracy of the output can degrade considerably for problems with ill-conditioned operators. A remedy is to replace Galerkin projection by  another type of projection for the approximation of $\bu(\mu)$ (and $\bu^\mathrm{du}(\mu)$). The randomized minimal residual projection proposed in~\cite{balabanov2018} preserves the quality of the classical minimal residual projection regardless of the operator's properties. Another remedy would be to  improve the condition number of $\bA(\mu)$ with {an} affine parameter-dependent preconditioner. We also have seen that preserving a high precision for the sketched primal-dual correction~(\ref {eq:skcorrection}) can require large sketching matrices. A way to overcome this issue was proposed. It consists in obtaining an efficient approximation $\bw^\mathrm{du}_r(\mu)$ of the solution $\bu_r^\mathrm{du}(\mu)$ (or $\bu_r(\mu)$). Such $\bw^\mathrm{du}_r(\mu)$ can be also used for reducing the cost of extracting the quantity of interest from $\bu_r(\mu)$, i.e., computing $\bl_r(\mu)$, which in general can be expensive (but was assumed to have a negligible cost). In addition, this approach can be used for problems with nonlinear quantities of interest. An approximation $\bw^\mathrm{du}_r(\mu)$ can be taken as a projection of $\bu_r^\mathrm{du}(\mu)$ (or $\bu_r(\mu)$) on a subspace $W^\mathrm{du}_r$. In the experiments $W^\mathrm{du}_r$ was constructed from the first several basis vectors of the approximation space $U^\mathrm{du}_r$. A better subspace can be obtained by approximating the manifold $\{{\bu_r^\mathrm{du}(\mu)} : \mu \in \mathcal{P}_\mathrm{train} \}$ with a greedy algorithm or POD. Here, random sketching can be again employed for improving the efficiency. The strategies for obtaining both accurate and efficient $W^\mathrm{du}_r$ with random sketching are discussed in details in~\cite{balabanov2018}.

\newpage 
\section*{Appendix} \label{appendix}
{Here we list the proofs of propositions and theorems from the paper.}
 
\begin{proof}[Proof of Proposition\nobreakspace \ref{thm:cea} (modified Cea's lemma)]
	For all $\bx \in U_r$, it holds
	\begin{align*} 
	\alpha_r(\mu) \| \bu_r(\mu)- \bx \|_{U} & \leq \| \br(\bu_r(\mu);\mu)- \br(\bx;\mu)\|_{U_r'} \leq \| \br(\bu_r(\mu);\mu) \|_{U_r'}+ \| \br(\bx; \mu) \|_{U_r'}\\
	& =  \| \br(\bx; \mu) \|_{U_r'} \leq \beta_r(\mu) \| \bu(\mu)-\bx \|_{U},
	\end{align*}
	where the first and last inequalities directly follow from the definitions of $\alpha_r(\mu)$ and $\beta_r(\mu)$, respectively. Now, 
	\begin{equation*} 
	\| \bu(\mu)-\bu_r(\mu)\|_{U} \leq \| \bu(\mu)- \bx \|_{U}+ \| \bu_r(\mu)- \bx \|_{U} 
	\leq \| \bu(\mu)- \bx \|_{U}+ \frac{\beta_r(\mu)}{\alpha_r(\mu)} \| \bu(\mu)- \bx \|_{U},
	\end{equation*}
	which completes the proof. 
\end{proof}

\begin{proof}[Proof of Proposition\nobreakspace \ref{thm:clsstability}]
	For all $\ba \in \mathbb{K}^{r}$ and $\bx:= \bU_r\ba$, it holds
	\begin{equation*} 
	\begin{split} 
	\frac{\| \bA_r(\mu)\ba \|}{\| \ba \|}&= \underset{\bz \in \mathbb{K}^{r} \backslash \{ \bnull \}} \max \frac{|\langle \bz, \bA_r(\mu)\ba \rangle|}{\| \bz \| \| \ba \|} = \underset{ \bz \in \mathbb{K}^{r} \backslash \{ \bnull \}} \max \frac{|\bz^{\mathrm{H}}\bU_r^{\mathrm{H}}\bA(\mu)\bU_r\ba|}{\| \bz \| \| \ba \|} \\
	&= \underset{\by \in U_r \backslash \{ \bnull \}} \max \frac{|\by^{\mathrm{H}}\bA(\mu)\bx|}{\| \by \|_U \| \bx \|_U}
	= \frac{\| \bA(\mu)\bx \|_{U'_r}}{ \| \bx \|_U}.
	\end{split} 
	\end{equation*} 
	Then the proposition follows directly from definitions of $\alpha_r(\mu)$ and $\beta_r(\mu)$.  
\end{proof}

 \begin{proof}[Proof of Proposition\nobreakspace \ref{thm:error_correction}] We have 
	\begin{align*}
	|s(\mu)- s_r^{\mathrm{pd}}(\mu)| &= |s(\mu) -  s_r(\mu)+ \langle \bu_r^\mathrm{du}(\mu), \br(\bu_r(\mu); \mu) \rangle| \\
	&= |\langle  \bl(\mu), \bu(\mu)- \bu_r(\mu) \rangle+ \langle \bA (\mu)^{\mathrm{H}}  \bu_r^\mathrm{du}(\mu), \bu(\mu)- \bu_r(\mu) \rangle| \\
	&= |\langle  \br^{\mathrm{du}} (\bu_r^\mathrm{du}(\mu); \mu), \bu(\mu)- \bu_r(\mu) \rangle |\\
	&\leq \| \br^{\mathrm{du}} (\bu_r^\mathrm{du}(\mu); \mu) \|_{U'} \| \bu(\mu) - \bu_r(\mu) \|_U,
	\end{align*}
	and the result follows from definition~(\ref {eq:errorind}). 
\end{proof}

\begin{proof}[Proof of Proposition\nobreakspace \ref{thm:approx_pod}]
	To prove the first inequality we notice that $\bQ\bP_{U_r}\bU_m$ has rank at most $r$. Consequently, 
	\begin{equation*}
	\| \bQ\bU_m - {\bB^*_r}\|^2_{F}\leq \| \bQ\bU_m- \bQ\bP_{{U_r}}\bU_m \|^2_{F}= \sum^{m}_{i=1} \| \bu (\mu^i) - \bP_{{U_r}} \bu (\mu^i)\|^2_{U}.
	\end{equation*}	
	For the second inequality let us denote the $i$-th column vector of {$\bB_r$} by {$\bb_i$}. Since $\bQ\bR_U^{-1}\bQ^\mathrm{H}= \bQ\bQ^\dagger$, with $\bQ^\dagger$ the pseudo-inverse of $\bQ$, is the orthogonal projection onto $\mathrm{range}(\bQ)$, we have
	\begin{equation*}
	\begin{split} 
	\| \bQ\bU_m- {\bB_r}\|^2_{F} &\geq \|\bQ \bR_{U}^{-1}\bQ^{\mathrm{H}} (\bQ\bU_m- {\bB_r})\|^2_{F}= \sum^{m}_{i=1} \|\bu (\mu^i) - \bR_{U}^{-1}\bQ^{\mathrm{H}}{\bb_i} \|^2_{U} \\
	&\geq \sum^{m}_{i=1} \|\bu (\mu^i)- \bP_{{U_r}}\bu (\mu^i) \|^2_{U}.
	\end{split} 
	\end{equation*}	
\end{proof}

\begin{proof}[Proof of Proposition\nobreakspace \ref{thm:innerproduct}]
	It is clear that $\langle \cdot, \cdot \rangle^{\bTheta}_{X}$ and $\langle  \cdot , \cdot  \rangle^{\bTheta}_{X'}$ satisfy (conjugate) symmetry, linearity and positive semi-definiteness properties. The definitenesses of $\langle \cdot, \cdot \rangle^{\bTheta}_{X}$ and $\langle  \cdot , \cdot  \rangle^{\bTheta}_{X'}$ on $Y$ and $Y'$, respectively, follow directly from Definition~\ref{def:epsilon_embedding} and Corollary~\ref{thm:dual_embedding}. 
\end{proof}

 \begin{proof}[Proof of Proposition\nobreakspace \ref{thm:skseminorm_ineq}]
	Using Definition\nobreakspace \ref {def:epsilon_embedding}, we have  
	\begin{equation*}
	\begin{split}
	\| \by' \|^{\bTheta}_{Z'} &= \underset{\bx \in Z \backslash \{ \bnull \}} \max \frac{|\langle \bR_X^{-1}\by', \bx \rangle^{\bTheta}_{X}|}{\| \bx \|^{\bTheta}_{X}}\leq \underset{\bx \in Z \backslash \{ \bnull \}} \max \frac{|\langle \bR_X^{-1}\by', \bx \rangle_{X}|+ \varepsilon \| \by' \|_{X'}  \| \bx \|_{X}} {\| \bx \|^{\bTheta}_{X}} \\ 
	&\leq \underset{\bx \in Z \backslash \{ \bnull \}} \max \frac{|\langle \bR_X^{-1}\by', \bx \rangle_{X}|+ \varepsilon  \| \by' \|_{X'} \| \bx \|_{X}} {\sqrt{1-\varepsilon}\| \bx \|_{X}} \\
	&\leq \frac{1}{\sqrt{1-\varepsilon}} \left ( \underset{\bx \in Z \backslash \{ \bnull \}} \max \frac{|\langle \by', \bx \rangle|} {\| \bx  \|_{X}} + \varepsilon\| \by' \|_{X'} \right ),
	\end{split}
	\end{equation*}
	which yields the right inequality. To prove the left inequality we assume that $ \| \by' \|_{Z'}-  \varepsilon\| \by' \|_{X'}\geq 0$.  Otherwise the relation is obvious because $\| \by' \|^{\bTheta}_{Z'}\geq 0$. By Definition\nobreakspace \ref {def:epsilon_embedding},
	\begin{equation*}
	\begin{split} 
	\| \by' \|^{\bTheta}_{Z'} &=\underset{\bx \in Z \backslash \{ \bnull \}} \max \frac{|\langle \bR_X^{-1}\by', \bx \rangle^{\bTheta}_{X}|} {\| \bx  \|^{\bTheta}_{X}}\geq \underset{\bx \in Z  \backslash \{ \bnull \}} \max \frac{|\langle \bR_X^{-1}\by', \bx \rangle_{X}|- \varepsilon \| \by' \|_{X'} \| \bx \|_{X} } {\| \bx \|^{\bTheta}_{X}} \\ 
	&\geq \underset{\bx \in Z \backslash \{ \bnull \}} \max \frac{|\langle \bR_X^{-1}\by', \bx \rangle_{X}|- \varepsilon \| \by' \|_{X'} \| \bx  \|_{X}} {\sqrt{1+\varepsilon}\| \bx  \|_{X}} \\
	&\geq \frac{1}{\sqrt{1+\varepsilon}} \left ( \underset{\bx \in Z \backslash \{ \bnull \}} \max \frac{|\langle \by', \bx \rangle|} {\| \bx \|_{X}}- \varepsilon \| \by' \|_{X'} \right ),
	\end{split}
	\end{equation*}
	which completes the proof. 
\end{proof}

 \begin{proof}[Proof of Proposition\nobreakspace \ref{thm:Rademacher}]
	{Let us start with the case $\mathbb{K} = \mathbb{R}$. For the proof we shall follow standard steps (see, e.g., \cite[Section~2.1]{woodruff2014sketching}). Given a $d$-dimensional subspace $V \subseteq \mathbb{R}^{n}$, let $\mathcal{S} = \{ \bx \in V : \ \|  \bx  \| = 1 \}$ be the {unit sphere} of $V$.  According to \cite[Lemma~2.4]{bourgain1989approximation}, for any $\gamma >0$ there exists a $\gamma$-net $\mathcal{N}$ of $\mathcal{S}$\footnote{{We have $\forall \bx \in \mathcal{S}, \exists \by \in \mathcal{N}$ such that $\|\bx - \by \| \leq \gamma$.}} satisfying $\# \mathcal{N}  \leq (1+2/\gamma)^d$. For $\eta$ such that $0<\eta<1/2$, let $\bTheta \in \mathbb{R}^{k \times n}$ be a rescaled Gaussian or Rademacher matrix with $k\geq {6} \eta^{-2} ({2} d \log (1+2/\gamma)+ \log ({1}/\delta))$. {By \cite[Lemmas~4.1 and 5.1]{achlioptas2003database} and an union bound argument we obtain for a fixed $\bx \in {V}$ $$\mathbb{P}(|\|\bx\|^2 - \|\bTheta \bx\|^2| \leq \eta \|\bx\|^2  ) \geq 1 - {2\exp(-k\eta^{2}/6)}.$$} Consequently, using a union bound for the probability of success, we have that
	\begin{equation*}
	\left\{\ | \| \bx + \by \|^2 - \|\bTheta(\bx + \by) \|^2 | \leq \eta \|\bx + \by \|^2, \quad \forall \bx, \by \in \mathcal{N}\right\},
	\end{equation*}
	holds with probability at least $1-\delta$. Then we deduce that
	\begin{equation}  \label{eq:JLT}
	\left\{\ | \langle \bx , \by  \rangle - \langle \bTheta \bx , \bTheta \by  \rangle| \leq \eta, \quad \forall \bx, \by \in \mathcal{N}\right\}
	\end{equation}
	holds with probability at least $1-\delta$. Now, let 
	$\bn$ be {some vector in $\mathcal{S}$}. Assuming $\gamma<1$, it can be proven by induction that 
	$  \bn = \sum_{i\ge 0} \alpha_i \bn_i,$ where 
	$\bn_i \in \mathcal{N}$ and  $0\le \alpha_i \leq \gamma^i$\footnote{{Indeed, $\exists \bn_0 \in \mathcal{N}$ such that  $\|\bn - \bn_0 \|:=\alpha_1 \leq \gamma$. Let $\alpha_0 = 1$. Then assuming that $\|\bn - \sum^m_{i=0}  {\alpha_i}\bn_i \|:=\alpha_{m+1} \leq \gamma^{m+1}$,  $\exists \bn_{m+1} \in \mathcal{N}$ such that $\| \frac{1}{\alpha_{m+1}} (\bn - \sum^m_{i=0} {\alpha_i} \bn_i) - \bn_{m+1} \| \leq  \gamma$ $\implies$ $\|\bn - \sum^{m+1}_{i=0} {\alpha_i} \bn_i \| \leq \alpha_{m+1} \gamma \leq \gamma^{m+2} $. }}. If~(\ref {eq:JLT}) is satisfied, then
	\begin{align*}
	\| \bTheta \bn \|^2 &  = \sum_{i,j \geq 0} \langle \bTheta \bn_i, \bTheta\bn_j  \rangle \alpha_i\alpha_j \\
	&\leq \sum_{i,j \geq 0} ( \langle \bn_i, \bn_j \rangle \alpha_i \alpha_j + \eta \alpha_i \alpha_j)  
	= 1 + \eta  ( \sum_{i \geq 0} \alpha_i  )^2 \le 1 + \frac{\eta}{(1-\gamma)^2},
	\end{align*}
	and similarly $
	\| \bTheta \bn \|^2 \geq 1-\frac{\eta}{(1-\gamma)^2}.$ Therefore, if~(\ref {eq:JLT}) is satisfied, we have
	\begin{equation} \label{eq:nbound}
	| 1- \| \bTheta \bn \|^2| \leq \eta / (1-\gamma)^2. 
	\end{equation}
	For a given $\varepsilon \leq 0.5/  (1-\gamma)^2$, let $\eta = (1-\gamma)^2 \varepsilon$. Since (\ref {eq:nbound}) holds for an arbitrary vector {$\bn \in \mathcal{S}$}, using the parallelogram identity, we easily obtain that 
	\begin{equation} 
	\ \left | \langle \bx, \by \rangle - \langle \bTheta \bx, \bTheta \by \rangle \right |\leq \varepsilon \| \bx \| \| \by \| 
	\end{equation}
	holds for all $\bx, \by \in V$ if~(\ref {eq:JLT}) is satisfied.
	We conclude that {if $k\geq {6} \varepsilon^{-2} (1-\gamma)^{-4} ({2} d \log (1+2/\gamma)+ \log ({1}/\delta))$ then}	
	 $\bTheta$ is a $\ell_2 \to \ell_2$ $\varepsilon$-subspace embedding for $V$ with probability at least $1-\delta$. The lower bound for the number of rows of $\bTheta$ is obtained by taking  $\gamma= \arg\min_{x \in (0,1)}(\log (1+2/x)/(1-x)^4)\approx 0.0656$. }
	
	The statement of the proposition for the case $\mathbb{K} = \mathbb{C}$ can be deduced from the fact that if $\bTheta$ is $( \varepsilon, \delta, 2d)$ oblivious $\ell_2 \to \ell_2$ subspace embedding for $\mathbb{K} = \mathbb{R}$, then it is $(\varepsilon, \delta, d)$ oblivious $\ell_2 \to \ell_2$ subspace embedding for $\mathbb{K} = \mathbb{C}$. {A detailed proof of this fact is provided in the supplementary material.} To show this we {first} note that the real part and the imaginary part of any vector from a $d$-dimensional subspace $V^* \subseteq \mathbb{C}^n$ belong to a certain subspace  $W \subseteq \mathbb{R}^n$ with $\mathrm{dim}(W) \leq 2d$. Further, one can show that if $\bTheta$ is $\ell_2 \to \ell_2$ $\varepsilon$-subspace embedding for $W$, then it is $\ell_2 \to \ell_2$ $\varepsilon$-subspace embedding for $V^*$.   
\end{proof}

\begin{proof}[Proof of Proposition\nobreakspace \ref{thm:P-SRHT}]
	Let $\bTheta \in \mathbb{R}^{k\times n}$ be a P-SRHT matrix, let $V$ be an arbitrary $d$-dimensional subspace of $\mathbb{K}^{n}$, and let $\bV \in \mathbb{K}^{n\times d}$ be a matrix whose columns form an orthonormal basis of $V$. Recall, $\bTheta$ is equal to the first $n$ columns of matrix $\bTheta^* = k^{-1/2} (\bR\bH_s\bD) \in \mathbb{R}^{k \times s}$.  {Next we shall use the fact that for any orthonormal matrix $\bV^* \in \mathbb{K}^{s\times d}$, all singular values of a matrix $\bTheta^*\bV^*$ belong to {the} interval $[\sqrt{1-\varepsilon}, \sqrt{1+\varepsilon}]$ with probability at least $1-\delta$. This result is basically a restatement of~\cite[Lemma 4.1]{boutsidis2013improved} and~\cite[Theorem 3.1]{tropp2011improved} including the complex case and with improved constants. It can be shown to hold by mimicking the proof in~\cite{tropp2011improved} with a few additional algebraic operations. For a detailed proof of the statement, see {the} supplementary material.}
	
	{By taking $\bV^*$ with the first $n\times d$ block equal to $\bV$ and zeros elsewhere, and using the fact that $\bTheta\bV$ and $\bTheta^*\bV^*$ have the same singular values, we obtain that
	\begin{equation} \label{eq:pf-SRHT}
	| \| \bV{\bz} \|^2 -  \| \bTheta\bV{\bz} \|^2 |=  |{\bz}^{\mathrm{H}} (\bI- \bV^{\mathrm{H}}\bTheta^{\mathrm{H}}\bTheta\bV) {\bz} | \leq \varepsilon \| {\bz} \|^2= {\varepsilon} \| \bV{\bz} \|^2, \quad \forall {\bz} \in \mathbb{K}^{d}
	\end{equation}
	holds with probability at least $1-\delta$. {Using {the} parallelogram identity, it can be easily proven that relation~(\ref{eq:pf-SRHT}) implies 
	$$
	 \ \left | \langle \bx, \by \rangle - \langle \bTheta\bx, \bTheta\by \rangle \right |\leq \varepsilon \| \bx \| \| \by \|, \quad  \forall \bx, \by \in V.
	$$}
	% can be brought to a form as in~(\ref {eq:epsilon_embedding}) using the parallelogram identity. 
	We conclude that $\bTheta$ is a $(\varepsilon, \delta, d)$ oblivious $\ell_2 \to \ell_2$ subspace embedding. }
\end{proof}

\begin{proof}[Proof of Proposition\nobreakspace \ref{thm:buildepsilon_embedding}]
	Let $V$ be any $d$-dimensional subspace of $X$ and let $V^*:= \{ \bQ\bx : \bx \in V \}$. Since the following relations hold $\langle \cdot, \cdot \rangle_U = \langle \bQ \cdot, \bQ \cdot \rangle$ and $\langle \cdot, \cdot \rangle^\bTheta_U = \langle \bQ \cdot, \bQ \cdot \rangle_2^\bOmega$, we have that {the} sketching matrix $\bTheta$ is an $\varepsilon$-embedding for $V$ if and only if $\bOmega$ is an $\varepsilon$-embedding for $V^*$. It follows from the definition of $\bOmega$ that this matrix is an $\varepsilon$-embedding for $V^*$ with probability at least $1-\delta$, which completes the proof. 
\end{proof}

\begin{proof}[Proof of Proposition\nobreakspace \ref{thm:SKquasi-opt} (sketched Cea's lemma)]
	The proof exactly follows the one of Proposition\nobreakspace \ref {thm:cea} with $\| \cdot \|_{U_r'}$ replaced by $\| \cdot \|^{\bTheta}_{U_r'}$. 
\end{proof} 

\begin{proof}[Proof of Proposition\nobreakspace \ref{thm:skcea}]
	According to Proposition\nobreakspace \ref {thm:skseminorm_ineq}, and by definition of $a_r(\mu)$, we have
	\begin{equation*}
	\begin{split} 
	\alpha^{\bTheta}_r(\mu) &=\underset{\bx \in U_r \backslash \{ \bnull \}} \min \frac{\| \bA(\mu)\bx \|^{\bTheta}_{U_r'}}{\| \bx \|_U}\geq \frac{1}{\sqrt{1+\varepsilon}}\underset{\bx \in U_r \backslash \{ \bnull \}} \min \frac{(\| \bA(\mu)\bx \|_{U_r'}- \varepsilon\| \bA(\mu)\bx\|_{U'})}{\| \bx \|_U} \\
	&\geq \frac{1}{\sqrt{1+\varepsilon}}(1-\varepsilon a_r(\mu)) \underset{\bx \in U_r \backslash \{ \bnull \}} \min \frac{\| \bA(\mu)\bx \|_{U_r'}}{\| \bx \|_U}.
	\end{split}
	\end{equation*}
	Similarly,
	\begin{equation*}
	\begin{split} 
	\beta^{\bTheta}_r(\mu) &= \underset{ \bx \in \left ( \mathrm{span} \{ \bu(\mu) \}+ U_r \right ) \backslash \{ \bnull \}} \max \frac{ \| \bA(\mu)\bx \|^{\bTheta}_{U_r'}}{\| \bx \|_U} \\
	&\leq \frac{1}{\sqrt{1-\varepsilon}} \underset{\bx \in \left ( \mathrm{span} \{ \bu(\mu) \} + U_r \right ) \backslash \{ \bnull \}} \max \frac{\| \bA(\mu)\bx \|_{U_r'}+ \varepsilon\| \bA(\mu)\bx \|_{U'}}{\| \bx \|_U} \\
	&\leq \frac{1}{\sqrt{1-\varepsilon}} \left ( \underset{\bx \in \left (\mathrm{span} \{ \bu(\mu) \}+ U_r \right ) \backslash \{ \bnull \}} \max \frac{\| \bA(\mu)\bx \|_{U_r'}}{\| \bx \|_U}+ \varepsilon \underset{\bx \in U \backslash \{ \bnull \}} \max \frac{\| \bA(\mu)\bx \|_{U'}}{\| \bx \|_U} \right ).
	\end{split}
	\end{equation*} 
\end{proof}

\begin{proof} [Proof of Proposition\nobreakspace \ref{thm:skalgstability}]
	Let $\ba \in \mathbb{K}^{r}$ and $\bx:= \bU_r\ba$.	Then
	\begin{equation} \label{eq:skcondproof1}
	\begin{split} 
	\frac{\| \bA_r(\mu)\ba \|}{\| \ba \|}&= \underset{\bz \in \mathbb{K}^{r} \backslash \{ \bnull \}} \max \frac{|\langle \bz, \bA_r(\mu)\ba \rangle|}{\| \bz \| \| \ba \|} = \underset{ \bz \in \mathbb{K}^{r} \backslash \{ \bnull \}} \max \frac{|\bz^{\mathrm{H}}\bU_r^{\mathrm{H}}\bTheta^{\mathrm{H}}\bTheta\bR_U^{-1}\bA(\mu)\bU_r\ba|}{\| \bz \| \| \ba \|} \\
	&= \underset{\by \in U_r \backslash \{ \bnull \}} \max \frac{|\by^{\mathrm{H}}\bTheta^{\mathrm{H}}\bTheta\bR_U^{-1}\bA(\mu)\bx|}{\| \by \|^{\bTheta}_U \| \bx \|^{\bTheta}_U}= \underset{\by \in U_r \backslash \{ \bnull \}} \max \frac{|\langle \by, \bR_U^{-1}\bA(\mu)\bx \rangle^{\bTheta}_{U}|}{\| \by \|^{\bTheta}_U \| \bx \|^{\bTheta}_U}	\\
	&= \frac{\| \bA(\mu)\bx \|^{\bTheta}_{U'_r}}{ \| \bx \|^{\bTheta}_U}.
	\end{split} 
	\end{equation}
	By definition, 
	\begin{equation} \label{eq:skcondproof2}
	\sqrt{1-\varepsilon} \| \bx \|_U \leq \| \bx \|^{\bTheta}_U \leq  \sqrt{1+\varepsilon} \| \bx \|_U.
	\end{equation}
	Combining~(\ref {eq:skcondproof1})~and~(\ref {eq:skcondproof2}) we conclude that  
	\begin{equation*}
	\frac{1}{\sqrt{1+\varepsilon}}\frac{\| \bA(\mu)\bx \|^{\bTheta}_{U'_r}}{\| \bx \|_U} \leq \frac{\|\bA_r(\mu)\ba \|}{\| \ba \|}\leq \frac{1}{\sqrt{1-\varepsilon}}\frac{\| \bA(\mu)\bx \|^{\bTheta}_{U'_r}}{\| \bx \|_U}.
	\end{equation*}		
	The statement of the proposition follows immediately from definitions of $\alpha^{\bTheta}_r(\mu)$ and $\beta^{\bTheta}_r(\mu)$. 
\end{proof}

\begin{proof} [Proof of Proposition\nobreakspace \ref{thm:skerrorind}]
 The proposition directly follows from relations~(\ref{eq:errorind}), (\ref{eq:thetainnerdef}), (\ref{eq:epsilon_embedding}) and (\ref{eq:skerrorind}). 
\end{proof}

\begin{proof} [Proof of Proposition\nobreakspace \ref{thm:skerror_correction}]
	We have
	\begin{equation} \label{eq:skerror_correction1}
	\begin{split} 
	|s^{\mathrm{pd}}(\mu)- s_r^{\mathrm{spd}}(\mu)|&= | \langle \bu_r^\mathrm{du}(\mu), \bR_U^{-1}\br(\bu_r(\mu);\mu) \rangle_{U}-\langle \bu_r^\mathrm{du}(\mu), \bR_U^{-1}\br(\bu_r(\mu);\mu) \rangle^{\bTheta}_{U}| \\
	&\leq \varepsilon \| \br(\bu_r(\mu); \mu) ||_{U'} \| \bu_r^\mathrm{du}(\mu) \|_{U} \\
	&\leq \varepsilon \| \br(\bu_r(\mu);\mu) ||_{U'} \frac{\| \bA(\mu)^{\mathrm{H}} \bu_r^\mathrm{du}(\mu) \|_{U'}}{\eta(\mu)}  \\
	&\leq \varepsilon \| \br(\bu_r(\mu);\mu) ||_{U'} \frac{\| \br^{\mathrm{du}}(\bu_r^\mathrm{du}(\mu);\mu)\|_{U'}+ \| \bl(\mu) \|_{U'}}{\eta (\mu)},
	\end{split} 
	\end{equation}	
	and  (\ref {eq:skerror_correction}) follows by combining~(\ref {eq:skerror_correction1}) with~(\ref {eq:error_correction}). 
\end{proof}

\begin{proof} [Proof of Proposition \nobreakspace \ref{thm:sk_greedy_finite}]
	{In total, there are at most $\binom{m}{r}$ $r$-dimensional subspaces that could be spanned from $m$ snapshots. Therefore, by using the definition of $\bTheta$, the fact that $\dim(Y_r(\mu))\leq 2 r+1$ and a union bound for the probability of success, we deduce that $\bTheta$ is a $U \to \ell_2$ $\varepsilon$-subspace embedding for $Y_r(\mu)$, for fixed $\mu \in \mathcal{P}_{\mathrm{train}}$, with probability at least $1- m^{-1}\delta$. The proposition then follows from another union bound.} 
\end{proof}

\begin{proof} [Proof of Proposition \nobreakspace \ref{thm:sk_poderror}]
We have,
$$\Delta^\mathrm{POD}(V) = \frac{1}{m} \|  \bU_m^\bTheta -  \bTheta \bP^\bTheta_{V} \bU_m\|_F. $$ 
Moreover, {the} matrix $\bTheta \bP^\bTheta_{U_r} \bU_m$ is the rank-$r$ truncated SVD approximation of $\bU_m^\bTheta$. The statements of the proposition can be then derived from the standard properties of SVD. 
\end{proof}

\begin{proof} [Proof of Theorem \nobreakspace \ref{thm:sk_podopt}]		
	Clearly, if $\bTheta$ is a $U \to \ell_2$ $\varepsilon$-subspace embedding for $Y$, then $\mathrm{rank}(\bU^\bTheta_m)\geq r$. Therefore $U_r$ is well-defined. Let $\{( \lambda_i, \bt_i) \}_{i=1}^{l}$ and $\bT_r$ be given by~Definition\nobreakspace\ref{thm:sk_pod}.
	In general, $\bP^\bTheta_{U_r}$ defined by~(\ref{eq:P*}) may not be unique. Let us further assume that $\bP^\bTheta_{U_r}$ is provided for $\bx \in U_m$  by
	$
	\bP^\bTheta_{U_r}\bx:= \bU_r\bU_r^\mathrm{H}\bTheta^\mathrm{H}\bTheta\bx, 
	$
	where $\bU_r= \bU_m[\frac{1}{\sqrt{\lambda_1}}\bt_1, ..., \frac{1}{\sqrt{\lambda_r}}\bt_r]$.
	Observe that 
	$
	\bP^\bTheta_{U_r} \bU_m = \bU_m \bT_r \bT_r^\mathrm{H}.
	$ 
	For the first part of the theorem, we establish the following inequalities. {Let $\bQ \in \mathbb{K}^{s\times n}$ be such that $\bQ^\mathrm{H} \bQ = \bR_U$, then}
	\begin{equation*}
	\begin{split}
	&\frac{1}{m} \sum^{m}_{i=1} \| (\bI -\bP_{Y})(\bu(\mu^i)- \bP^\bTheta_{U_r}\bu(\mu^i)) \|_U^2= \frac{1}{m} \| \bQ(\bI -\bP_{Y})\bU_m (\bI - \bT_r\bT_r^\mathrm{H}) \|_F^2 \\
	&\leq \frac{1}{m} \| \bQ(\bI- \bP_{Y})\bU_m\|_F^2 \| \bI- \bT_r\bT_r^\mathrm{H} \|^2= \Delta_Y \| \bI- \bT_r\bT_r^\mathrm{H} \|^2\leq \Delta_Y, 
	\end{split}
	\end{equation*}
	and
	\begin{equation*}
	\begin{split}
	&\frac{1}{m} \sum^{m}_{i=1} \left ( \| (\bI- \bP_{Y})(\bu(\mu^i)- \bP^\bTheta_{U_r}\bu(\mu^i)) \|^\bTheta_{U} \right )^2= \frac{1}{m} \| \bTheta(\bI- \bP_{Y})\bU_m(\bI- \bT_r\bT_r^\mathrm{H}) \|_F^2 \\
	&\leq \frac{1}{m} \| \bTheta(\bI- \bP_{Y})\bU_m \|_F^2 \| \bI- \bT_r\bT_r^\mathrm{H} \|^2 \leq(1+\varepsilon) \Delta_Y \| \bI- \bT_r\bT_r^\mathrm{H} \|^2\leq (1+\varepsilon)\Delta_Y. 
	\end{split}
	\end{equation*}
	Now, we have	
	\begin{equation*}
	\begin{split}
	&\frac{1}{m} \sum^{m}_{i=1} \| \bu(\mu^i)- \bP_{U_r}\bu(\mu^i) \|_{U}^2\leq \frac{1}{m} \sum^{m}_{i=1} \| \bu(\mu^i)- \bP^\bTheta_{U_r}\bu(\mu^i) \|_{U}^2 \\
	&= \frac{1}{m} \sum^{m}_{i=1} \left ( \| \bP_{Y}(\bu(\mu^i)- \bP^\bTheta_{U_r}\bu(\mu^i)) \|_{U}^2+ \| (\bI- \bP_{Y})(\bu(\mu^i)- \bP^\bTheta_{U_r}\bu(\mu^i)) \|_{U}^2 \right ) \\
	& \leq \frac{1}{m} \sum^{m}_{i=1} \| \bP_{Y}(\bu(\mu^i)- \bP^\bTheta_{U_r}\bu(\mu^i)) \|_{U}^2+ \Delta_Y
	\leq \frac{1}{m} \frac{1}{1-\varepsilon} \sum^{m}_{i=1} \left (\| \bP_{Y}(\bu(\mu^i)- \bP^\bTheta_{U_r}\bu(\mu^i)) \|^\bTheta_{U} \right )^2+ \Delta_Y \\
	&\leq \frac{1}{1-\varepsilon} \frac{1}{m} \sum^{m}_{i=1} 2 \left ( \left ( \| \bu(\mu^i)- \bP^\bTheta_{U_r}\bu(\mu^i) \|^\bTheta_{U} \right)^2 + \left ( \| (\bI- \bP_{Y})(\bu(\mu^i)- \bP^\bTheta_{U_r}\bu(\mu^i)) \|^\bTheta_{U} \right)^2 \right )+ \Delta_Y \\
	&\leq \frac{1}{1-\varepsilon} \frac{1}{m} \sum^{m}_{i=1} 2 \left ( \| \bu(\mu^i)- \bP_{U^*_r}\bu(\mu^i) \|^\bTheta_{U} \right)^2+ (\frac{2(1+\varepsilon)}{1-\varepsilon}+1)\Delta_Y \\
	&\leq \frac{2(1+\varepsilon)}{1-\varepsilon} \frac{1}{m} \sum^{m}_{i=1} \| \bu(\mu^i)- \bP_{U^*_r}\bu(\mu^i) \|_{U}^2+ (\frac{2(1+\varepsilon)}{1-\varepsilon}+1)\Delta_Y,
	\end{split}
	\end{equation*}
	which is equivalent to~(\ref{eq:sk_podoptY}). 
	
	The second part of the theorem can be proved as follows. Assume that $\bTheta$ is $U \to \ell_2$ $\varepsilon$-subspace embedding for $U_{m}$, then
	\begin{equation*}
	\begin{split}
	&\frac{1}{m} \sum^{m}_{i=1} \| \bu(\mu^i)- \bP_{U_r}\bu(\mu^i) \|_{U}^2 \leq \frac{1}{m} \sum^{m}_{i=1} \| \bu(\mu^i)- \bP^\bTheta_{U_r}\bu(\mu^i) \|_{U}^2\leq \frac{1}{m} \frac{1}{1-\varepsilon}\sum^{m}_{i=1} \left (\| \bu(\mu^i)- \bP^\bTheta_{U_r}\bu(\mu^i) \|^{\bTheta}_{U} \right )^2 \\
	&\leq \frac{1}{m} \frac{1}{1-\varepsilon} \sum^{m}_{i=1} \left (\| \bu(\mu^i)- \bP_{U^*_r}\bu(\mu^i) \|^{\bTheta}_{U} \right )^2\leq \frac{1}{m} \frac{1+\varepsilon}{1-\varepsilon} \sum^{m}_{i=1} \| \bu(\mu^i)- \bP_{U^*_r}\bu(\mu^i) \|_{U}^2,  
	\end{split}
	\end{equation*}
	which completes the proof.
\end{proof}

\end{document}